\documentclass[12pt]{amsart} 
\usepackage{amssymb}
\usepackage{amsmath}
\usepackage{latexsym}
\usepackage{amscd}
\usepackage{amsfonts}
\usepackage{pb-diagram}
\usepackage{xypic}
\usepackage[mathcal,mathscr]{eucal}
\usepackage{amsthm}
\usepackage{epsfig}
\usepackage{url}
\usepackage{epsfig}
\usepackage{bbm}
\usepackage{mathrsfs}

\usepackage{color}
\usepackage{comment}

\usepackage[pdftex,colorlinks]{hyperref}

\usepackage{pifont}
\usepackage{amsbsy}
\usepackage{epstopdf}
\usepackage{times}
\usepackage{marginnote}
\usepackage{epic, eepic}

\newtheorem{theorem}{\textbf{Theorem}}[section]
\newtheorem{prop}[theorem]{\textbf{Proposition}}
\newtheorem{lemma}[theorem]{Lemma}
\newtheorem{corollary}[theorem]{Corollary}
\newtheorem{definition}[theorem]{\textbf{Definition}}
\theoremstyle{remark}
\newtheorem{remark}[theorem]{{\textbf{Remark}}}
\theoremstyle{example}
\newtheorem{example}[theorem]{\textbf{Example}}

\numberwithin{equation}{section}

\newcommand{\R}{\mathbb{R}}  

\DeclareMathOperator{\vol}{vol}
\DeclareMathOperator{\supp}{supp}

\DeclareMathOperator{\rank}{rank}

\DeclareMathOperator{\End}{End}
\DeclareMathOperator{\Ker}{Ker}

\DeclareMathOperator{\BS}{BS}

\DeclareMathOperator{\diver}{div}

\DeclareMathOperator{\PD}{PD}

\DeclareMathOperator{\Bl}{Bl}
\DeclareMathOperator{\lk}{lk}

\DeclareMathOperator{\cpt}{cpt}

\DeclareMathOperator{\loc}{loc}

\DeclareMathOperator{\dvol}{dvol}

\DeclareMathOperator{\ori}{or}

\DeclareMathOperator{\Aut}{Aut}

\DeclareMathOperator{\id}{id}

\DeclareMathOperator{\codim}{codim}

\DeclareMathOperator{\dist}{dist}
\DeclareMathOperator{\pt}{pt}

\DeclareMathOperator{\Imag}{Im}

\newcommand\bG{{\mathbb G}}
\newcommand\bH{{\mathbb H}}

\newcommand\bK{{\mathbb K}}

\newcommand{\bN}{{{\mathbb N}}}

\newcommand\bR{{\mathbb R}}

\newcommand\bZ{{\mathbb Z}}

\newcommand{\eR}{\EuScript{R}}

\newcommand{\ra}{\rightarrow}

\newcommand{\tcr}{\textcolor{red}}

\newcommand{\disp}{\displaystyle}
\newcommand{\eps}{\varepsilon}
\newcommand{\metric}{\langle \, , \, \rangle}

\begin{document}

\title{On linking numbers and Biot-Savart kernels}
\author{Daniel Cibotaru}
\address{Departamento de Matem\' atica, Universidade Federal do Cear\'a, Fortaleza, Brasil}
\email{daniel@mat.ufc.br}

\author{
 Luciano Mari}
\address{Dipartimento di Matematica, Universit\`a degli Studi di Milano, Milano, Italia}
\email{luciano.mari@unimi.it}

\begin{abstract} On a compact manifold,  the solutions of the initial value problem for the heat equation with currential initial conditions are smooth families of forms for $t>0$. If the initial condition is an exact submanifold $L$ then the integral in $t$ of this family gives a smooth form $\Omega$ on the complement of $L$ such that $\omega:=d^*\Omega$ is a  solution for the exterior derivative equation $d\omega=L$.  We introduce, for small $t$, an asymptotic approximation of these solutions  in order to show that $d^*\Omega$ is extendable to the oriented  blow-up of $L$ in codimension $1$ and $3$ and also $2$ when $L$ is minimal.   When $L$ is the diagonal in $M\times M$ we adapt these ideas to obtain a  {differential linking form}  for \emph{any} compact, ambient Riemannian manifold $M$ of dimension $3$. This coincides up to sign with the kernel of the Biot-Savart operator $d^*G$ and recovers the well-known Gauss formula for linking numbers in $\bR^3$.
\end{abstract}

\maketitle
\tableofcontents

\section{Introduction}

The linking number of two oriented knots $K_1$ and $K_2$ in $\bR^3$ can be computed in differential geometry by Gauss formula. This is a  consequence of the fact that, up to a sign,  this number coincides with the degree of the Gauss mapping 
\[K_1\times K_2\ra \mathbb{S}^2,\qquad (x,y)\ra \frac{x-y}{|x-y|}\]
where $\mathbb{S}^2$ is the unit sphere.  This extends easily to the higher dimensional  Euclidean space $\bR^n$. But it is not clear at all how one can produce such formulas in general geometric ambients, since with a few notable exceptions  there is no  analogue of the Gauss mapping.  Interest in the problem of producing integral formulas that compute the linking number in other Riemannian ambients, with an emphasis on "universal" such formulas has been quite rich, not least because of the connection  with classical electromagnetism, in particular with the Biot-Savart and Ampere laws. We refer the reader to \cite{dTG, dTG2, Ku, Par, Ri, SVV, Vo}. By universal, we mean here formulas which use for the computation of the linking number the same "universal" form $\Omega$ defined on some "universal" target space $N$ such that the linking number is computed for any two knots $K_1$ and $K_2$ by the integral of the pull-back of $\Omega$ through a "natural" map $K_1\times K_2\ra N$. They are also often called in the literature \emph{linking forms} \cite{Vo} not to be confused with the topological notion of linking form, which is a related but different object.    

The formulas that these linking forms give are to be compared with other types of integral formulas\footnote{some might call them asymmetric linking formulas} where one  gives a recipe for constructing a differential  form induced by one of the knots, whose integral over the other knot computes again the linking number. In this article, we describe a general framework that allows one to produce both kinds of formulas. 

Inspired mainly by the work of Bechtluft-Samiou \cite{BSS1} and deTurck-Gluck \cite{dTG,dTG2},  we propose a new approach based fundamentally on the properties of particular solutions to the PDE for the exterior derivative:
\begin{equation}\label{eqi0} d\omega=L\end{equation}
where $L$ is an oriented, \emph{exact} submanifold of a Riemannian manifold $M^n$ of dimension $k$. Exact is a synonym here with submanifold that bounds another submanifold. We are of course aiming for degree $n-k-1$ forms $\omega$ which are smooth outside $L$ like  the classical fundamental solution $\omega_3$ of the equation $d\omega_3=\delta_0$ (here $\delta_0$  is the Dirac $0$-dimensional current) in $\bR^3$:
\begin{equation}\label{eqi1}\omega_3:=-\frac{1}{\sigma_2}\frac{1}{{(x^2+y^2+z^2)}^{3/2}}\left(xdy\wedge dz-ydx\wedge dz+zdx\wedge dy\right).\end{equation}
  In this example, the form $\omega_3$ is smooth on  $\bR^3\setminus\{0\}$ and has locally integrable coefficients in $\bR^3$ and thus defines a current of dimension $1$ in $\bR^3$. "Singular" currents such as the one induced by $\omega_3$ might be of limited use since,   currents do not naturally  pull-back.  We are aiming to at least be able to pull-back $\omega$ via mappings that are transverse to the set $L$ of "singularities" of $\omega$. The natural idea is to search for solutions $\omega$ of (\ref{eqi0}) that can be lifted to the total space of the oriented blow-up of $L$ as a continuous form. For example, it is easily checked that  the form $\omega_3$ of (\ref{eqi1}) is  extendible under the blow-up map of $\{0\}$ in $\bR^3$:
\[[0,\infty)\times S^2\ra \bR^3,\qquad (t,v)\ra tv. \]
to a continuous (in fact, smooth) form on $[0,\infty)\times S^2$. It turns out that this extendibility property is all what one needs in order to show that the linking number of $L$ with another exact submanifold $K^{n-k-1}\subset M$ (with $L\cap K=\emptyset$)  is given by the integral of $\omega$ over $K$. We call them \emph{extendible solutions} of (\ref{eqi0}). 

The main idea for generalizing Gauss formula is the following.   For closed but non-exact, oriented $L$ in order to still guarantee the  existence of solutions for  (\ref{eqi0})  one needs to alter the right hand side as follows
\begin{equation}\label{eqi2}
	d\omega=L-H_L
\end{equation}
where $H_L$ is the harmonic representative of the Poincar\'e dual of $L$. Now, if one can produce extendible solutions of the altered (\ref{eqi2}) when $L$ is the diagonal in $M\times M$ then one can produce universal linking number formulas by integrating such a form in $M\times M$ over the product of knots. 

 We then restrict to compact oriented ambient  manifolds $M$ because standard Hodge theory provides us with a candidate for the solution of (\ref{eqi0}). Indeed by extending, via dualization, the Green operator $G_k$ for the Laplacian on $k$-forms to act on currents, one can show that for an exact submanifold $L$, the form $d^*G_k(L)$ gives a solution to (\ref{eqi0}). Work of deRham \cite{dR} already implies that this is a smooth form away from $L$.  The difficult part is understanding the behaviour of $d^*G_k(L)$ close to $L$. For this purpose we introduce initial value problems (IVP) for the heat equation with submanifolds as initial values. The solution of an IVP  is a family of smooth degree $n-k$ forms $\Omega_t$ parametrized by $t>0$,  such that, in the sense of currents the following hold:
\begin{equation} \label{eqPVI}\left\{ \begin{array}{ccl}\left(\frac{\partial}{\partial t}+\Delta_k\right)\Omega_t&=&0\\[0.2cm]
\displaystyle\lim_{t\ra 0} \Omega_t&=&L. 
\end{array}\right.\end{equation}

For an exact $L$, the  unique solution $\Omega_t$ satisfies, as expected:
\begin{equation} \label{eqi1}\int_0^\infty\Omega_t~dt=G_k(L).\end{equation}

 More importantly, in a neighborhood of $L$ the solution $ \Omega_t$ can be written as
\[ 
\frac{1}{(4\pi t)^{\frac{n-k}{2}}}e^{-\frac{r^2}{4t}}\sum_{i=0}^{N} t^i\eta_i + \ \ \mathcal{O}\left(t^{N +1 - \frac{n-k}{2}}\right)
\]
 when $t\ra 0$. Here $r$ is the distance function to $L$ and $\eta_i$ are smooth forms in a neighborhood of  $L$ that satisfy a recurrence relation similar to the forms that make up the Hadamard parametrix (the formal solution) of the heat kernel.

More precisely, up to multiplication with the square root of the jacobian of the normal exponential map, $\eta_0$ is the parallel transport along normal geodesics of the normal volume form of $L$. Then $\eta_{i+1}$ is the solution of an ODE that involves the Laplacian of $\eta_i$. These ODE's have smooth solutions. However, in order to prove the blow-up extendability of $d^*\displaystyle\int_0^{\infty}\Omega_t~dt$ more information is needed about the forms $\eta_i$. 

Our main result shows that with the amount of information we have about $\eta_0$ the following holds:

\begin{theorem} \label{TA} Let $L\subset M^n$ be a closed oriented, submanifold of codimension $1$, $2$, $3$ or $n$. If $\codim{L}=2$ assume that $L$ is a minimal submanifold. Then $\BS(L):=d^*G(L-H_L)$ is a smooth form on $M\setminus L$ which has locally integrable coeficients\footnote{being $L^1_{\loc}$ is in fact a consequence of being blow-up extendible} is extendible to the oriented blow-up of $L$ and satisfies \[d\BS(L)=L.\]
 \end{theorem}

The extendibility in codimension $3$ allows us to reach our main application. By definition a knot in $M$ is always an oriented, closed, simple, smooth curve which is the boundary of an oriented (Seifert) surface $S$.

\begin{theorem}\label{MT1} Let $M$ be an oriented, compact, Riemannian manifold of dimension $3$. Let $\delta^M\subset M\times M$ be the diagonal and $H_{\delta^M}$ be the harmonic representative of the Poincar\'e dual of $\delta^M$. Let  $\BS(\delta^M):=d^*G(\delta^M-H_{\delta^M})$ be the solution that Hodge theory provides of the equation
\[ d\omega=\delta^M-H_{\delta^M}\]

For any two  oriented knots $K_1$ and $K_2$ in $M$ the following holds
\[\lk(K_1,K_2)=\int_{K_1\times K_2} \iota^*\BS(\delta^M).\]
where $\iota:K_1\times  K_2\ra M\times M\setminus \delta^M$ is the natural inclusion.
\end{theorem}
We show in Section \ref{rnrn} that when $M=\bR^3$ the same recipe gives back Gauss formula. In fact, our Theorem \ref{MT1} recovers Vogel's differential linking form \cite{Vo} proved there only for real cohomology spheres, because of the following result (see Theorem \ref{Gradiag}).
\begin{theorem} The Biot-Savart form $\BS(\delta^M)$ of Theorem \ref{MT1} equals $d_y^*\bG$ where $\bG(x,y)$ is the total Green kernel of the Laplacian $ \Delta:\Omega^*(M)\ra \Omega^*(M)$. This is the kernel of the Biot-Savart operator $(-1)^nd^*G:\Omega^*(M)\ra \Omega^{*-1}(M)$. 
\end{theorem}

As a consequence of Theorem \ref{TA} we get some asymmetric linking formulas as well in a general $n$-dimensional Riemannian manifold $M$.

\begin{theorem} Let $L\subset M$ be  an exact submanifold of codimension  $1$, $2$ or $3$. If $\codim L=2$ assume $L$ is minimal. Let $K$ be an oriented closed submanifold with $\dim{K}=\codim{L}-1$ such that $L\cap K=\emptyset$. Then 
\[ \lk(K_1,L)=(-1)^{\dim{L}}\int_{K_1}\BS(L)\]
where $\BS(L)=d^*G(L)$.
\end{theorem}

The simple "philosophy" behind equation (\ref{eqi0}) as a pointer towards differential linking formulas is the following. The linking number $\lk(K_1,K_2)$ between two boundaries $K_1$ and $K_2$  in other ambient manifolds is the intersection number between an oriented "Seifert surface" $S_1$ with $\partial S_1=K_1$ and $K_2$.  It thus feels natural to substitute the Seifert surface $S_1$ by a differential form $\Omega_1$ that satisfies (\ref{eqi0}) with $L=K_1$. After all, they are both solutions in the world of currents of the same equation $dT=K_1=:L$. In a sense, this is what Bott-Tu \cite{BT} (page 230) do with their differential topology definition of linking number valid on any homology spheres, except that they start by switching first to the Poincar\'e dual of $K_1$, which eliminates the need for the use of currents but makes concrete formulas perhaps less tractable.

 The equality of the intersection number $I(S_1,K_2)$ with the integral of $\Omega_1$ over $K_2$ is not immediate. We look at it as a question about intersection theory (or more generally, pull-back) of currents. In short, we want to ensure that the pull-back of $\Omega_1$ via a transverse map $\varphi$  to $K_1$ (like the inclusion $S_2\hookrightarrow M$ for example) goes on to satisfy (\ref{eqi0}) now with $\varphi^{-1}K_1$ on the right hand side. This seems to go to the heart of what (\ref{eqi0}) means and our Theorem \ref{th1} gives an equivalent characterization of (\ref{eqi0}) where the notion of blow-up extendibility enters fundamentally. 

While an equation such as (\ref{eqi0}) has an infinite number of solutions, we want one that is rather explicit and blow-up extendible. This is where Hodge theory provides an answer. While Hodge theory has been used in connection with linking numbers our approach is slightly different.

  Following a line sketched in  \cite{ArK}, Vogel uses in   \cite{Vo} the full Green kernel as an explicit linking form on {real homology spheres} of dimension $3$. Our approach recovers his result as a particular case, adding further insight on the role played by the Green kernel in the construction of a differential linking form.  In higher dimension,  on negatively curved symmetric spaces,  Bechtluft-Samiou \cite{BSS1,BSS2} using Hodge theory, construct explicit  solutions of the equation $d\Omega=\alpha$ where $\alpha$ is a smooth form and use this to produce a Gauss formula under the additional assumption of vanishing for certain homology groups. One important notion in \cite{BSS1} inspired by the work of deTurck-Gluck \cite{dTG}  that plays a prominent role here  as well is that of  Biot-Savart operators $\BS:=d^*G_k$ where $G_k$ is the Green operator on $k$-forms. We should mention that in \cite{dTG} the authors also present universal linking formulas in $S^3$ and $\bH^3$ (the hyperbolic $3$-space) which use certain fundamental solutions of the Laplacian operator on functions.

   In the holomorphic context, more precisely in connection with arithmetic intersection theory,  Gillet-Soul\'e \cite{GS} considered   "fundamental solutions" for the $dd^c$ operator (i.e. roughly replace $d$ by $dd^c$ in (\ref{eqi0}), which they called Green currents. Mostly within the context of Green currents  (\cite{Ha,He}) but also in connection with the universal Thom form of the diagonal (\cite{Ma})  some authors used the heat kernels of the Laplacian on forms in order to "regularize" currents like we do. This is all very natural. However, in order to overcome the analytical difficulty of being able to guarantee that the pull-back of the solutions of (\ref{eqi0}) stay solutions of the same equation we had to obtain more refined results (for small $t$) about the forms that regularize the currents.

For our purpose, the  available information about the kernels of classical operators is not enough. A certain description of the kernels $\bG_k$  of the Green operators $G_k$ on forms can be traced back to the work of deRham \cite{dR}. He showed that the kernel is smooth away from the diagonal and that close to the diagonal the forms defining $\bG_k$ are of type $d(x,y)^{2-n}$. This is enough to know they are locally integrable. However, the blow-up extendibility property requires more detailed knowledge of the behavior of $d^*G(L)$ close to $L$, not only local integrability. One of the main contributions of this article is the following asymptotic development for small $t$ of the initial value solutions for the heat kernel equation.
\begin{theorem}\label{MT2} Let $M$ be a compact manifold and $L$ a closed submanifold of dimension $k$. Let $\Omega_t$ be the $(n-k)$-forms which are unique solutions to (\ref{eqPVI}). Then close to $L$ 
\begin{equation}\label{asaprox}
	\Omega_t =P_t(L)=\frac{1}{(4\pi t)^{\frac{n-k}{2}}}e^{-\frac{r^2}{4t}}\sum_{i=0}^{N} t^i\eta_i+O(t^{N-(n-k)/2 +1})
\end{equation}
where $r$ is the distance function to $L$ and $\eta_i$ are smooth forms that satisfy certain ODE recursive relations. 
\end{theorem}
The proof of this result proceeds as in the case of the asymptotic approximation of the heat kernel (\cite{BGV}, Ch 2)  with the construction of a formal solution. However the transition from a formal solution to a genuine solution needs a new idea since the "Volterra series" process  does not work anymore. Fortunately a slightly different embodiment of the Duhamel principle comes to help.

With Theorem \ref{MT2} in hand,  Theorem \ref{TA} is then a detailed analysis of the divergence (i.e. $d^*$) applied to the integral  $\int_0^1~dt$ of the form on the right of  (\ref{asaprox}). 

The authors would like to thank Marco Magliaro and Anna Fino for interesting discussions on the topic of this paper. 

\vspace{0.4cm}

\section{Fundamental equations for the exterior derivative}\label{Sec2}

In $\bR^3$,  there are several equivalent definitions one can give to the \emph{linking number} that counts how many times an oriented knot $K_1$ wraps around another oriented knot $K_2$. We will stop short of reviewing them, only indicating some excellent sources  \cite{Li} or \cite{Ri} where the reader can get more information. With an eye towards generalization there are two definitions that stand out. One is intersection theoretic, where one uses a Seifert surface $S_1$ for $K_1$, i.e. an oriented surface in $\bR^3$ such that $
\partial S_1=K_1$  and defines
\[\lk(K_1,K_2)=I(S_1,K_2),\]
where $I(S_1,K_2)$ counts the points in $S_1\cap K_2$ with plus if $\ori S_1\wedge \ori K_2=\ori \bR^3$ at those points and with minus, otherwise. 

The other one  is given by the degree of the Gauss map
\[\mathscr{G}:K_1\times K_2\ra S^2,\qquad (x,y)\ra \frac{x-y}{|x-y|}.\]

It turns out that in fact these two are only equal up to a sign and the reader is invited to check this for a simple example like
 \[K_1:=\{(\cos (s),\sin (s),0)~|~s\in [0,2\pi]\} \quad\mbox{and}\quad K_2:=\{(\cos (t)+1,0, \sin (t))~|~t\in[0,2\pi]\}\] 
 
Both definitions make sense for closed oriented submanifolds $K_1$ and $K_2$  of dimensions $k$ and $n-k-1$ in $\bR^n$ which are boundaries of oriented submanifolds $S_1$ and $S_2$ respectively. We need them to be boundaries in order for the intersection number $I(S_1,K_2)$ not to depend on the choice of  $S_1$. The relation now,  is
\begin{equation}\label{eq01}\deg\mathscr{G}=(-1)^{n-k-1}I(S_1,K_2)\end{equation}
There exists nevertheless a cleaner relation in $\bR^n$, namely 
\begin{equation}\label{l10}\deg\mathscr{G}=I(S_1\times K_2,\delta_{\bR^n})\end{equation}
where $\delta_{\bR^n}$\footnote{The unorthodox notation for the diagonal is needed because of the Laplacian which will enter  later.} is the diagonal in $\bR^n\times \bR^n$. Relation (\ref{l10}) is quickly explained by the following differential topology result applied to 
\begin{equation*}\begin{split} F:\bR^n\times \bR^n\ra \bR^n\\ F(x,y)=x-y 
 \end{split}
 \end{equation*} and $L=S_1\times K_2$

\begin {lemma} \label{l1} If $F:=\bR^N\ra\bR^n$ is a smooth map $0$ is a regular point for $F$ and $L\subset \bR^N$ is a smooth, oriented, compact manifold with boundary of dimension $n$ such that $L\pitchfork F^{-1}(0)$ then 
\[ I(L,F^{-1}(0))=\deg \left(\frac{F}{|F|}\biggr|_{\partial L}\right).
\]
\end{lemma}
\begin{proof} Let $p\in L\cap F^{-1}(0)$. For a small geodesic ball $D_{\epsilon}(p)\subset L$ we have that $\tilde{F}^p:=F\bigr|_{D_{\epsilon}(p)}$ is a diffeomorphism onto is image in $\bR^n$ and so $I_p(L,F^{-1}(0))=\pm 1$ depending on whether $\tilde{F}^p$ preserves ($+1$) or reverses ($-1$) orientation. But it is standard then that this is also the degree of $\frac{\tilde{F}^p}{|\tilde{F}^p|}\biggr|_{\partial D_{\epsilon}(p)}$ with $\partial D_{\epsilon(p)}$ oriented as a boundary.

We use then the invariance under cobordism of the degree for the mapping $\frac{F}{|F|}\biggr|_{L\setminus \bigcup_p D_{\epsilon}(p)}$ and note that $\partial (L\setminus \bigcup_p D_{\epsilon}(p))=\partial L\bigsqcup_p \partial D_{\epsilon}(p)$ in order to finish the proof.
\end{proof}
The following straightforward property  of intersection numbers  explains also (\ref{eq01})
\begin{lemma}\label{l2} If $N_1$ and $N_2$ are two oriented submanifolds of an oriented manifold $M$ intersecting transversely at a single point $p$ then
\[I_p(N_1,N_2)=(-1)^{\dim N_2}I_{(p,p)}(N_1\times N_2,\delta^M)
\]
\end{lemma}
By using the  Lemmas \ref{l1} and \ref{l2} also for $K_1\times S_2$ one gets the following
\begin{corollary} The degree of the Gauss map $\mathscr{G}:K_1\times K_2\ra \mathbb{S}^{n-1}$ equals 
\[I(S_1\times K_2,\delta_{\bR^n})=(-1)^kI(K_1\times S_2,\delta_{\bR^n})=(-1)^{n-k-1}I(S_1,K_2)=(-1)^nI(K_1,S_2).
\]
\end{corollary}

This motivates the following

\begin{definition}\label{deflk} An oriented knot of dimension $k$ an oriented manifold $M^n$ is an oriented compact submanifold $K$ for which there exists an oriented manifold with boundary $S$ and an immersion $\iota:S\ra  M$, such  that
\begin{itemize}
\item[(i)] $\iota^{-1}(K)=\partial S$
\item[(ii)] $\iota\bigr|_{\partial S}$ is an oriented diffeomorphism onto  $K$.
\end{itemize}
 We write $K=\partial S$.

 Let $K_1,K_2$ be oriented knots of dimensions $k$ and $n-k-1$ in $M^n$ and $\iota_i:S_i\ra M$ ($i=1,2$) immersions   such that $K_i:=\partial S_i$, $S_1\pitchfork K_2$ and $S_2\pitchfork K_1$. The linking number $\lk(K_1,K_2)$ of $K_1$ and $K_2$ is defined as any of the following $4$ equal quantities:
\[I(S_1\times K_2,\delta^M)=(-1)^kI(K_1\times S_2,\delta^M)=(-1)^{n-k-1}I(S_1,K_2)=(-1)^nI(K_1,S_2).
\]
\end{definition}
\begin{remark} With this definition, the homology class in $H_*(M)$ of an oriented knot is trivial.  The fact that the  intersection numbers do not depend on the choice of $S_1$ (nor $S_2$) is immediate. If $S_1'$ is another choice which induces the same orientation on  $K_1$ then $S:=S_1\cup_{K_1} S_1'$ is, with a bit of smoothing of the corners, a closed immersed submanifold that intersects $K_2$ transversely. But then the geometric intersection number $I(S,K_2)$ equals the  intersection number $[S]\cdot [K_2]$ in homology, which is $0$ because $K_2$ is a boundary.

In fact, one can define a knot to be just as a closed submanifold whose class is trivial in homology and with a bit of care define the intersection and linking numbers. For a more efficient presentation, we prefered a "boundary-as-a-manifold" approach. 
\end{remark}




\vspace{0.4cm}

In this section we will see that certain solutions $\Omega_1$ and $\Omega$ of the equations (in the sense of currents)
\begin{equation}\label{sec3eq0} d\Omega_1=K_1 \qquad \mbox{ and }\qquad d\Omega=\delta^M-H
\end{equation}
can be used to produce formulas that compute the linking number $\lk(K_1,K_2)$. In (\ref{sec3eq0})
\begin{itemize}
\item $\Omega_1$ is a smooth form on $M\setminus K_1$ with $L^1_{\loc}$ coefficients on $M$.
\item $\Omega$ is a smooth form on $M\times M\setminus \delta^M$ with $L^1_{\loc}$ coefficients on $M\times M$.
\item $H$ is the smooth harmonic representative of the Poincar\'e dual to $\delta^M$ in $M\times M$.
\end{itemize}

Here, our sign convention is that the Poincar\'e duality  
\begin{equation}\label{eq_poincare_duality}
\PD: H^{n-k}_{dR}(M)\simeq H_k(\mathscr{D}'_*(M))   
\end{equation}
(cf. \cite{dR}) is induced by the inclusion
\[
\Omega^*(M)\hookrightarrow \mathscr{D}'_{n-*}(M)\qquad \omega\ra T_{\omega}(\eta):=\int_{M}\omega\wedge \eta
\]

\vspace{0.4cm}

In order to make the connection between (\ref{sec3eq0}) and linking numbers we need a result (see Theorem \ref{th1} below) that gives an equivalent characterization of solutions $\Omega$ of the equation
\[d\Omega=L\]
where $L$ is a oriented closed submanifold  of $M$. In order to be able to state that theorem, we need first a definition for which we recall some standard notions. Let $\nu L$ be the normal bundle of $L$, and for $\gamma>0$ denote by $S_\gamma(\nu L)$ the sphere subbundle of normal vectors with norm $\gamma$. The oriented blow-up  $\Bl_L(M)$ of $L$ in $M$ is a smooth manifold with boundary that comes  with a smooth (blow-down) map  $\Bl:\Bl_L(M)\ra M$ such that 
\begin{itemize}
\item The restriction $\Bl\bigr|_{\Bl_L(M)\setminus \partial \Bl_L(M)}$ is a diffeomorphism onto $M\setminus L$. 
\item There exists a diffeomorphism $\alpha:\partial \Bl_L(M)\simeq S_1(\nu L)$ such that the following diagram is commutative
\[\xymatrix {\partial \Bl_L(M) \ar[rr]^{\alpha}\ar[dr]_{\Bl} & & S_1(\nu L) \ar[dl]^{\pi}\\
 & L&
}\]
where $\pi:S_1(\nu L)\ra L$ is the natural projection of the spherical normal bundle to $L$. 
\end{itemize}

Note that the spherical normal bundle $S_1(\nu L)$ is naturally defined, without reference to a particular metric  as the\emph{ oriented}  projective fiber bundle associated to  $\nu L:= TM\bigr|_{L}/TL$, i.e. as the quotient $\nu L\setminus [0]/\sim$ where $(p,v)\sim (p,w)$ if there exists $\lambda >0$ such that $v=\lambda w$.

The above properties characterize the pair $(\Bl_L(M),\Bl)$ uniquely up to a diffeomorphism which commutes with the blow-down maps to $M$.

 Once a metric is fixed on $M$, the manifold $\Bl_L(M)$ has a concrete realization as the "gluing" of the manifold with boundary $[0,\epsilon)\times S_1(\nu L)$ with $ M\setminus L $ via the following mapping diffeomorphism onto its image
\begin{equation}\label{bleq}
	(0,\epsilon)\times S_1(\nu L)\ra M\setminus L,\qquad  (t,p,v)\ra \exp_p(tv).
\end{equation}
The tubular neighborhood theorem says that for small $\epsilon>0$ this is a diffeomorphism onto its image, an open subset $T \subset M\setminus L$. 

Note also that this concrete realization of the blow-up comes with more structure than what is assumed for $\Bl_L(M)$. Namely one gets also fixed a collar of the boundary $[0,\epsilon)\times S_1(\nu L)$. This gives rise to a natural bigrading of differential forms defined on this collar, the bigrading being the same as on any product manifold.

 \begin{definition} \label{wkext}  A smooth form $\Omega$ of degree $k$ on $M\setminus L$ is said to be (blow-up) extendible if the pull-back $\Bl^*\Omega$ can be extended to a continuous differential form on all of $\Bl_L(M)$.

  A smooth form $\Omega$ of degree $k$ on $M\setminus L$ is said to be weakly (blow-up) extendible if there exists a Riemannian metric for which,  in the concrete realization of $\Bl_L(M)$ described above,  the bidegree $(0,k)$ part of $\Bl^*\Omega$ (with respect to the natural bigrading induced by the collar splitting) extends to a continuous form on $\Bl_L(M)$. 
\end{definition}

\begin{remark} The property of a form of being extendible is a differentiable property, i.e. if there exists a diffeomorphism $\varphi (M_1,L_1)\ra (M,L)$ of pairs and $\Omega\in\Omega^*(M\setminus L)$ is an extendible form then clearly $\varphi^*\Omega$ is an extendible form on $M_1\setminus L_1$. On the other hand the weak extendibility property depends on the choice of a collar for $\partial \Bl_L(M)$, 
\end{remark}
\begin{remark} The two notions of Definition \ref{wkext} make sense on a non-compact Riemannian manifold by adapting the blow-up construction via standard modification of (\ref{bleq}).
\end{remark}
Section  \ref{apB} in the Appendix is dedicated to dealing with alternative characterizations of weak and full-extendibility. For convenience, we state one result which we will need momentarily. Let
\[ \varphi_{\gamma}:S_1(\nu L)\ra S_{\gamma}(\nu L)\qquad (p,v)\ra (p,\gamma v) \]
be a rescaling diffeomorphism. Assuming for the moment that $S_\gamma(\nu L)$ is contained in the open set where the normal exponential is a diffeomorphism, let \[
\hat{\varphi}_{\gamma} \ : \ S_1(\nu L) \to T
\]
be the composition of $\varphi_{\gamma}$ with the tubular neighborhood diffeomorphism. 

The space of forms $\Omega^*(S_1(\nu L))$ is clearly a Banach space, hence a continuous curve $[0,\epsilon]\ra \Omega^*(S_1(\nu L))$ is uniformly continuous. 
\begin{prop} \label{prwex} \begin{itemize} 
\item[(i)] A form $\Omega\in \Omega^*{(M\setminus L)}$ is weakly extendible if and only if the family of forms $\gamma\ra \hat{\varphi}_{\gamma}^*(\Omega)$  is uniformly continuous on some interval $(0,\epsilon]$.
\item[(ii)]  A form $\Omega\in \Omega^*{(M\setminus L)}$ is blow-up extendible if and only if both $\Omega$ and $\iota_{\nabla r }\Omega$ are weakly extendible where $\nabla r$ is the unit radial vector field defined in a tubular neighborhood of $L$.\\
\end{itemize}
\end{prop}

\begin{proof} See Appendix \ref{apB}, Proposition \ref{prwexB}
\end{proof}

The next result gives a geometric characterization of the equation
 \[d\Omega=L.\]  

We need to fix some conventions first. The embedding $\Omega^k(M)\hookrightarrow \mathscr{D}'_{n-k}(M)$ is given by 
\[\omega\ra \left\{\eta\ra \int_{M}\omega\wedge \eta\right\},
\]
and the exterior derivative $d$ on $\mathscr{D}'_{n-k}(M)$ is defined as to extend $d$ on $\Omega^k(M)$, see Definition \ref{sgnl0} and  Lemma \ref{sgnl}.

An oriented submanifold $N\subset M$  that intersects $L$ transversely at isolated points, is said to be $L$-transversely oriented if at any point of intersection  the following holds:
\[\ori N\wedge\ori L=\ori M.\]

The rest of the section contains a proof and  some relevant consequences for the following:

\begin{theorem} \label{th1} Let $L^k \subset M^n$ be a connected, oriented, closed (as a subset) submanifold of an oriented manifold $M$. Let $f:L\ra \bR$ be a continuous function, $\Omega$ a  $C^1$ form on $M\setminus L$ of degree $n-k-1$ with locally integrable coefficients on $M$ and $H$ a  continuous  form on $M$ of degree $n-k$.
 
Consider the following statements:
 \begin{itemize}
 \item[(a)] $d\Omega =fL-H$ in the sense of currents;
\vspace{0.3cm}

 \item[(b1)] $d\Omega=-H$ as forms on $M\setminus L$;
 \item[(b2)]  \begin{equation}\label{dOH} \int_{B(p)}H+\int_{\partial B(p)}\Omega=f(p),\end{equation}   for any  compact, $L$-transversely oriented submanifold with boundary $B(p)\subset M$ of dimension $n-k$ that intersects transversely $L$ uniquely at a point $p\notin \partial B(p)$ . \end{itemize}
  
Then (a) implies (b1) and (b2) for \emph{any} $B(p)$ with the mentioned properties.

Conversely, under the assumption that $\Omega$ is weakly extendible, if (b1) holds  and (b2) holds for \emph{some} $B(p)$  then (a) is valid. 
  
Moreover, equation (a) implies that $f$ is a constant function along $L$, $H$ is closed if it is $C^1$ and $fL$ and $H$ are Poincar\'e duals to each other.

\end{theorem}
Before we give the proof we start with a digression about how  (b1) implies that the quantity
 \[
 \int_{B(p)}H+\int_{\partial B(p)}\Omega
 \]  
 depends neither on the choice of the submanifold $B(p)$ with the property described at (b2) nor on the point $p\in L$.

\begin{definition} Let $(\omega,\eta)\in \Omega^k(M)\oplus \Omega^{k-1}(U)$ be a pair of forms where $U$ is an open subset of $M$. The pair is called closed if $d\omega=0$ on $M$ and $d\eta=-\omega\bigr|_{U}$ on $U$, in other words if 
\[ d(\omega,\eta):=(-d\omega,\iota^*\omega+d\eta)=0\]
where $\iota:U\ra M$ is the inclusion
\end{definition}

\vspace{0.2cm}

\begin{example}
In Theorem \ref {th1}, if  $H$ is $C^1$ then condition (b1) says that $(H,\Omega)$ is a closed pair on $(M,M\setminus L)$. 
\end{example}

\vspace{0.2cm}

\begin{example}
A well-known example of a closed pair is the Pfaffian and the negative of its  transgression on a pair $(E,E\setminus {\bf 0})$ where $E$ is a Riemannian vector bundle of even rank on $M$ and ${\bf 0}$ is its zero-section .
\end{example}

\vspace{0.2cm}

\begin{definition}
The integral of a pair $(\omega,\eta)\in  \Omega^k(M)\oplus \Omega^{k-1}(U)$ (not necessarily closed) on a compact, oriented  submanifold with boundary $(N,\partial N)\subset (M,U)$ where $\dim{N}=k$ is defined to be
\[
\int_{(N,\partial N)}(\omega,\eta):=\int_N\omega+\int_{\partial N} \eta
\]
\end{definition}
\vspace{0.2cm} 

Note that if $N\subset U$ and the pair $(\omega,\eta)$ is closed then the integral is $0$. The following is an analogue for pairs of the homotopy invariance of a top degree form.

\begin{lemma}\label{hompa} Let $(\omega,\eta)\in \Omega^k(M)\oplus \Omega^{k-1}(U)$ be a closed pair of forms. Let \[
	H:I\times (N,\partial N)\ra (M,U) 
\]
be a smooth homotopy of pairs, i.e $H_t$ is a mapping of pairs for all $t \in I$. Then
\[ 
\int_{(N,\partial N)}(H_1^*\omega,H_1^*\eta)= \int_{(N,\partial N)}(H_0^*\omega,H_0^*\eta)
\]
\end{lemma}
\begin{proof} For $\omega$ of degree equal to $\dim N$ we use Stokes on the cylinder $I\times N$:
\begin{equation}\label{Stc1}
	\int_{I\times N}dH^*\omega=\int _NH_1^*\omega-\int_{N}H_0^*\omega-\int_{I\times \partial N} H^*\omega
\end{equation}
where $I\times \partial N$ is oriented using the product orientation.  Moreover we have
\begin{equation}\label{Stc2}\int_{I\times \partial N}dH^*\eta=\int_{\partial N}H_1^*\eta-\int_{\partial N}H_0^*\eta.
\end{equation}
Taking the sum of \eqref{Stc1} and \eqref{Stc2} and rearranging we get
\[\int_{(N,\partial N)}(H_1^*\omega,H_1^*\eta)- \int_{(N,\partial N)}(H_0^*\omega,H_0^*\eta)=\int_{I\times N}dH^*\omega+\int_{I\times \partial N}(H^*\omega + dH^*\eta)
\]
and since the pair $(\omega,\eta)$ is closed both integrals on the right hand side vanish.
\end{proof}

As a consequence of the above Lemma, we prove an invariance property for the left-hand side of \eqref{dOH}.

\begin{lemma}\label{lem_invariance} 
	Let $(\Omega,H,f)$ as in the statement of Theorem \ref{th1}. If \tcr{(b1)} holds, i.e., if
\[
d\Omega=-H \quad\mbox{on} \quad M\setminus L,
\] 
then the quantity 
\begin{equation}\label{eq_sum_invariance}
\int_{(B(p),\partial B(p))} (H,\Omega) : = \int_{B(p)}H+\int_{\partial B(p)}\Omega 
\end{equation}
does not depend on the choice of the transverse manifold $B(p)$ of dimension $n-k$ nor on the point $p$. Therefore, if (b2) holds then  $f$ must be a constant function on each connected component of $L$.
\end{lemma}
\begin{proof} 
	First of all, by Stokes Theorem and property \tcr{(b1)} the sum  \eqref{eq_sum_invariance} is zero for any connected component $C$ of $B(p)$ with $C\subset B(p)\cap (M\setminus L)$, if such a component exists. Thus, without loss of generality we can assume that $B(p)$ is connected. Again Stokes Theorem and $d\Omega=-H$ give that the sum equals  $\int_{B(p)\cap D}H+\int_{\partial (B(p)\cap D)}\Omega$ where $D$ is a small closed tubular neighborhood of $L$ with regular boundary. Therefore, if $p_1,p_2 \in L$ and $B_i(p_i)$ are transverse manifolds at $p_i$, by choosing $D$ small enough the intersections $B_i(p_i) \cap D$ are small normal disks. Such disks are isotopic to each other, with the isotopy preserving separately $D \backslash L$ and $L$. Thus, Lemma \ref{hompa} applies to conclude that the integral $\int_{(B(p),\partial B(p))}(H,\Omega)$ does not depend on the choices of $p$ and $B(p)$.

%
\end{proof}

  \begin{proof}[Proof of Theorem \ref{th1}] Let us split the proof into three steps: \\[0.2cm]
  \noindent \textbf{Step 1:} (a) implies (b1) and (b2).\\[0.2cm] 
  By our sign agreement on $d$, we have that $d\Omega=-H$ as currents on $M\setminus L$ is equivalent to $d\Omega=-H$ as forms on $M\setminus L$, hence (b1) holds. To prove (b2), we fix a metric on $M$. Let $\delta>0$ be a small number and let $D_{\delta}$ be a tubular neighborhood of radius $\delta$ around $L$ with $D_{\delta}(p)\subset D_{\delta}$ the normal disk corresponding to $p$ while $S_{\delta}(p):=\partial D_{\delta}(p)$.  By Lemma \ref{lem_invariance} the quantity
\[ 
\sigma =\int_{D_{\delta}(p)}H+\int_{\partial S_{\delta}(p)}\Omega
\]
does not depend on $\delta$ \text{and $p$.}

 Let $r(\cdot):=\dist(\cdot,L)$ be the distance function to $L$ and let $\rho(r)$ be a function which is $1$ on $D_{\epsilon}$ and $0$ outside $D_{2\epsilon}$. We prove that for every form $\eta$ on $L$ the following holds:
\begin{equation}\label{fl0} 
	\int_L (\sigma -f)\eta=0
\end{equation}
from which one  deduces immediately that $\sigma \equiv f$ on $L$. In fact, we prove that
\begin{equation}\label{fl1} \lim_{\gamma\ra 0}\int_L\left[\int_{D_{\gamma}/L}H+\int_{\partial D_{\gamma}/L}\Omega\right]\eta=\int_L f\eta. 
\end{equation}
This is a "fake" limit because  the quantity in square brackets does not depend on $\gamma$. Hence (\ref{fl1}) implies (\ref{fl0}).
 
Let $\hat{\pi}:D_{\gamma}\ra L$ be the projection that the tubular neighborhood diffeomorphism provides.
 Due to the continuity of $H$ we have that:
 \[\lim_{\gamma\ra 0}\int_L\left(\int_{D_{\gamma}/L}H\right)\eta=\lim_{\gamma\ra 0}\int_{D_{\gamma}}H\wedge \hat{\pi}^*\eta=0
 \] 
 Hence, in order to justify (\ref{fl1})  it is enough to prove the following 
 \begin{equation}\label{fl5}\lim_{\gamma \ra 0}\int_L\left( \int_{\partial D_{\gamma}/L}\Omega\right)\eta=\int_L\eta f \qquad \forall \, \eta \in \Omega^k(L).
 \end{equation}
 
 We apply Stokes to $\Omega\wedge\rho\hat{\pi}^*\eta $ on $D_{2\epsilon}\setminus D_{\gamma}$. Recall that $\rho$ is zero outside $D_{2\epsilon}$ and $1$ on $D_{\epsilon}$. Hence
 \begin{equation}\label{fl2}-\int_{D_{2\epsilon}\setminus D_{\gamma}} d(\Omega\wedge \rho\hat{\pi}^*\eta)=\int_{\partial D_{\gamma}}\Omega\wedge \hat{\pi}^*\eta =\int_L\left(\int_{\partial D_{\gamma}/L}\Omega\right)\eta
 \end{equation}
 On the other hand
 \[-\int_{D_{2\epsilon}\setminus D_{\gamma}} d(\Omega\wedge\rho\hat{\pi}^*\eta)=\int_{D_{2\epsilon}\setminus D_{\gamma}}H \wedge (\rho\hat{\pi}^*\eta)+(-1)^{n-k}\int_{D_{2\epsilon}\setminus D_{\gamma}} \Omega\wedge d(\rho\hat{\pi}^*\eta)
 \]
 Owing to the fact that $\Omega$ has $L^1_{loc}$ coefficients and $H$ as well we can pass to limit and obtain
 \begin{equation}\label{fl3}-\lim_{\gamma\ra 0}\int_{D_{2\epsilon}\setminus D_{\gamma}} d(\Omega\wedge \rho\hat{\pi}^*\eta)=\int_{D_{2\epsilon}}H\wedge (\rho\hat{\pi}^*\eta)+(-1)^{n-k}\int_{D_{2\epsilon}}\Omega\wedge d(\rho\hat{\pi}^*\eta)
 \end{equation}
Owing to the fact that $\rho\hat{\pi}^*\eta$ is a form with compact support in $D_{2\epsilon}$ and with our sign conventions we get that (\ref{fl3}) equals:
 \begin{equation}\label{fl4}H(\rho\hat{\pi}^*\eta)+d\Omega(\rho\hat{\pi}^*\eta)=fL(\rho\hat{\pi}^*\eta)=\int_{L}f\rho\eta=\int_Lf\eta
 \end{equation}
 where we used (a) in the first equality. Then (\ref{fl2}) and (\ref{fl4}) imply (\ref{fl5}) and Step 1 is proved.\\[0.2cm]
 \noindent \textbf{Step 2:} (b1)+(b2) and weak extendibility imply (a). \\[0.2cm] 
 Note that (a) rewrites as 
 \[
 (-1)^{n-k} \int_M \Omega \wedge d\eta = \int_L f\eta - \int H \wedge \eta
 \]
 for each test form $\eta \in \Omega_c^k(M)$. Fix a tubular neighborhood $D_{2\epsilon}\supset L$. Writing $\eta$ as $\eta_1 + \eta_2$, where $\eta_1$ has support in $D_{2\epsilon}$ while $\eta_2$ has support away from $L$, observe that for $\eta_2$ item (a) is equivalent with item (b1). Without loss of generality, we can thus assume that $\eta$ is supported in $D_{2\epsilon}$.

 For $\gamma<2\epsilon$ we have that
 \begin{equation}\label{nneq1}\int_{\partial D_{\gamma}}\Omega\wedge \eta=-\int_{D_{2\epsilon}\setminus D_{\gamma}}d(\Omega\wedge \eta)=-\int_{D_{2\epsilon}\setminus D_{\gamma}}d\Omega\wedge \eta+(-1)^{n-k}\int_{D_{2\epsilon}\setminus D_{\gamma}}\Omega\wedge d\eta
 \end{equation}
 Letting $\gamma\ra 0$ in (\ref{nneq1}), by (b1)  we get that
 \[\lim_{\gamma\ra 0}\int_{\partial D_{\gamma}}\Omega\wedge \eta=H(\eta)+ d\Omega(\eta).\]
 
Therefore, since the limit of the  first term in the equation below is obviously $0$ we have
 \[\lim_{\gamma\ra 0}\int_{D_{\gamma}}H\wedge\eta+\int_{\partial D_{\gamma}}\Omega\wedge\eta=H(\eta)+d\Omega(\eta)
 \]
 On the other hand, we already know by (b2) that for all small $\gamma$ the following holds
 \[\int_{D_{\gamma}}H\wedge (\hat{\pi}^*\iota^*\eta)+\int_{\partial D_{\gamma}}\Omega\wedge(\hat{\pi}^*\iota^*\eta)=\int_{L}f\iota^*\eta=fL(\eta)
 \]
 and we also know due to the continuity of $H$ that 
 \[\lim_{\gamma\ra 0}\int_{D_{\gamma}}H\wedge\eta+\int_{D_{\gamma}}H\wedge (\hat{\pi}^*\iota^*\eta)=0\]
 hence if we can prove that
 \begin{equation}\label{eqwex}\lim_{\gamma\ra 0}\int_{\partial D_{\gamma}}\Omega\wedge (\eta-\hat{\pi}^*\iota^*\eta)=0
 \end{equation}
 then we are done. We use that the rescaling maps $\hat{\varphi}_{\gamma}$ satisfy
\[ 
\lim_{\gamma \ra 0} \hat{\varphi}_{\gamma}^*\eta =\pi^*\iota^*\eta \qquad \text{uniformly in } \, S_1(\nu L),
\]
and that the identity $\pi = \hat \pi \circ \hat{\varphi}_{\gamma}$ implies  $\hat{\varphi}_{\gamma}^*\eta = \hat{\varphi}_{\gamma}^*\hat\pi^*\iota^*\eta$ for each (small) $\gamma$. By using the weak extendibility of $\Omega$ as in Proposition \ref{prwex} we therefore get
\[\lim_{\gamma\ra 0}\int_{S_1(\nu L)} \hat{\varphi}_{\gamma}^*\Omega\wedge  \hat{\varphi}_{\gamma}^*(\eta-\hat\pi^*\iota^*\eta)=0\]
which of course is a change of variables in (\ref{eqwex}).

  \end{proof}
  \begin{remark} It does not seem that the  local integrability  of $\Omega$, rather than weak extendibility is  enough to get (\ref{eqwex}). 
\end{remark}
\begin{remark}\label{rth1} The proof of Theorem \ref{th1} can be adapted in a straightforward manner to the case  of a (non-compact, oriented) manifold \emph{with boundary} $M$ provided one imposes that
 \begin{itemize}
\item[(i)] $L\cap \partial M=\emptyset$,
\item[(ii)] $B(p)\cap\partial M=\emptyset$,
\item[(iii)]  the definition of the operator $d$ on currents $T\in\mathscr{D}'_{k}(M)$  which are $C^1$ (i.e. representable by a $C^1$ form) near $\partial M$ is
\[dT(\eta):=(-1)^{n-k+1}T(d\eta)+T\wedge \partial M(\eta),\]
\end{itemize}
The redefinition of $d$ is made to ensure that $dT_{\omega}=T_{d\omega}$ if $T=T_{\omega}$ is represented by a $C^1$-form $\omega$.

In the proof, one always takes the neighborhoods of $L$ not to intersect $\partial M$.
\end{remark}
\begin{remark}\label{r2th1} Another straightforward extension of Theorem \ref{th1} is to the case where $L$ is not connected. Statement (b2) is required to hold for at least one point $p$ in each connected component of $L$. 
\end{remark}

Unfortunately weak extendibility is not a property preserved by transverse pull-backs, however the stronger extendibility property is.
  \begin{prop} \label{P2}   Consider the data of Theorem \ref{th1} but with $f\equiv 1$ and assume $\Omega$ is blow-up extendible.
  
  Let $S$ be an oriented manifold, possibly with boundary and $F:S\ra M$ be a smooth map such that $F\pitchfork L$ and $F^{-1}(L)\cap \partial S=\emptyset$.  Then
  \[ dF^*\Omega=F^*d\Omega:=F^{-1}(L)-F^*H
  \]
in the sense of currents on $S$,  where $d$ is the currential exterior derivative and $F^ {-1}L $ is cooriented first (this is the Guillemin-Pollack convention)\footnote{this means the normal bundle of $F^ {-1}(L)$ receives an orientation from the normal bundle of $L$ and that is used in the first place to get an orientation on the tangent bundle of $F^ {-1}(L)$.}
  \end{prop}
  \begin{proof} The property of a form of being blow-up extendible is  preserved by transverse pull-back, since there exists a smooth lift of $F$:
\[ \hat{F}: \Bl_{F^{-1}(L)}(S)\ra \Bl_L(M)\]

 By Theorem \ref{th1} (complemented by Remark \ref{rth1}) it is therefore enough to check properties (b1) and (b2) for $F^*\Omega$. Property (b1) is trivial while property (b2) follows from the change of variables, owing to the fact that  the preimage via $F$ of a small submanifold $B(p)$ transverse to $L$  at a point $p\in \Imag F $ is a transverse submanifold to $F^{-1}(L)$ in $S$, of the same dimension as $B(p)$.
  \end{proof}
  
We return now to the equations (\ref{sec3eq0}).
  
  \begin{corollary}\label{corf} Suppose $\Omega_1$, $\Omega_2$, respectively $\Omega$ are smooth forms on $M\setminus K_1$, $M\setminus K_2$ and  $M\times M\setminus \delta^M$ respectively which are solutions of the following equations.
  \[ d\Omega_1=K_1\qquad d\Omega_2=K_2\qquad d\Omega =\delta^M-H.
  \]
 If $\Omega$ is blow-up extendible then
  \[\lk(K_1,K_2)= \int_{K_1\times K_2}\Omega
  \]
Similarly if either $\Omega_1$ or $\Omega_2$ is  blow-up extendible then the corresponding equality below holds
\[\lk(K_1,K_2)=(-1)^ {n-k-1}\int_{K_1}\Omega_2\quad \mbox{ or }\quad \lk(K_1,K_2)=(-1)^ {k(n-k)+n}\int_{K_2}\Omega_1 \]
  \end{corollary}
  \begin{proof} The following  equalities hold
\begin{equation}\label{ino} I(S_1\times K_2,\delta^M)=\int_{K_1\times K_2}\Omega
\end{equation}
\[I(S_1,K_2)=\int_{K_1}\Omega_2\qquad\qquad I(S_2,K_1)=\int_{K_2}\Omega_1
\]
We include the details just for (\ref{ino})  since the later two are analogous and easier. Consider the embedding
\[ \iota: S_1\times K_2\hookrightarrow M\times M
\]
Use Proposition \ref{P2} to conclude that
\[ d\iota^ *\Omega=(S_1\times K_2)\wedge\footnote{this is the current represented by the geometric intersection} \delta^M- \iota^ *H
\]
  
  Property (b2) of Theorem \ref{th1} and Stokes Theorem  imply that 
  \begin{equation}\label{is1}  I(S_1\times K_2,\delta^M)= \int_{S_1\times K_2}\iota^ *H+\int_{K_1\times K_2}\Omega\end{equation}
Indeed we first get (see also Remark \ref{r2th1}) from Theorem \ref{th1} that
\[I(S_1\times K_2,\delta^M)=\sum_{p \in S_1\times K_2\cap \delta^M}\int_{B(p)}H+\sum_{p \in S_1\times K_2\cap \delta^M}\int_{\partial B(p)} \Omega \]
for some disjoint, coordinate balls $B(p)\subset (S_1\setminus \partial S_1)\times K_2$.  But Stokes Theorem gives
\[\sum_{p \in S_1\times K_2\cap \delta^M}\int_{\partial B(p)} \Omega=\int_{(\partial S_1)\times K_2 } \Omega+\sum_{p \in S_1\times K_2\cap \delta^M}\int_{B(p)^c}H\]
From the last two relations we get immediately (\ref{is1}).

By a well-known fact (emboddied also in Theorem \ref{Dlf} below) the form $H$ can be written as 
   \begin{equation}\label{eqh}\sum_{k=0}^n (-1)^{kn}\sum_{i=1}^{N_k}\pi_1^*(*\omega_{k,i})\wedge \pi_2^* \omega_{k,i} 
\end{equation} where $\omega_{k,i}$ form a basis of (harmonic) forms for $H^k(M)$. Hence
  \[\int_{S_1\times K_2}\iota^ *H=\sum_{i} (-1)^{(n-k-1)n}\int_{S_1}*\omega_{n-k-1,i}\cdot \int_{K_2}\omega_{n-k-1,i}\]
  But $\omega_{k,i}$ are all closed forms and $K_2$ is a boundary hence $\int_{K_2}\omega_{n-k-1,i}=0$ and therefore $\int_{S_1\times K_2}\iota^* H=0$.

In order to get the statements for the linking number just use Definition \ref{deflk} and elementary properties of the intersection number.
  \end{proof}

  The rest of this article is concerned with the existence and explicit construction of forms as in Corollary \ref{corf}.

  \section{Hodge theory and Biot-Savart forms}
  
Let $L^k\subset M^n$ be an oriented closed submanifold of $M$, 
and let $H \in \Omega^{n-k}(M)$ be the harmonic representative of the Poincar\'e dual  $\PD^{-1}([L])$ (see (\ref{eq_poincare_duality})). In this section, we use Hodge theory in order to find explicit solutions $\Omega$ to 
  \[ d\Omega=L-H.
  \]
  
    Classical Hodge theory decomposes orthogonally (with respect to the $L^2$ inner product) the space of smooth forms 
  \[\Omega^k(M)=\Imag \Delta\oplus \Ker \Delta=\Imag (d+d^*)\oplus \Ker (d+d^*)\]
  Every exact form $\omega=d\eta$ is in $\Imag d+d^*=\Imag \Delta$. Let $G:=\Delta^{-1}$ be the Green operator, the inverse of $\Delta$ on $\Imag \Delta$. Then 
  \[ \omega=\Delta(G\omega)=dd^*(G\omega)+d^*d(G\omega)
  \]
  But $\omega-dd^*(G\omega)$ is exact and the images of $d$ and $d^*$ are $L^2$ orthogonal. Alternatively, $G$ commutes with $dd^*$ and $d^*d$. Hence
  \[\omega=dd^*(G\omega)\quad \mbox{and} \quad d^*d(G\omega)=0
  \]
  Therefore, once we know that the equation $d\eta=\omega$ has a solution, Hodge theory provides an explicit such solution:
  \begin{equation}\label{dBS}\eta=d^*G(\omega)=:\BS(\omega)
  \end{equation}
  which we call the \emph{Biot-Savart form} following \cite{BSS1, dTG}. If $\omega$ is just closed we do the following. Let $\omega^h$ be the harmonic representative of the deRham cohomology class of $\omega$. Then 
 \[\omega_0:=\omega-\omega^h\]
  is exact and hence $\omega_0=d(d^*G\omega_0)$. Consequently
  \[\omega=d[d^*G(\omega-\omega^h)]+\omega^h
  \]
  We will substitute $\omega-\omega^h$ by a current of type $L-H$.
  
  \vspace{0.5cm}
  
  We now present the currential set-up.
  
   The extension of the  Green operator to an operator $G_k:\Omega^k(M)\ra \Omega^k(M)$ is defined as  the composition $G\circ P^{\Imag \Delta}$ where  $P^{\Imag \Delta}$ is the $L^2$-orthogonal projection onto $\Imag \Delta$. 
  \begin{definition}\label{sgnl0} Let $O$ be any of the operators $d$, $d^*$, and $G_{n-k}$ acting on $\Omega^{n-k}(M)$. Define $\tilde{O}$ on $\mathscr{D}'_{k}(M)$ by duality:
  \[ \tilde{O}(T)(\omega):=T(O(\omega)).
  \]
  Define the following sign alterations for the same operators when acting on $T\in\mathscr{D}'_{n-k}(M)$:
\[ \epsilon^O(T)=\left\{\begin{array}{ccc}(-1)^ kT & \mbox{if} & O=d\\
(-1)^ {k-1}T & \mbox{if} & O=d^ *\\
T& \mbox{if} & O=G
\end{array}
\right.
\]
Define $O$ via
\[ O(T)=\epsilon^O(\tilde{O}(T)).
\] 
\end{definition}
 
  \begin{remark} According to the next Lemma, these definitions of operators extend the definitions on forms when we view $\Omega^ {n-k}$ as a subspace of $\mathscr{D}' _{k}$. We emphasize that this also means that given $\omega\in L^1_{\loc}(M)$ which is smooth on an open, dense subset $U$ then the forms $d\omega$ resp. $d^*\omega$ will represent the currents $T_{d\omega}$ and $T_{d^*\omega}$ first on $U$  and more generally on $M$ if they have locally integrable coefficients as well. 
\end{remark}
 
  \begin{lemma}\label{sgnl} If $\omega$ is a $k$-form and $\mathscr{D}'_{n-k}(M)\ni T_{\omega}(\eta):=\int_M\omega\wedge\eta$ is the current it induces, then the following relations hold:
  \[  \tilde{d}T_{\omega}=(-1)^{k+1}T_{d\omega}
  \]
  \[ \tilde{d^*}T_{\omega}=(-1)^{k} T_{d^*\omega}
  \]
  \[ G_{k}T_{\omega}=T_{G_k\omega}
  \]
  \end{lemma}
  \begin{proof}  The first one is straightforward:
  \[\tilde{d}T_{\omega}(\eta):=\int_M\omega\wedge d\eta=(-1)^{k}\left(\int_Md(\omega\wedge\eta)-\int_M d\omega\wedge \eta\right)=(-1)^{k+1}T_{d\omega}(\eta)\]
  
  Denote $\nu_{n,k}:=nk+n+1$. It is known that for the operators $*$, $d^*$ and  the Laplacian $\Delta_k$  acting on $\Omega^k$ one has
   \[d^*=(-1)^{\nu_{n,k}}*d*\quad \mbox{and}\quad *^2=(-1)^{k(n-k)}\id.\]
   \begin{equation}\label{delst}*\Delta_k*=(-1)^{k(n-k)}\Delta_{n-k}
   \end{equation}
   In what follows for the form $\omega$ of degree $k$ and $\eta$ of degree $n-k+1$
  
  \[ \widetilde{d^*}T_{\omega}(\eta):=T_{\omega}(d^*\eta)=\int_{M}\omega\wedge d^*\eta=(-1)^{\nu_{n,n-k+1}}\int_{M} \langle \omega, d*\eta\rangle~\dvol_M=\]\[=(-1)^{\nu_{n,n-k+1}}\int_{M} \langle d^*\omega, *\eta\rangle~\dvol_M=(-1)^{\nu_{n,n-k+1}+\nu_{n,k}}\int_M\langle *d*\omega,*\eta\rangle~\dvol_M\]\[=(-1)^{\nu_{n,n-k+1}+\nu_{n,k}}\int_M\langle d*\omega,\eta\rangle~\dvol_M=(-1)^{\nu_{n,n-k+1}+\nu_{n,k}}\int_M\eta\wedge *d*\omega=\]\[=(-1)^{\nu_{n,n-k+1}}\int_M\eta\wedge d^*\omega=(-1)^{\nu_{n,n-k+1}+(n-k+1)(k-1)}\int_M d^*\omega\wedge \eta=(-1)^kT_{d^*\omega}(\eta)
  \]
  since
  \[{\nu_{n,n-k+1}}+(k-1)(n-k+1)=n^2-nk+2n+1+kn-k(k-1)-n+k-1\equiv k~ (mod~2)\]
  
  For the third identity, one either combines the first two or uses that $*G_k*=(-1)^{k(n-k)}G_{n-k}$, consequence of (\ref{delst}) and that $G_k$ is self-adjoint. Hence
  \[G_k(T_{\omega})(\eta)=\widetilde{G_{k}}T_{\omega}(\eta):=\int_{M}\omega\wedge G_{n-k}\eta=(-1)^{k(n-k)}\int_{M}\langle \omega,  *G_{n-k}\eta\rangle=(-1)^{k(n-k)}\int_{M}\langle \omega, G_k*\eta\rangle=
  \]
  \[=(-1)^{k(n-k)} \int_M\langle G_k\omega, *\eta\rangle=\int_{M} (G_k\omega)\wedge \eta=T_{G_k(\omega)}(\eta)
  \]
  \end{proof}

\begin{remark}
	Let us note a basic property of these operators, namely that $\Delta$ commutes with both $d$ and $d^*$ and therefore all operators derived from $\Delta$ via functional calculus like $G$ and $e^ {-t\Delta}$ will also commute with $d$ and $d^*$.
\end{remark}
  
  The following important fact is straightforward. 
  \begin{lemma}\label{limd} Let $T\in\mathscr{D}'_{k}(M)$ be any current.  If $\omega_{\epsilon}$ ($\epsilon>0$) is a family of smooth forms of degree $n-k$ such that 
  \[\lim_{\epsilon\ra 0}T_{\omega_{\epsilon}}=T\]
  in the sense of currents, then 
  \[ \lim_{\epsilon\ra 0}O(T_{\omega_{\epsilon}})=O(T).
  \]
  \end{lemma}
  \begin{proof} For every test form $\eta$, $O(\eta)$ is another test form in our compact context. Hence
   \[\tilde{O}(T_{\omega_{\epsilon}})(\eta)=T_{\omega_{\epsilon}}(O(\eta))\ra T(O(\eta))=\tilde{O}(T)(\eta).\]
One gets the statement for $O$ by multiplying with the appropriate sign.
  \end{proof}
  
  \begin{prop} Let $S=dT$ be an exact current with $S\in \mathscr{D}'_{k}(M)$. Then in fact
  \[ S= d(d^*G(S)).
  \] 
  \end{prop}
  \begin{proof}  We use the regularization process of currents described by Federer  \cite{Fe} on page 374 (i.e. ch. 4.1.18). For a choice of a mollifier we have smooth forms $S_{\epsilon}$ and $T_{\epsilon}$ such that
 \[S_{\epsilon}-S=d(H_{\epsilon}S)+H_{\epsilon}(dS)
 \] 
and 
 \begin{equation}\label{dsvT} T_{\epsilon}-T=d(H_{\epsilon} T)+H_{\epsilon}(dT)= d(H_{\epsilon} T)+H_{\epsilon}(S)
 \end{equation}
 where $H_{\epsilon}:\mathscr{D}'_m(M)\ra \mathscr{D}'_{m+1}(M)$ is the homotopy operator. Since $S$ is exact we have
 \[S_{\epsilon}-S=d(H_{\epsilon}S)
 \]
Apply $d$ to (\ref{dsvT}) to get
 \[ dT_{\epsilon}-S=d(H_{\epsilon}(S))=S_{\epsilon}-S
 \]
 From here we deduce that  $S_{\epsilon}$ are all exact smooth forms (and equal to $dT_{\epsilon}$). Hence 
 \[S_{\epsilon}=d(d^*G(S_{\epsilon}))\] by standard Hodge theory.
 
 By the previous Lemma when $\epsilon\ra 0$ the left hand side converges to $S$ while the right hands side to $d(d^*G(S))$.
   \end{proof}
   
   \begin{definition}\label{defBS} Given $L$ an oriented closed submanifold of $M$, the current 
   \[ \BS(L):=d^*G(L-H)
   \]
   where $H$ is the harmonic representative of $\PD(L)$ is called the Biot-Savart form of $L$.
   \end{definition}
  
   We will see in the next sections that $\BS(L)$ is indeed represented by a smooth form on $M\setminus L$ with locally integrable coefficients on $M$.

  \section{The case of ambient Euclidean space} \label{rnrn}
Before we get into details on compact manifolds we take the  time to make some explict computations in $\bR^n$ emphasizing the ideas of the previous sections and connecting with the classical Gauss formula. This is self-contained and not needed anywhere else.

We assume $n\geq 3$ in what follows. One has Green operators $G_k$ acting on smooth $k$-forms with compact support.  However, the dualization is a bit more subtle since the Green operator does not preserve the compact support of test functions. But before we tackle this let us say what we aim for in this section.

Schwartz Kernel Theorem in our context says that the Green operator on forms $G_k:\Omega^k_c(\bR^n)\ra \Omega^k(\bR^n)$ is represented by a kernel $\bG$ which is an $(n-k,k)$ form with locally integrable coefficients on $\bR^n\times \bR^n$ such that
\begin{equation}\label{Gko}G_k(\omega)=(\pi_2)_* (\bG_k\wedge \pi_1^*\omega)
\end{equation}
where $\pi_j : M \times M \to M$ is th projection onto the $j$-th factor, and $(\pi_2)_*$ is integration in the first variable, see Appendix \ref{Sc}.
	 
We can now state the following

\begin{theorem}\label{Rn} Let $\delta_{\bR^n}$ be the diagonal in $\bR^n\times \bR^n$. The Biot-Savart form $\BS(\delta_{\bR^n}):=d^*G(\delta_{\bR^n})$ equals
\begin{equation}\label{Rn4}\frac{1}{|\mathbb{S}^{n-1}|}\frac{1}{|x-y|^{n}}\sum_{i=1}^n(-1)^{i-1}(x_i-y_i)(dx_1-dy_1)\wedge\ldots\wedge \widehat{(dx_i-dy_i)}\wedge \ldots\wedge (dx_n-dy_n)\end{equation}
The form $\BS(\delta_{\bR^n})$ is extendible to $\Bl_{\delta_{\bR^n}}(\bR^n\times \bR^n)$, hence a current and it satisfies 
\[d\BS(\delta_{\bR^n})=\delta_{\bR^n}\]
Moreover,  the following identities hold
\begin{itemize}
\item[(i)] $\BS(\delta_{\bR^n}) = \beta^*(\BS(\delta_0))$ where $\beta:\bR^n\times\bR^n\ra \bR^n$ is 
\[\beta(x,y):=x-y\]
and $\BS(\delta_{0})=\frac{1}{|\mathbb{S}^{n-1}|}\frac{1}{|x|^n}\sum_{i=1}^n(-1)^{i-1}x_idx_1\wedge \ldots\widehat{dx_i}\wedge \ldots\wedge dx_n$
\item[(ii)] 
\[
\BS(\delta_{\bR^n}) =  \sum_{k=0}^nd_y^*\bG_k\footnote{For the definition of $d_y^*$ see (\ref{dy})}
\] 
where $\bG_k$ are the kernels of the Green operators $G_k$ acting on $k$-forms in $\bR^n$ and $d_y^*\bG_k$ are the kernels of the Biot-Savart operators $(-1)^nd^*G_k: \Omega^{k}_{c}(\bR^n)\ra \Omega^{k-1}(\bR^n)$. 
\end{itemize}
\end{theorem}
\begin{remark} The $(n-k,k-1)$ component of $\BS(\delta_{\bR^n})$ corresponds exactly to $d_y^*\bG_k$ in the sum of item (ii).
\end{remark}
\begin{corollary} \label{BSd1}  The  form $\BS(\delta_{\bR^n})$  coincides with the pull-back $\psi^*_n\vol_{\mathbb{S}^{n-1}}$ where:
\[\vol_{\mathbb{S}^{n-1}}:=\frac{1}{|\mathbb{S}^{n-1}|}\sum (-1)^{i-1}x_idx_1\wedge \ldots \wedge\widehat{dx_i}\wedge\ldots\wedge dx_n\]
is the normalized volume of the sphere and $\psi$ is the Gauss map
\[\psi_n:\bR^n\times \bR^n\setminus \delta_{\bR^n}\ra \mathbb{S}^{n-1},\qquad (x,y)\ra \frac{x-y}{|x-y|}\]
\end{corollary}
\begin{proof} The form $\BS(\delta_{0})$ is the unique form on $\bR^{n}\setminus \{0\}$ which is scale invariant and coincides with $\vol_{\mathbb{S}^{n-1}}$ when restricted to $\mathbb{S}^{n-1}$, i.e.
\begin{equation}\label{BSd}
	\BS(\delta_{0})=\pi^*\vol_{\mathbb{S}^{n-1}} \qquad \text{on } \, \mathbb{R}^n \backslash \{0\}
\end{equation}
where $\pi:\bR^{n}\setminus \{0\}\ra \mathbb{S}^{n-1}$ is the radial projection. Relation (\ref{BSd}) is straightforward. Theorem \ref{Rn} item (i) finishes the proof.
\end{proof}
\begin{corollary}[Gauss]\label{Gauss} Let $K_1,K_2\subset \bR^3$ be two disjoint oriented knots and let $j:K_1\times K_2\ra \bR^3\times \bR^3\setminus \delta_{\bR^3}$ be the inclusion map. Then
 \begin{equation}\label{lkK1} \lk(K_1,K_2)=\int_{K_1\times K_2}j^*\BS(\delta_{\bR^3})=\deg(\psi\bigr|_{K_1\times K_2})\end{equation}
where $\psi=\pi\circ\beta\circ j$ with $\pi$ being here the radial projection onto $S^2$.
\end{corollary}

\begin{proof} The form $\BS(\delta_{\bR^3})$ is blow-up extendible to $\Bl_{\delta_{\R^3}}(\bR^3\times \bR^3)$ since it is the pull-back via a map $\beta$  transverse to $\{0\}$ with $\beta^{-1}(0)=\delta_{\bR^3}$  of the form $\BS(\delta_0)$ which is itself blow-up extendible. By Corollary \ref{corf} (which holds also in this context) the first equality of (\ref{lkK1}) holds, while the second follows from Corollary \ref{BSd1}.

\end{proof}
\begin{remark} Note that since $K_1\times K_2$ is a $(1,1)$ current in $\bR^3\times \bR^3$, it is only the $(1,1)$-part of $\BS(\delta_{\bR^3})$ that matters in the first integral of (\ref{lkK1}).  One reads the $(1,1)$ part of $\BS(\delta_{\bR^3})$ from Theorem \ref{Rn} as being 
\begin{equation}\label{BS1}
	\mathbb{BS}_{1,1} :=\sum_{i,j=1}^3b_{ij}dx_i\wedge dy_j\end{equation}
where $B=(b_{ij})_{1\leq i,j\leq 3}$ is an antisymmetric matrix with 
\[b_{12}=\frac{1}{4\pi}\frac{y_3-x_3}{|x-y|},\quad b_{13}=-\frac{1}{4\pi}\frac{y_2-x_2}{|x-y|^3},\quad b_{23}=\frac{1}{4\pi}\frac{y_1-x_1}{|x-y|^3}\]
By item (ii) of Theorem \ref{Rn} with the choices $n=3$ and $k=2$ it holds $\mathbb{BS}_{1,1} = d_y^* \mathbb{G}_2$, the kernel of the operator 
\[
-d^*G_2:
\Omega^2_c(\bR^3)\ra \Omega^1(\bR^3)
\]
With (\ref{BS1}) or with the expression of $\bG_2$ from Lemma \ref{Grk} below, it is straightforward to check that $d^*G_2$ acts on a compactly supported $2$-form $\eta=\eta_1dx_2\wedge dx_3-\eta_2dx_1\wedge dx_3+\eta_3dx_1\wedge dx_2$ as follows:
\[(4\pi)d^*G_2(\eta)= \left(\int_{\bR^3}\frac{y_3-x_3}{|x-y|^3}\eta_2(x)-\frac{y_2-x_2}{|x-y|^3}\eta_3(x)dx\right)dy_1+\]\[+\left(\int_{\bR^3}\frac{y_1-x_1}{|x-y|^3}\eta_3(x)-\frac{y_3-x_3}{|x-y|^3}\eta_1(x)dx\right)dy_2
 +\left(\int_{\bR^3}\frac{y_2-x_2}{|x-y|^3}\eta_1(x)-\frac{y_1-x_1}{|x-y|^3}\eta_2(x)dx\right)dy_3
 \]
With the Vector Analysis identifications,
\[\mathscr{X}(\bR^3)\ra \Omega^2(\bR^3),\qquad V=\eta_1\partial_{x_1}+\eta_2\partial_{x_2}+\eta_3\partial_{x_3}\ra *(\tcr{V_{\flat}})=\eta\]
\[\Omega^1(\bR^3)\ra \mathscr{X}(\bR^3),\qquad \omega\ra \omega^{\sharp}\]
the operator $-d^*G_2$ becomes the classical Biot-Savart operator $\mathscr{X}(\bR^3)\ra \mathscr{X}(\bR^3)$.
\[V\ra \left\{y\ra \frac{1}{4\pi}\int_{\bR^3} U(x,y)\times V(x)~dx\right\}\]
where $U(x,y)=\frac{y-x}{|x-y|^{3}}$. This justifies the choice of nomenclature used in this article.
\end{remark}

 We use the sign conventions  from Appendix \ref{Sc}. We kick off  the proof of Theorem \ref{Rn} with some elementary considerations.

 Every smooth form $\omega\in \Omega^k(\bR^n)$ can be written as 
\[\omega=\sum_{|I|=k}\omega_I dy_I\]
where for every ordered $I=\{i_1<i_2<\ldots <i_k\}$, $\omega_{I}:\bR^n\ra \bR$ is a smooth function and
\[dy_I=dy_{i_1}\wedge\ldots\wedge dy_{i_k}\]
form a basis for $\Lambda^k\bR^n$. It is not difficult to  check that
\[\Delta\omega=\sum_{|I|=k}(\Delta_0\omega_I)dy_I\]
where $\Delta_0$ is the Laplace operator on functions. It follows that if we want to solve 
\[\Delta \omega=\eta\]
we can use the Green operator $G_0$ on functions and solve $\Delta_0\omega_I=\eta_I$ for each $I$. Recall that 
\[G_0(f)(y)=\int_{\bR^n}\frac{1}{(n-2)|\mathbb{S}^{n-1}|}\frac{f(x)}{|x-y|^{n-2}}~dx\]
where $|\mathbb{S}^{n-1}|$ is the volume of the unit sphere $\mathbb{S}^{n-1}$.

\begin{lemma}\label{Grk} The Green kernel on $k$-forms in $\bR^n$ is given by the $(n-k,k)$ double form
\[\bG_k=\frac{1}{(n-2)|\mathbb{S}^{n-1}|}\frac{(-1)^{kn}}{|x-y|^{n-2}}\sum _{|I|=k}(-1)^{\epsilon(I)}dx_{I^c}\wedge dy_I\footnote{Note the decomposition fiber first with respect to $\pi_2$.} \]
where $I^c:=\{1,2\ldots,n\}\setminus I$ and for every $J\subset \{0,1,2\ldots\}$:
\[\epsilon(J):=\sum_{j\in J}j-\sum_{j=1}^{|J|}j\]
\end{lemma}
\begin{proof} The number $\epsilon(I)$ is  the minimal number of elementary moves necessary to bring $dx_{I}\wedge dx_I^c$ to the volume form $dx:=dx_1\wedge\ldots\wedge dx_n$. It also satisfies:
 \[*dx_{I}=(-1)^{\epsilon(I)}dx_{I^c},\quad \mbox{and}\quad \epsilon(I)+\epsilon(I^c)=k(n-k).\]

The local integrability of the coefficients of $\bG_k$ is immediate since, up to a sign, one has only one coefficient namely the kernel of Green operator on function which is well-known (and easy to show) to be integrable



We then have for $\omega=\sum_{|I|=k}\omega_Idx_{I}$:
\[G_k(\omega):=(\pi_2)_*(\bG_k\wedge \pi_1^*\omega)=\]\[=\frac{(-1)^{kn}}{(n-2)|\mathbb{S}^{n-1}|}\sum_{|I|=k}(-1)^{\epsilon(I)+\epsilon(I^c)+k^2}(\pi_2)_*\left( \frac{1}{|x-y|^{n-2}} \omega_I(x)dx\wedge dy_{I}\right)=\]\[=\sum_{|I|=k}\frac{1}{(n-2)|\mathbb{S}^{n-1}|}\left(\int_{\bR^n}\frac{\omega_I(x)}{|x-y|^{n-2}}~dx\right)dy_I=\sum_{|I|=k}G_0(\omega_I)(y)dy_I\]
\end{proof}
\begin{corollary} \label{Corgrker} The Green kernels satisfy the symmetry property
\[ R^*\bG_{n-k}=(-1)^n\bG_k\]
where $R:\bR^n\times \bR^n\ra \bR^n\times \bR^n$ is the reflection $R(x,y)=(y,x)$. Note that $(-1)^n=\det R$.
\[\]
\end{corollary}

We describe next on what currents can the Green operator act. Let $\Omega^k_{\infty}(\bR^n)$ be the space of degree  $k$ smooth forms  which vanish at $\infty$ with all its derivatives. This becomes a Frechet space when endowed with the family of semi-norms
\[\|\omega\|_N:=\sum_{|\alpha|\leq N}\sup_{x\in\bR^n}|\partial^{\alpha}\omega(x)| \]
\begin{prop} The Green operator $G_k:\Omega^k_{\cpt}(\bR^n)\ra\Omega^k_{\infty}(\bR^n)$ is continuous. 
\end{prop}
\begin{proof} We will include some details for $G_0$, the adaptation for $G_k$ being immediate. First $G_0$ takes $C^{\infty}_{\cpt}(\bR^n)$ to $C^{\infty}_{\infty}(\bR^n)$. Let $c_n:=[(n-2)|\mathbb{S}^{n-1}|]^{-1}$. For $f\in C^{\infty}_{\cpt}(\bR^n)$  let $C:=\sup |f(x)|$ and $R$ be such that $\supp f\subset B(0,R)$.
\[|G_0(f)(y)|=c_n\left|\int_{\bR^n}\frac{f(u+y)}{|u|^{n-2}}~du\right|\leq C\int_{B(-y,R)}\frac{1}{|u|^{n-2}}~du  \] 
If $|y|>MR$ for some $M$ big then
 \[|u|\geq |y|-|u+y|>(M-1)R,\qquad \forall u\in B(-y,R)\]
Hence
\[|G_0(f)(y)|\leq C|B^n|R^2\frac{1}{|M-1|^{n-2}}\Rightarrow \lim_{|y|\ra \infty}G_0(f)(y)=0.\]
By induction, since one can differentiate under the integral sign, one shows similarly that
\[\partial^{\alpha} G_0(f)(y)\ra 0\]
The continuity of $G_0$ follows remembering that if $f_n\ra f$ in $C^{\infty}_{\cpt}(\bR^n)$ then there exists an $R>0$ such that $\supp f_n \subset B(0,R)$ and implicitly $\supp f\subset B(0,R)$. Then one repeats the previous inequalities for $f_n-f$.
\end{proof}
We take now the $1$-point compactification $\widehat{\bR^n}$ of $\bR^n$ and consider on it a structure of smooth, compact manifold diffeomorphic to $\mathbb{S}^n$ via the stereographic projections. A form which is $C^{\infty}$ on $\bR^n$ and whose partial derivatives all vanish at $\infty$ extends to a smooth form on $\widehat{\bR^n}$.  Every current $\hat{T}$ on $\widehat{\bR^n}$ restricts to a current on the open subset $\bR^n$. We will denote the set of currents  on $\bR^n$ which are restrictions of flat currents (or alternatively finite order currents) from $\widehat{\bR^n}$ by  $\mathscr{F}(\bR^n)_{\infty}$. For example, every topologically closed oriented submanifold in $\bR^n$ which is the restriction of a submanifold in $\widehat{\bR^n}$ is of this type. Clearly all flat currents with compact support in $\bR^n$ are also of this type.

  We define the Green operator on currents by duality on $\mathscr{F}(\bR^n)_{\infty}$
\[ G:\mathscr{F}(\bR^n)_{\infty}\ra \mathscr{D}'_{*}(\bR^n),\qquad G(T)(\omega):=\hat{T}(G(\omega))\]
where $\hat{T}$ is flat such that $\hat{T}\bigr|_{\bR^n}=T$. Flatness (or finite order) guarantees that this is well-defined. This is because if we take two such currents $\hat{T}_1$ and $\hat{T}_2$ then $\supp(\hat{T}_1-\hat{T}_2)\subset \{\infty\}$. Therefore either section 4.1.15 in \cite{Fe}  (for flat currents) or Theorem 2.3.4 in \cite{H1} implies that $\hat{T}_1-\hat{T}_2$ is a finite sum of Dirac distributions and their derivatives at $\{\infty\}$. But $\partial^{\alpha}(G(\omega))(\infty)=0$.


\vspace{0.4cm}

With this definition we prove
\begin{prop}\label{let} Let $L^k\subset \bR^n$ be a linear subspace. If $n-k\geq 3$, the current $G(L)$ is represented by the $(n-k)$-form on $\bR^n$ with locally integrable coefficients
\begin{equation}\label{GLx} G(L)(x)=\frac{1}{(n-k-2)|\mathbb{S}^{n-k-1}|}\cdot \frac{1}{d(x,L)^{n-k-2}}d\vol_{L^{\perp}}\end{equation} 
where the orientation on $L^{\perp}$ is such that $\ori(L^{\perp})\wedge \ori(L)=\ori(\bR^n)$. 

The form coincides with $\pi_{L^{\perp}}^*(G(\delta_0))$ where $\pi_{L^{\perp}}:\bR^n\ra L^{\perp}$ is the orthogonal projection and $\delta_0$ is here the Dirac $0$-current in $L^{\perp}$.
\end{prop}
\begin{proof} It is enough to prove the result for $L:=\{0\}\times \bR^k\subset \bR^{n-k}\times \bR^k$. The reason for choosing $\{0\}\times \bR^{k}$ rather than $\bR^{k}\times \{0\}$  is because
 \[\ori(\bR^{n-k}\times \{0\})\wedge \ori(\{0\}\times \bR^k)=\ori(\bR^n)\]
as required by the convention (see Appendix \ref{Sc}) of putting the normal bundle first.

One recovers the result for $L$ by choosing $A\in SO(n)$ such that  $A\bigr|_{\{0\}\times \bR^k}:\{0\}\times \bR^k\ra L$ is an isometry preserving the orientation and so is $A\bigr|_{\bR^{n-k}\times \{0\}}:\bR^{n-k}\ra L^{\perp}$. Clearly $A_*$ leaves $\mathscr{F}_{\infty}$ invariant and the following holds:
 \[G\circ A_*=A_*\circ G\] 
Then 
\[G(L)=G(A_*(\{0\}\times \bR^k))=A_*(G(\{0\}\times \bR^k))=(A^{-1})^*G(\{0\}\times \bR^k)\]
The result for general $L$  then follows from 
\[
(A^{-1})^*(d(x,\{0\}\times \bR^k))=d(x,L) \quad \mbox{and}\quad (A^{-1})^*\dvol_{\bR^{n-k}\times \{0\}}=d\vol_{L^{\perp}}.
\]

Denote by  $(\underline{x},\underline{\underline{x}})$ the coordinates on $\bR^{n-k}\times \bR^{k}$ and set $c_n:=[(n-2)|\mathbb{S}^{n-1}|]^{-1}$. We show that $G(L)$ is represented by the form
\begin{equation}\label{GL}(-1)^{n(n-k)} (\pi_2\bigr|_{L\times \bR^n})_*\left (\bG_{n-k}\bigr|_{L\times \bR^n}\right)=c_n\left(\int_{L}\iota_{L}^*\left(\frac{1}{|(0,\underline{\underline{x}})-(\underline{y},\underline{\underline{y}})|^{n-2}}d\underline{\underline{x}}\right) \right)d\underline{y}\end{equation}

Before proving this claim we note there are two apriori issues with (\ref{GL}):
\begin{itemize}
\item[(i)] the local integrability of $\Omega:=(-1)^{n(n-k)}\bG_{n-k}\bigr|_{L\times \bR^n}=(-1)^{n(n-k)}\iota_{L\times \bR^n}^*\bG_{n-k}$ as a form on $L\times \bR^n$ in order to determine a current in $\bR^n\times \bR^n$;
\item[(ii)] the fiber integrability of $\Omega$ along the fibers of $\pi_2^L:=\pi_2\bigr|_{L\times \bR^n}:L\times \bR^n\ra \bR^n$. 
\end{itemize}
Both integrability problems and the equality of (\ref{GL}) are dealt with below. Assuming that they are true we check that
 $(-1)^{n(n-k)}({\pi}_2^L)_*\bG_{n-k}$  represents $G(L)$ in order to finish the proof of (\ref{GLx}). We will use  $R(x,y)=(y,x)$ as an isomorphism of fiber bundles over $L$ with total spaces $\bR^n\times L$ and $L\times \bR^n$  and Corollary \ref{Corgrker}:
\[G(L)(\omega)=L(G(\omega))=\int_{L}(\pi_2)_*(\bG_k\wedge \pi_1^*\omega)=\int_L\int_{\pi_2}\bG_k\wedge\pi_1^*\omega=\]
\[=(-1)^{n}\int_L\int_{\pi_1}\bG_{n-k}\wedge\pi_2^*\omega=(-1)^{n^2}(-1)^{kn}\int_{L\times \bR^n}\bG_{n-k}\wedge \pi_2^*\omega\]
The last equality is because $\pi_1:L\times \bR^n\ra L$ is oriented base first and  Fubini gets a $(-1)^{kn}$ sign. Note that $\pi_2^{L}:L\times \bR^n\ra \bR^n$ is oriented fiber first and the projection formula  which works for fiber integrable forms as well, gives that
\[G(L)(\omega)=(-1)^{n(n-k)}\int_{\bR^n}\left(\int_{\pi_2^L}\bG_{n-k}\right)\wedge \omega.\]

 Returning to the integrability issues,  the only term in $\bG_{n-k}=\frac{c_n(-1)^{n(n-k)}}{|x-y|^{n-2}}\sum_{|J|=n-k}(-1)^{\epsilon({J})}dx_{J^c}\wedge dy_J$ that survives when restricting to $L\times \bR^n$ is the one with $J=\{1,\ldots,n-k\}$ for which $(-1)^{\epsilon(J)}=1$. Hence
\[\Omega=\frac{c_n}{|(0,\underline{\underline{x}})-(\underline{y},\underline{\underline{y}})|^{n-2}}d\underline{\underline{x}}\wedge d\underline{y}=\frac{c_n}{\sqrt{|\underline{y}|^2+|\underline{\underline{x}}-\underline{\underline{y}}|^2}^{n-2}}d\underline{\underline{x}}\wedge d\underline{y}\]
Ignoring $c_n$ in the numerator, a change of variables $v:=\underline{\underline{x}}-\underline{\underline{y}}$, $w:=\underline{\underline{x}}+\underline{\underline{y}}$  turns the unique coefficient of $\Omega$ into the following function (up to a universal constant):
\[\bR^k\times \bR^{n-k}\times \bR^k,\qquad (v,\underline{y},w)\ra \frac{1}{\sqrt{|v|^2+|\underline{y}|^2}^{n-2}}\]
which  is clearly locally integrable around any point $(0,0,p)\in \bR^k\times \bR^{n-k}\times \bR^k$.

 With respect to the fiber integral we have

\[  \int_{\bR^k}\frac{1}{\sqrt{|\underline{\underline{x}}-\underline{\underline{y}}|^2+|\underline{y}|^2}^{n-2}}d\underline{\underline{x}}=
\int_{\bR^k}\frac{1}{\sqrt{|v|^2+|\underline{y}|^2}^{n-2}}dv=\frac{1}{|\underline{y}|^{n-2}}\int_{\bR^k}{\left(\frac{|v|^2}{|\underline{y}|^2}+1\right)^{-\frac{n-2}{2}}}dv=\]
\[=\frac{1}{|\underline{y}|^{n-k-2}}\int_{\bR^k}\frac{1}{(|u|^2+1)^{(n-2)/2}}du=\frac{|\mathbb{S}^{k-1}|}{|\underline{y}|^{n-k-2}}\int_0^{\infty} \frac{\rho^{k-1}}{(\rho^2+1)^{(n-2)/2}}~d\rho=\]
\[=\frac{|\mathbb{S}^{k-1}|}{|\underline{y}|^{n-k-2}}\int_0^{\frac{\pi}{2}}\cos^{n-k-3}(\theta)\sin^{k-1}(\theta)~d\theta=\frac{|\mathbb{S}^{k-1}|}{|\underline{y}|^{n-k-2}}\frac{\Gamma(\frac{k}{2})\Gamma(\frac{n-k}{2}-1)}{2\Gamma(\frac{n}{2}-1)}\]
where in the last line we used with $z_1=\frac{k}{2}$ and $z_2=\frac{n-k}{2}-1$ the fundamental property of the beta functions
\[B(z_1,z_2)=\frac{\Gamma(z_1)\Gamma(z_2)}{\Gamma(z_1+z_2)}=2\int_{0}^{\frac{\pi}{2}}\sin^{2z_1+1}(\theta)\cos^{2z_2+1}(\theta)~d\theta\]

Finally by using $|\mathbb{S}^{n-1}|=\frac{2\pi^{n/2}}{\Gamma(\frac{n}{2})}$ we see that
\[
\frac{|\mathbb{S}^{k-1}|}{(n-2)|\mathbb{S}^{n-1}|}\frac{\Gamma(\frac{k}{2})\Gamma(\frac{n-k}{2}-1)}{2\Gamma(\frac{n}{2}-1)}=\frac{1}{n-k-2}\frac{1}{|\mathbb{S}^{n-k-1}|}
\]
and therefore (returning the $c_n$):
\[\frac{1}{(n-2)|\mathbb{S}^{n-1}|} \int_{\bR^k}\frac{1}{\sqrt{|\underline{\underline{x}}-\underline{\underline{y}}|^2+|\underline{y}|^2}^{n-2}}d\underline{\underline{x}}=\frac{1}{(n-k-2)|\mathbb{S}^{n-k-1}|}\frac{1}{|\underline{y}|^{n-k-2}}.\]

\end{proof}
\begin{corollary} \label{BSL} Let $n-k\geq 3$ and $L^k$ be a linear subspace  with the same orientation convention as the  previous Proposition. Denote by $r$ the distance function to $L$. Then
\[
\BS(L)=\frac{1}{|\mathbb{S}^{n-k-1}|}\frac{1}{r^{n-k-1}}\iota_{\nabla r}\dvol_{L^{\perp}}=\pi_{L^{\perp}}^*(\BS(\delta_0))
\]
\end{corollary}
\begin{proof} By the same argument as in Proposition \ref{let} it is enough  to prove the case $L=\{0\}\times \bR^{k}$. We use Proposition \ref{let}, the fact that $d^*(\dvol_{L^{\perp}})=0$ and the Leibniz relation for $d^*$ to get the first equality. Computing $\pi_{L^{\perp}}^*(\BS(\delta_0))$ is straightforward.
\end{proof}

\begin{proof} \emph{ of Theorem} \ref{Rn} 
Let $L=\{0\}\times \bR^n$.  We use the isometry $\alpha\in SO(2n)$:
\[\alpha:\bR^n\times \bR^n\ra \bR^n\times \bR^n,\qquad (x,y)\ra \frac{1}{\sqrt{2}}(x+y,y-x)\]
Then  $\alpha (\{0\}\times \bR^{n})=\delta_{\bR^n}$. Note that 
\[\alpha^{-1}(z,u)=\frac{1}{\sqrt{2}}(z-u,z+u)\]
The following commutativity properties are immediate:
\begin{itemize}
\item[(1)] $\alpha_* \circ G= G\circ \alpha_*$
\item[(2)] $d^*\circ G=G\circ d^*$ 
\item[(3)] $d^*\circ \alpha_*=\alpha_*\circ d^*$
\end{itemize}  in order to infer via Corollary \ref{BSL}
\[
\begin{array}{lcl}
\BS(\delta_{\bR^n}) & = & \disp d^*G(\alpha_*L)=\alpha_*d^*G(L)=\alpha_*\pi_1^*(\BS(\delta_0)) \\[0.2cm]
& = & \disp (\alpha^{-1})^*\pi_1^*(\BS(\delta_0))=(\alpha^{-1})_1^*(\BS(\delta_0))
\end{array}
\]
where $(\alpha^{-1})_1(z,u)=2^{-1/2}(z-u)$. By the scaling invariance $\lambda^*\BS(\delta_0)=\BS(\delta_0)$ for each $\lambda \in \R^+$ we get that
\[
(\alpha^{-1})_1^*(\BS(\delta_0))=\beta^*(\BS(\delta_0))
\]
This proves both item (i) and formula (\ref{Rn4}) which is an immediate consequence. Moreover it also proves that $d \BS(\delta_{\bR^n})=\delta_{\bR^n}$, since $0$ is a regular value of $\beta$  and we can use Proposition \ref{P2} and $d\BS(\delta_{0})=\delta_{0}$ which is well-known.

\vspace{0.4cm}
In order to prove item (ii) we see from Proposition \ref{let} that $G(\delta_0)=\frac{1}{(n-2)|\mathbb{S}^{n-1}|}\frac{1}{|x|^{n-2}}dx$. We check that
\begin{equation}\label{eqo1}\sum_{k=0}^n\bG_k=\beta^*(G(\delta_{0}))\end{equation}
by comparing the $(n-k,k)$ pieces of both sides. On the right the sign of $dx_{I^c}\wedge dy_I$ (with $|I|=k$) is 
\[ (-1)^k(-1)^{\epsilon(I^c)}=(-1)^{kn}(-1)^{\epsilon(I)}=\quad \mbox{sign on the left side.}\] It follows from (\ref{eqo1}) that
\begin{equation}\label{eqo2}\sum_{k=0}^nd_y^*\bG_k=\frac{1}{(n-2)|\mathbb{S}^{n-1}|}d_y^*\left(\frac{1}{|x-y|^{n-2}} (dx_1-dy_1)\wedge\ldots \wedge (dx_n-dy_n)\right)\end{equation}
One uses that $d_y^*(f\omega)=-\iota_{\nabla^yf}\omega+fd_y^*\omega$ where $\nabla^yf$ is the projection of $ \nabla f$ onto $\{0\}\times \bR^n$ with $f=\frac{1}{|x-y|^{n-2}}$. Note that since $d_y^*(dx_i)=0=d_y^*(dy_i)$ and
\[\nabla^yf=-(n-2)\sum_{i=1}^n\frac{y_i-x_i}{|x-y|^n}\partial_{y_i}\Rightarrow (dx_i-dy_i)(-\nabla^yf)=(n-2)\frac{x_i-y_i}{|x-y|^n}\]
and so the right hand side of (\ref{eqo2}) equals
\[\frac{1}{|\mathbb{S}^{n-1}|}\frac{1}{|x-y|^{n}}\sum_{i=1}^n(-1)^{i-1}(x_i-y_i)(dx_1-dy_1)\wedge\ldots \wedge\widehat{(dx_i-dy_i)}\wedge \ldots \wedge(dx_n-dy_n)=BS(\delta_{\bR^n})\]
The fact that $d_y^*\bG_{k}$ represents the kernel of $(-1)^nd^*G_k$ is a rather straightforward computation.
\end{proof}

  \section{Initial value problems for the heat equation}\label{S5}

 We will consider the Hodge Laplacian on $k$-forms $\Delta=\Delta_k$ on an oriented Riemannian manifold $M$. For this section our main reference is \cite{BGV}.

 Let $S\in \mathscr{D}'_k(M)$.
  \begin{definition} A solution to an initial value problem (IVP) for the heat equation with initial condition $S$ is  a family of smooth differential forms $\omega_t\in \Omega^{n-k}(M)$ for $t>0$  (alternatively $\omega\in \Omega^{(0,n-k)}(0,\infty)\times M$) such that
  \begin{equation}\label{ivp} \left\{\begin{array}{ccc} \left(\frac{\partial}{\partial t}+\Delta\right)\omega_t&=&0\\
   \disp{\lim_{t\searrow 0}} ~\omega_t=S\end{array}\right.
  \end{equation}
  where $\Delta$ is the Laplacian on $k$-forms and the limit is taken in the sense of currents.
  \end{definition}
  
A construction for the heat kernel of any generalized Laplacian is given in \cite{BGV} Chapter 2. Adapting to the Laplacian on $k$-forms we have the following. For every $t>0$, $p_t^k(x,y)$  is an $(n-k,k)$ form on $M\times M$, i.e. a section of $\Lambda^{n-k}\pi_1^*T^*M\otimes \Lambda^{k}\pi_2^*T^*M$ where $\pi_1,\pi_2:M\times M\ra M$ are the projections onto the factors.

For every $t>0$, one has  the heat operator 
\begin{eqnarray} P_t^k:\Omega^k(M)\ra \Omega^k(M),\qquad P_t^k(\omega):=e^{-t\Delta}(\omega):=(\pi_2)_*(p_t^k\wedge\pi_1^*\omega)\qquad \mbox{i.e.}\\
P_t^k(\omega)=\left\{y\ra \int_{
\pi_2}p_t^k(x,y)\wedge\omega(x)=\int_{M}^{dx}p_t^k(x,y)\wedge \omega(x)\right\}\qquad ~\end{eqnarray}
where  $dx$ has the unorthodox meaning of "integration with respect to the first variable".

A "functional calculus" construction of $p_t$ starts with the spectral decomposition of $\Delta_k$, i.e. with the eigenvalues $\sigma(\Delta_k):=\{\lambda_i~|~i\in \bN\}$ with corresponding eigenvectors which are smooth forms $e_i\in \Omega^k(M)$. Then
\begin{equation}\label{ptk}p_t^k(x,y)=(-1)^{kn}\sum_{i\in \bN}e^{-t\lambda_i}\pi_1^*(*e_i)\wedge \pi_2^*(e_i)\footnote{The sign arises from our pull back push forward convention and they are consistent with Section \ref{rnrn}.}
\end{equation}

This however is of limited use. We need more refined properties that come out from the parametrix construction of the heat kernel.  The reference \cite{BGV}  lists the following axiomatic properties for the heat kernels $p_t^k(x,y)$:
\begin{itemize}
\item[(i)] $p_t^k$ is  $C^1$ in $t$ and $C^2$ in $(x,y)$ for all $t>0$.
\item[(ii)] It holds
\[ (\partial_t+\Delta_y)p_t^k=0
\]
where $\Delta_y$ is the laplacian on $k$-forms in the second variable. 
\item[(iii)] For every smooth $k$-form $\omega$ on $M$ the family of $k$-forms $P_t^k(\omega)$  converges uniformly in the  $C^0$ norm to $\omega$ when $t\ra 0$.
\end{itemize} 
The operator $\Delta_y$ and similarly $\Delta_x$ is defined on $\Omega^{i,j}(M\times M)$ as follows
\[\Delta_y(\pi_1^*\omega\wedge\pi_2^*\eta)=\pi_1^*\omega\wedge\pi_2^*(\Delta\eta))\]
\[\Delta_x(\pi_1^*\omega\wedge\pi_2^*\eta)=\pi_1^*(\Delta \omega)\wedge\pi_2^*\eta\]
This definition is valid only for forms of "product type" in $\Omega^{i,j}(M\times M)$. However, they represent a dense subset in the space of all $(i,j)$-forms with the $C^{\infty}$ topology. Hence one extends these operators by continuity to all of $\Omega^{i,j}(M\times M)$. Fur further use we will also need the similarly defined $d_x$,$d_y$ $d^*_x$  and $d^*_y$
\[d_x(\pi_1^*\omega\wedge\pi_2^*\eta):=\pi_1^*(d\omega)\wedge\pi_2^*\eta \qquad d_x^*(\pi_1^*\omega\wedge\pi_2^*\eta):=\pi_1^*(d^*\omega)\wedge\pi_2^*\eta\]
\begin{equation}\label{dy} d_x(\pi_2^*\eta\wedge \pi_1^*\omega):=\pi_2^*(d\eta)\wedge\pi_1^*\omega \qquad  d_y^*(\pi_2^*\eta\wedge \pi_1^*\omega)=\pi_2^*(d^*\eta)\wedge \pi_1^*\omega.\end{equation}
We have the following fundamental relations
\begin{prop}  The following is true on $M\times M$
\begin{itemize}
\item[(i)] $d=d_x+d_y$
\item[(ii)] $d^*=d^*_x+d_y^*$
\item[(iii)] $\Delta=\Delta_x+\Delta_y$ 
\end{itemize}
As a consequence of (iii), $\Delta$ preserves the bidegree since both $\Delta_x$ and $\Delta_y$ do.
\end{prop}
\begin{proof} All identities are proved on product forms. The first is trivial. The other two need the following straightforward relation between Hodge star $\hat{*}$ on $M\times M$ and $*$ on $M$ 
 
\[\hat{*}\pi_1^*\omega\wedge \pi_2^*\gamma=(-1)^{(n-|\omega|)|\gamma|}\pi_1^*(*\omega)\wedge \pi_2^*(*\gamma)\]

Note that (iii) appears in Corollary 3.56 of \cite{BGV}.
\end{proof}

It turns out that (i)-(iii) can be strengthen to:
\begin{itemize}
\item[(is)] $p_t^k$ is  $C^{\infty}$ on $((0,\infty)\times M\times M)$.

This is a consequence of Lemma 6.1, Ch.III in \cite{LM} for example.
\item[(iis)] 
\[ (\partial_t+\Delta_y)p_t^k=0=(\partial_t+\Delta_x)p_t^k
\]

This follows from the relation $R^*\Delta_x=\Delta_yR^*$  where $R(x,y)=(y,x)$ is the reflection and from
 \begin{equation}\label{Rpt} R^*p_t=(-1)^np_t\end{equation}
where $p_t:=\sum_{k=0}^np_t^k$ is the full heat kernel and (\ref{Rpt}) is an immediate consequence of (\ref{ptk}).
\item[(iiis)]  $P_t^k(\omega)$  converges uniformly to $\omega$ in \emph{every  $C^l$ norm} with $l\geq 0$.

For this, one combines item (2) of Theorem 2.20 in \cite{BGV} and item (2) of Theorem 2.23 in \cite{BGV}. The estimate in the latter coupled with the fact that the operations of wedging with a fixed smooth form and integrating along the fibers of the projections $M\times M\ra M$ are operations which are continuous in the $C^l$ topology proves the uniform convergence for the $C^l$-norm.
\end{itemize} 

\vspace{0.5cm}

If $T\in \mathscr{D}'_{k}(M)$ is an $k$ current then $\pi_2^*T$ is the following $(n,k)$ current on $M\times M$:
\[\pi_2^*T(\phi)=T((\pi_2)_*\phi)=T\left(\int_M^{dx}\phi(x,y) \right)
\]

Because $p_t$ is smooth and $M$ is compact,  the definition of $P_t^k$ extends to an operator $P_t^k: \mathscr{D}'_{k}(M)\ra  \mathscr{D}'_{k}(M)$. 
\[P_t^k(T)(\phi):= T(P_t^k(\phi))\]
This can be written as
\begin{equation}\label{Ptkeq}P_t^k(T)(\phi)=T((\pi_2)_*(p_t^k\wedge\pi_1^*\phi ))=\pi_2^*T( p_t^k\wedge \pi_1^*\phi)=(\pi_1)_*(p_t^k\wedge \pi_2^*T)(\phi)\end{equation}
 where $p_t^k$  is the $(n-k,k)$ kernel of $P_t^k: \Omega^{k}(M)\ra \Omega^{k}(M)$.

\begin{example}
When $T=L$ is a $k$-submanifold, let $\pi_1^L:M\times L\ra M$ be the restriction of $\pi_1$. Then, by Fubini (see (\ref{POm})) for every $\phi\in \Omega^k(M)$:
\[P_t(L)(\phi)=\int_L\int_{\pi_2}p_t^k\wedge \pi_1^*\phi=\int_{M\times L}p_t^k\wedge\pi_1^*\phi=(-1)^{kn}\int_M\left(\int_{\pi_1^L}p_t^k\right)\wedge \phi
\]
Hence $P_t(L)$ is a current represented by the smooth $(n-k)$-form on $M$:
\begin{equation}\label{beq1}{(-1)^{kn}\int_L^{dy}p_t(x,y)=(\pi_1)_*\left(p_t\bigr|_{M\times L}\right)}
\end{equation}
If we consider the dependence in $t$ as well, this is a smooth form on $[0,\infty)\times M\setminus \{0\}\times L$.
\end{example}

\begin{lemma} 
	If $(\xi_{j})_{j\in \mathbb{N}}$ is a sequence of degree $n-k$ forms converging weakly to $T\in\mathscr{D}'_k(M)$ then
\[ P_t^k(T)=\lim_{j\ra \infty}P_t^{n-k}(\xi_j)
\]
also weakly.
\end{lemma}
\begin{proof} The commutativity of the Laplacian with the Hodge star $\Delta_{n-k} *=*\Delta_k$ for smooth forms implies  
\[P_t^{n-k}*=*P_t^k.\] We will also use the fact that $e^{-t\Delta}$ is self-adjoint for a fixed $t$.
\[ \int_{M}P_t^{n-k}(\xi_j)\wedge\omega=(-1)^{k(n-k)}\int_M\langle P_t^{n-k}(\xi_j),*\omega\rangle=(-1)^{k(n-k)}\int_M\langle \xi_j,P_t^{n-k}*\omega\rangle=
\]
\[=(-1)^{k(n-k)}\int_M\langle \xi_j,*P_t^{k}\omega\rangle=\int_M\xi_j\wedge P_t^{k}\omega\overset{j\ra \infty}{\longrightarrow} T(P_t^k(\omega))=P_t^k(T)(\omega).
\]
\end{proof}

  \begin{prop}\label{PtT} Let $T\in\mathscr{D}'_k(M)$. For every $t>0$, the current $P_t^k(T)$ is representable by a smooth form. If $T=L$ is an oriented smooth, compact submanifold then the form is (\ref{beq1}).
\end{prop}
\begin{proof} Let $\iota_x:\{x\}\times M\hookrightarrow M\times M$ be the natural inclusion for a choice of point $x\in M$. Then $p_t\circ\iota_x$ is a smooth  $k$-form on $M$ with values in the vector space $V:=\Lambda^{n-k}T^*_xM$, i.e. a section of $\Lambda^{k}T^*M\otimes V$.

Choose a basis $e_1,\ldots, e_l$ for $V$. Every degree $k$-form $\omega$ with values in $V$ can be written as a 
finite 
\[ \omega=\sum_{i=1}^l \omega_i\otimes e_i\]
for some $k$-forms $\omega_i$. Given $T\in\mathscr{D}'_k(M)$ define 
\[T(\omega):=\sum_i T(\omega_i)e_i\in V\]
It is immediate that this does not depend on the choice of the basis $\{e_1,\ldots e_l\}$. Due to the smoothness of $p_t$ an obvious generalization of  Theorem 2.1.3 from \cite{H1} shows that the correspondence 
\[ x\ra T(p_t\circ\iota_x)
\]
is a \emph{smooth} $(n-k)$-form on $M$, smoothly depending on $t$.

We then check that 
\begin{equation}\label{MTp} (-1)^k\int_M T(p_t\circ \iota_x)\wedge \phi(x)~dx=T\left(\int_Mp_t(\cdot,x)\wedge \phi(x)~dx\right) 
\end{equation}
since this equals $P_t(T)(\phi)$ by (\ref{Ptkeq}). The justification for (\ref{MTp}) is in Proposition \ref{Repfor}.
\end{proof}
  \begin{theorem}\label{exunisol} Let $M$ be a compact, oriented Riemannian manifold. Let $T\in\mathscr{D}'_k(M)$ be a current. Then a smooth solution $\omega_t:= e^{-t\Delta} (T)\in\Omega^{0,n-k}((0,\infty)\times M)$ of the IVP exists and is unique.    
  If $S=L$ is a smooth compact submanifold, the solution is represented by the form (\ref{beq1})  and hence it is smooth on $[0,\infty)\times M\setminus \{0\}\times L$.
    \end{theorem}
  \begin{proof} Following the proof of  Proposition \ref{PtT}, more exactly (\ref{MTp}) let $\omega_t(x):=(-1)^k T(p_t\circ\iota_x)$. By Proposition \ref{Repfor} and the immediate property that $T$ commutes with $\partial_t$ we get:
  \[ (\partial_t+\Delta) T(p_t\circ\iota_{(\cdot)})=T(((\partial_t+\Delta_x) p_t)\circ \iota_{(\cdot)})=0.
  \]
  The last equality holds due to property (iis) of heat kernels. On the other hand, by property (iiis) of $p_t$ we have
  \[ \lim_{t\ra 0}P_t(\phi)=\phi,\qquad \forall \phi
  \]
 in the $C^{\infty}$ topology. By the continuity of $T$ we get
  \begin{equation}\label{beq11}\lim_{t\ra 0}T(P_t(\phi))=T(\phi)\end{equation}
  By  (\ref{MTp}), $T(P_t(\phi))=(-1)^k\int_M T(p_t\circ \iota_{(\cdot)})\wedge \phi$ and relation (\ref{beq11}) says that $(-1)^kT(p_t\circ \iota_{(\cdot)})$ converges weakly to $T$, i.e. the initial condition is fulfilled. 
  
Uniqueness is essentially Lemma 2.16 in \cite{BGV}.  If $\eta_t$ is a family of smooth forms  that converges weakly to $0$ and  such that 
\[(\partial_{\theta}+\Delta)\eta_{\theta}=0\] then $\eta_t\equiv 0$. Let $t$ be fixed. Let $\omega$ be a smooth $n-k$ form. Consider the function
  \[f(\theta)=\int_M\langle\eta_{\theta},P_{t-\theta}^{n-k}(\omega)\rangle
  \]
  for $\theta\in[0,t)$. We have that $f(t)=\lim_{\theta\searrow t}f(\theta)=\int_M\eta_t\wedge \omega$ due to $\lim_{\theta\ra 0} P_{\theta}(\omega)=\omega$ uniformly. On the other hand $f(0)=0$ due to the weak convergence of $\eta_t$ to $0$ when $t\ra 0$. The claim is proved if we can show that $f'(\theta)=0$. We have
  \[ f'(\theta)=\int_{M}\langle \partial_{\theta}\eta_{\theta},P_{t-\theta}^{n-k}(\omega)\rangle-\langle\eta_{\theta},-\Delta(P_{t-\theta}^{n-k}(\omega))\rangle\]
  where we used that $\partial_t \gamma=-\Delta \gamma$ for $\gamma=P_{t}(\omega)$. But using that $\Delta$ is self-adjoint we get
  \[f'(\theta)=\int_M\langle (\partial_{\theta}+\Delta)\eta_{\theta},P_{t-\theta}^{n-k}(\omega)\rangle=0
  \]
  \end{proof} 

\subsection{The Green and the heat kernels}
 We now recall a definition.
  \begin{definition} A current $T\in\mathscr{D}'_k(M)$ is representable by integration if there exists a Radon measure $\mu$ on $M$ and a section $\xi:M\ra \Lambda^kTM$ which is $\mu$-measurable such that $|\xi|=1$  $\mu$-a.e. and 
  \[T(\omega)=\int_M\omega(\xi)~d\mu,\qquad\forall \omega\in \Omega^k(M)
  \]
  \end{definition}

  \begin{example}\begin{itemize} \label{ex12}
  \item[(1)]  A rectifiable current of finite mass is in particular a current representable by integration.
  \item[(2)] Let $\omega$ be a form with $L^1$ integrable coefficients on the compact manifold $M$. Then $T_{\omega}(\eta):=\int_M\omega\wedge \eta$ is a current representable by integration, since it can be written as
  \[ T_{\omega}(\eta)=\int_M\eta((*\omega)^{\sharp})~\dvol_M \]
where $\sharp$ is the musical isomorphism.
  \end{itemize}
  \end{example}

The following result is standard
 \begin{theorem} \label{prevthm}
 The following relation holds irrespective of $\dim{M}$.
   \[G_k(\omega)=\int_0^ {\infty}e^ {-t\Delta}\omega~ dt\]
   for every $\omega\in L^ 2(\Omega^ k(M))$ such that $\omega\in \Imag \Delta.$ 
   
   The integral can be understood as both the operator $f(\Delta)$ applied to $
   \omega$ for \[f(x)=\left\{\begin{array}{cc}\displaystyle \int_0^ {
   \infty}e^ {-tx}~dt & x>0\\
 0& x\leq 0\end{array}\right.\]
\noindent and as the Bochner integral for the family of forms $t\ra e^ {-t\Delta}\omega$ with values in $L^ 2(\Omega^k(M))$.
 
 When $\omega$ is smooth, $G(\omega)$ is smooth.
 \end{theorem}
 \begin{proof} The statements follow from the spectral description of $e^ {-t\Delta}$. See \cite{BGV}, Theorem 2.38 and Proposition 2.37. 
 \end{proof}

  \begin{prop}\label{511} Let $S=\partial T$ be an exact current such that $T$ is representable by integration. Then 
  \[G(S)=\int_0^{\infty} e^{-t\Delta}(S)~ dt
  \]
  where $\left(\int_0^{\infty} e^{-t\Delta}(S)~ dt\right)(\phi):=\int_0^{\infty} e^{-t\Delta}(S)(\phi)~ dt$.
  \end{prop}
 
 \begin{proof} Due to Riesz Representation Theorem, the currents $T\in\mathscr{D}'_{k+1}(M)$ representable by integration  are also those which extend to continuous linear functionals on $L^1(\Omega^k(M))$. 

Let $\phi\in \Omega^k(M)$. By the property (iii) of heat kernels, the family of forms $t\ra \omega_t:= e^{-t\Delta}(d\phi)$ is continuous in the $L^1$ norm all the way to $t=0$. Moreover, since $d\phi\in\Imag \Delta$, we get by Theorem \ref{prevthm} for the Bochner integral of this family 
\[
\int_0^ 
{\infty}\omega_t~dt= \int_0^{\infty}e^{-t\Delta}(d\phi)~dt=G(d
\phi)
\] 
 Now,  using that $[d,e^{-t\Delta}]=0$ we get
\begin{equation}\label{beq3} e^{-t\Delta}(S)(\phi)=S(e^{-t\Delta}\phi)=T(d(e^{-t\Delta}\phi))=T(e^{-t\Delta}(d\phi))\end{equation}
Since $T$ is an $L^1$-continuous functional, a property of Bochner integrals allows us to interchange the order of integration, hence
\[\int_0^{\infty}T(e^{-t\Delta}(d\phi))~dt=T\left(\int_0^{\infty}e^{-t\Delta}(d\phi)\right)=T(G(d\phi))=T(d(G\phi))=\]\[=S(G(\phi))=G(S)(\phi).\]
 
 Putting together (\ref{beq3}) and the last relation concludes the proof. 
  \end{proof}

  We have the following important
  \begin{corollary}\label{GLH} Let $L$ be a closed oriented submanifold of $M$ and $H$ the harmonic representative of $PD(L)$. Then
  \[ G(L-H)=\int_{0}^{\infty}e^{-t\Delta}(L-H)~dt.\]
  \end{corollary}
  \begin{proof} To see that $L-H$ can be written as $\partial T$ with $T$ representable by integration one uses the regularization process (see 4.1.18 in \cite{Fe}) in order to write
  \[ (L-H)-(L-H)_{\epsilon}=dH_{\epsilon}(L-H)\]
  Now $(L-H)_{\epsilon}$ is a smooth exact form being in the same (currential) homology class as $L-H$ hence it is a $d\eta$ for some smooth form $\eta$. On the other hand the finiteness of the mass of $L-H$ implies the finiteness of the mass of  $H_{\epsilon}(L-H)$ as follows from the same section in \cite{Fe}.
  \end{proof}
 
  As a consequence of this we have that (see Definition \ref{defBS})
  \[\BS(L)=d^*\int_0^{\infty}e^{-t\Delta}(L-H)~dt\]
  
In order to prove the extendability of the Biot-Savart form $\BS(L)$ it is fundamental to understand the small time behaviour of the solution to the initial value problem. Small time asymptotics for the heat kernel of functions are known since the work of Minakshisundaram and Pleijel \cite{MP} with further elaboration by Berger, Gauduchon and Mazet \cite{BGM} and in the case of the Hodge-Laplacian for forms by Patodi\cite{Pa}.  In  Section \ref{IVPs}, we adapt the general results of \cite{BGV} to initial value problems.
\subsection{Green kernels and the diagonal}

  \begin{remark} Let us note an interesting consequence of relation (\ref{beq3}). If we let $t\ra \infty$ in that equality, we notice  that $e^ {-t\Delta}(d
  	\phi)$ converges to $0$ in the $L^2$ norm (hence also $L^ 1$ due to the compactness of $M$) since, by functional calculus, $e^{-t\Delta}$ converges strongly (see Theorem VIII.5 (d) in \cite{RS}) to the orthogonal projection onto the space of harmonic forms  and exact forms project to $0$, by Hodge decomposition. Hence
  	\[\lim_{t\ra \infty}e^ {-t\Delta}S(\phi)=\lim_{t\ra \infty}T(e^{-t\Delta}(d\phi))=0\]
  	for every exact current $S$ as in Proposition \ref{511}. In particular, we have that 
  	\[\lim_{t\ra\infty}e^{-t\Delta}(L-H)=0\]
  	But $e^{-t\Delta}H=H$ since $H$ is harmonic. Hence
  	\[\lim_{t\ra\infty}e^{-t\Delta}L=H.\]
  	By the semi-group property, if $L_s:=e^{-s\Delta}L$ then 
  	\[\lim_{t\ra\infty}e^{-t\Delta}L_s=H\]
  	But a similar argument as above shows, since $L_s$ is a closed smooth form, that $\lim_{t\ra \infty}e^{-t\Delta}L_s$ converges to the harmonic representative of $L_s$. It follows that all forms $L_s$ are cohomologous to each other and in the same cohomology class as $H$ and thus Poincar\'e duals to $L$. In other words, the regularization via heat-kernels preserves the cohomology class. Clearly the same argument works for any closed current $L$, not only a closed oriented submanifold. 
  \end{remark}

This has consequences. 
\begin{prop}\label{H+D} Let $T$ be a closed current, $H$ a harmonic form and $\Omega$ a form with locally integrable coefficients such that
\begin{equation}\label{H+} H+\Delta \Omega =T.\end{equation}
Then $H$ is the harmonic representative of the Poincar\'e dual of $T$ and $d\Omega=0$.
\end{prop}
\begin{proof} Suppose first that $T$ and $\Omega$ are smooth. Then write 
\[ d^*d\Omega=T-H-dd^*\Omega\]
We have
\[\|T-H-dd^*\Omega\|^2=\int_M\langle T-H-dd^*\Omega, T-H-dd^*\Omega\rangle~ d\vol_M=\int_M\langle d^*d\Omega,T-H-dd^*\Omega\rangle~d\vol_M=\]\[=\int_M\langle d\Omega, d(T-H-dd^*\Omega) \rangle~d\vol_M=0\]
Hence $T-H-dd^*\Omega=0$ which means that $T$ and $H$ represent the same cohomology class. Also $d^*d\Omega=0$ which immediately implies that $d\Omega=0$.

In general, denote $\Omega_{\epsilon}:=e^{-\epsilon\Delta}\Omega$  and $T_{\epsilon}:=e^{-\epsilon \Delta}T$. Since $e^{-\epsilon \Delta}H=H$, by applying $e^{-\epsilon\Delta}$ to (\ref{H+}) we get
\[H+\Delta\Omega_{\epsilon}=T_{\epsilon}\]
By the first part we get that $T_{\epsilon}$ and $H$ are in the same cohomology class and $d\Omega_{\epsilon}=0$. By the Remark preceding the Proposition we have that $T$ and $T{\epsilon}$ are Poincar\'e duals. Hence $H$ is the harmonic representative of $PD^{-1}([T])$.

 On the other hand, by Lemma \ref{limd} we get $d\Omega=0$. 
\end{proof}
 
We apply this to a concrete situation. Recall the heat kernel (\ref{ptk}). Let $\sigma(\Delta)^*:=\{\lambda_i~|~\lambda_i\neq 0\}$.  The following holds
\begin{lemma}\label{expGk} On an oriented compact manifold $M$, the $(n-k,k)$ form
\[\bG_{k}:=(-1)^{kn}\sum_{\lambda_i\in \sigma(\Delta)^*}\lambda_i^{-1}\pi_1^*(*e_i)\wedge\pi_2^*(e_i)\]
is the  kernel of the Green operator $\bG_k:\Omega^k(M)\ra\Omega^k(M)$.
\end{lemma}
\begin{proof} One checks easily that for every $j$ such that $\lambda_j\in \sigma(\Delta)^*$ the following holds
\[(\pi_2)_*(\bG_k\wedge \pi_1^*e_j)=\lambda_j^{-1}e_j=G_k(e_j)\]
while for $e_j$ that corresponds to $\lambda_j=0$
\[[(\pi_2)_*(\bG_k\wedge \pi_1^*e_j) =0\]
due to the orthogonality of $\{e_i\}$.
\end{proof}
Let now $\delta^{k,n-k}$ be the $(k,n-k)$ part of the diagonal $\delta$ in $M\times M$. Let
 \[\langle\langle \omega,\gamma\rangle\rangle:=\int_M\omega \wedge *\gamma\] denote the inner product of smooth $k$-forms on $M\times M$. The following  is straightforward.
\begin{lemma}\label{dknk} For every smooth $k$-forms $\omega$ and $\gamma$ we have
\[\delta^{k,n-k}(\pi_1^*\omega\wedge\pi_2^*(*\gamma))=\langle \langle \omega,\gamma \rangle\rangle=\sum_{i\in \bN}\langle \langle \omega,e_i \rangle\rangle \langle \langle e_i,\gamma \rangle\rangle=\sum_{i\in \bN} \int_M\langle \omega,e_i\rangle\int_M\langle e_i,\gamma\rangle\]
\end{lemma}

We denote by $f_i^k$, $1\leq i\leq N_k$ an orthonormal basis of harmonic forms of degree $k$ on $M$. Let
\begin{equation}\label{Hk} H^{k,n-k}:=(-1)^{kn}\sum_{i=1}^{N_k}\pi_1^*(*f_k)\wedge \pi_2^*f_k\end{equation}
\begin{theorem}\label{Dlf} The form $H^{k,n-k}$ is harmonic and the following holds on $M\times M$:
\begin{equation}\label{dlf}\Delta\left(\frac{1}{2}\bG_k\right)=\delta^{k,n-k}-H^{k,n-k}\end{equation}
Consequently, $\bG:=\sum_{k=0}^n\bG_k$ is a closed current and $H:=\sum_{k=0}^nH^{k,n-k}$ is the harmonic representative of the Poincar\'e dual of $\delta$ in $M\times M$.
\end{theorem}
\begin{proof} The fact that $H^{k,n-k}$ is harmonic is immediate from $\Delta=\Delta_x+\Delta_y$ on $M\times M$ and $\Delta *=*\Delta$ on $M$.

We prove equality (\ref{dlf}). Since $\Delta=\Delta_x+\Delta_y$ we get for a form $\omega$  of degree $k$ and $\eta$ of degree $n-k$:
 \begin{equation}\label{Tbgk} T_{\bG_k}(\Delta (\pi_1^*\omega\wedge\pi_2^*\eta)=\int_{M\times M}\bG_k\wedge \pi_1^*(\Delta \omega)\wedge \pi_2^*\eta+\int_{M\times M}\bG_k\wedge \pi_1^*(\omega)\wedge \pi_2^*(\Delta\eta)\end{equation}
We deal with the first integral to begin with.
\[\int_{M\times M}\bG_k\wedge \pi_1^*(\Delta \omega)\wedge \pi_2^*\eta=\sum_{\lambda_i\neq 0}\lambda_i^{-1}\int_{M\times M}\pi_1^*(\Delta \omega\wedge *e_i)\wedge \pi_2^*(e_i\wedge \eta)=\]\[=\sum_{\lambda_i\neq 0}\lambda_i^{-1}\int_M\left(\int_{\pi_2}\pi_1^*(\Delta \omega\wedge *e_i)\right)\cdot e_i\wedge \eta\]
Now $\pi_1^*(\Delta \omega\wedge *e_i)$ is the same form when restricted to the fibers of $\pi_2$ hence in every fiber one gets
\[\int_{\pi_2}\pi_1^*(\Delta \omega\wedge *e_i)=\int_M\Delta \omega\wedge *e_i=\langle\langle  \Delta \omega,e_i\rangle\rangle=\langle\langle  \omega,\Delta e_i\rangle\rangle=\lambda_i\langle\langle \omega,e_i\rangle\rangle\]
and so
\[\int_{M\times M}\bG_k\wedge \pi_1^*(\Delta \omega)\wedge \pi_2^*\eta=\sum_{\lambda_i\neq 0}\langle\langle \omega,e_i\rangle\rangle\int_{M\times M}e_i\wedge\eta =(-1)^{k(n-k)}\sum_{\lambda_i\neq 0}\langle\langle \omega,e_i\rangle\rangle \langle\langle e_i,*\eta\rangle\rangle\]
Note however that according to  Lemma \ref{dknk} we have
\[(-1)^{k(n-k)}\sum_{i\in \bN}\langle\langle \omega,e_i\rangle\rangle \langle\langle e_i,*\eta\rangle\rangle=(-1)^{k(n-k)} \langle\langle \omega,*\eta\rangle\rangle=(-1)^{k(n-k)}\delta^{k,n-k}(\pi_1^*\omega\wedge \pi_2^*(**\eta))=\]
\[=\delta^{k,n-k}(\pi_1^*\omega\wedge \pi_2^*\eta)\]
Note that
\[\delta^{k,n-k}(\pi_1^*\omega\wedge \pi_2^*\eta)-T_{\bG_k}(\pi_1^*\omega\wedge \pi_2^*\eta)=(-1)^{k(n-k)}\sum_{\lambda_i=0}\langle\langle \omega,e_i\rangle\rangle \langle\langle e_i,*\eta\rangle\rangle=T_{H^{k,n-k}}(\pi_1^*\omega\wedge\pi_2^*\eta) \]
the last equality being a simple check. We have thus proved
\[\Delta_x \bG_k=\delta^{k,n-k}-H^{k,n-k}\]
Analogously one shows that $\Delta_y \bG_k=\delta^{k,n-k}-H^{k,n-k}$ and thus using (\ref{Tbgk}) one concludes (\ref{dlf}).

It is obvious that (\ref{dlf}) implies that 
\[\Delta\left(\frac{1}{2}\bG\right)=\delta-H\]
and the last conclusions follow from Proposition \ref{H+D}.
\end{proof}

We list now some symmetry relations for the Green kernels. Let $R:M\times M\ra M\times M$ be the reflection in the diagonal $R(x,y)=(y,x)$. 
\begin{lemma} The  following hold
\begin{itemize}
\item[(i)] $R^*\bG_k=(-1)^n\bG_{n-k}$
\item[(ii)] $*\bG_k=(-1)^{n-k}\bG_{n-k}$
\item[(iii)] $d_x\bG_k=-d_y\bG_{k-1}$
\item[(iv)] $d_x^*\bG_{k}=d_y^*\bG_{k+1}$
\end{itemize}

\end{lemma}
\begin{proof} The first two items are immediate consequences of Lemma \ref{expGk}. Item (iii) is a consequence of the fact that $\bG$ is a closed current, according to Theorem \ref{Dlf}. Item (iv) needs a proof.

We define 
 \begin{eqnarray} *_x(\pi_1^*\omega\wedge\pi_2^*\eta):=(\pi_1^*(*\omega)\wedge\pi_2^*\eta) \\ *_y(\pi_2^*\eta\wedge\pi_1^*\omega):=\pi_2^*(*\eta)\wedge\pi_1^*(\omega)\end{eqnarray}
so that $d_y^*=(-1)^{nj+n+1}*_yd_y*_y$ on forms of bidegree $(i,j)$.
We will use the following immediate commutativity relations
\[*_x*_y=(-1)^n*_y*_x,\quad d_x*_y=(-1)^n*_yd_x,\quad d_y*_x=(-1)^n*_xd_y.\]
We also need that Hodge star $*$ on forms of bidegree $(i,j)$ works as follows
\[*=(-1)^{(n-i)(n-j)}*_y*_x=(-1)^{ij+ni+nj}*_x*_y\]
from which we deduce that
\[*_x=(-1)^{ij+ni+j}**_y=(-1)^{(n-i-j)(n-j)}*_y*.\]
We make use of item (ii) and (iii) and the previous relations in the next lines:
\[d_x^*\bG_k=(-1)^{n(n-k)+n+1}*_xd_x(*_x\bG_k)=(-1)^{n(n-k)+n+1}*_xd_x*_y(*\bG_k)=(-1)^{(n+1)(k-1)}*_xd_x*_y\bG_{n-k}=\]
\[=(-1)^{(n+1)(k-1)}*_y*_xd_x\bG_{n-k}=(-1)^{(n+1)(k-1)+1}*_y*_xd_y\bG_{n-k-1}=(-1)^{nk+k}*_yd_y(*_x\bG_{n-k-1})=\]
\[=(-1)^{nk+k}*_yd_y*_y(*\bG_{n-k-1})=(-1)^{nk+1}*_yd_y*_y\bG_{k+1}=d_y^*\bG_{k+1}\]

\end{proof}
We get immediately from item (iv) and $d^*=d_x^*+d_y^*$ the next statement.
\begin{corollary}\label{Dlf2} The following holds on $M\times M$:
\[d^*\left(\frac{1}{2}\bG\right)=d_y^*\bG\]
\end{corollary}

Putting together Theorem \ref{Dlf} and Corollary \ref{Dlf2} we get the analogous result of Theorem \ref{Rn} (ii) for a compact manifold $M$.
\begin{theorem}\label{Gradiag} Let $M$ be a compact, oriented Riemannian manifold.  Then the Biot-Savart form $\BS(\delta)$ of the diagonal $\delta$ in $M\times M$ is the form $d_y^*\bG$ where $\bG$ is the total Green kernel and $(-1)^nd_y^*\bG$ is the kernel of the total Biot-Savart operator
 \[d^*G:\Omega^*(M)\ra \Omega^{*-1}(M).\] 
\end{theorem}
 
\begin{remark} This result together with the fact that $\BS(\delta)$ is blow-up extendible when $M$  has dimension $3$ (Theorem \ref{codim3} below) implies that Vogel's differential linking form \cite{Vo} for real cohomology spheres of dimension $3$ coincides with ours, modulo the different sign conventions.
\end{remark}

   \section{Small time asymptotics for IVPs} \label{IVPs}

  In this section we discuss the construction of a parametrix in the sense of Hadamard for an initial value problem. The scheme follows in its initial steps the construction of the parametrix for the heat kernel, in other words one starts by looking for a formal solution of the heat equation  which leads to a recurrence relation that takes the form of an ODE.
  \begin{definition} A parametrix for the initial value problem (\ref{ivp}) is any smooth form  $\Psi\in \Omega^{0,n-k}((0,\infty)\times M)$ such that 
  \begin{itemize}
  \item[(i)] $H\Psi : = (\partial_t+\Delta)\Psi$ extends to a continuous form on $[0,\infty)\times M$. 
  \item[(ii)] $\disp\lim_{t\ra 0}\Psi_t=S$ weakly.
  \end{itemize}
  \end{definition}
  
  We follow the construction of a parametrix for the heat kernel in order to construct a parametrix for an IVP when $S=L$ is a compact oriented $k$-dimensional submanifold, $k \le n-1$. Let $T$ be a tubular neighbourhood of $L$ where the normal exponential map 
\[
\exp^\perp \ \ : \ \ D_{\eps}(\nu L) \doteq \big\{ v \in \nu L : |v|< \eps\big\} \to T, \qquad v \mapsto \exp_{\pi(v)}(v)
\]
is a diffeomorphism. Via $\exp^\perp$, we identify $T$ with $D_{\eps}(\nu L)$ endowed with the pulled-back metric $g$, and $L$ with the zero-section  ${\bf 0}$. Then, the distance function writes on $D_{\eps}(\nu L)$ as 
\[
r(v) = |v|.
\]
Notice that $r$ is a $1$-Lipschitz function and $g(\nabla r, \gamma') = 1$ along the integral curves $\gamma(t) = tv/|v|$ of the unit vector field $\partial_r$, thus $\nabla r = \partial_r$. 

For later use, it will be useful to define another metric on $D_{\eps}(\nu L)$ (in fact, on the entire $\nu L$), that will be compared  to $g$ in due time: the Sasaki metric $\bar g$. We recall its construction in local coordinates. Denote by $\pi : \nu L \to L$ the bundle projection, let $\{x^i\}$, $1 \le i \le k$ be a local system of coordinates on $U \subset L$ and write $g = g_{ij} dx^i \otimes dx^j$ for the induced metric on $L$. Up to shrinking $U$, we can assume that $U$ supports a local orthonormal basis of sections $\{\nu_\alpha\}$ for $\nu L$. Let $\nabla^\perp$ be the normal connection, and write
\[
\nabla^\perp_{e_j} \nu_\beta = \Gamma^\alpha_{\beta j} \nu_\alpha, \qquad \Gamma^\alpha_{\beta j} \in C^\infty(U).
\]
Defining $y^i = x^i \circ \pi$ and $y^\alpha(v) = g(v, \nu_\alpha)$, then in the chart $\{y^i,y^\alpha\}$ on $\pi^{-1}(U)$ the metric $\bar g$ writes as
\[
\bar g = \bar g_{ij} dy^i \otimes dy^j + \delta_{\alpha\beta} (dy^\alpha + y^\gamma \bar{\Gamma}^\alpha_{\gamma i} dy^i) \otimes (dy^\beta + y^\delta \bar{\Gamma}^\beta_{\delta j} dy^j).
\]
where $\bar g_{ij} = g_{ij}\circ \pi$, $\bar{\Gamma}^\alpha_{\gamma i} = \Gamma^\alpha_{\gamma i} \circ \pi$. Up to renaming, we can assume that the form $dy^1 \wedge \ldots \wedge dy^n$ respects the natural orientation on $\nu L$ making $\exp^\perp$ orientation preserving.

\begin{lemma}\label{lem_prop_Sasaki}
The Sasaki metric $\bar g$ on $\nu L$ has the following properties: 
\begin{itemize}
	\item[(i)] The volume form $\bar\omega$ has the local expression
\[
\bar\omega = \sqrt{\det[\bar g_{ij}]} \ dy^1 \wedge \ldots \wedge dy^n.
\]
	\item[(ii)] The vector field $r \partial_r$ satisfies 
\begin{equation}\label{eq_diver_sasaki}
	\overline{\diver}(r \partial_r) = n-k \qquad \text{on } \, \nu L.
\end{equation}
\end{itemize}
\end{lemma}

\begin{proof}
Extracting an orthonormal coframe $\theta^i$ out of the forms $\{dx^i\}$ on $U \subset L$, and letting $\gamma^\alpha_{ij}$ the components of $\nabla^\perp$ in the basis dual to $\theta^i$, by construction of the Sasaki metric the frame
\[
\bar \theta^i = \pi^* \theta^i, \qquad \bar \theta^\alpha = dy^\alpha + y^\beta \bar \gamma^\alpha_{\beta i} dy^i
\]
is orthonormal for $\bar g$, where $\bar \gamma^\alpha_{\beta i} = \gamma^\alpha_{\beta i} \circ \pi$. The volume form is therefore $\bar \theta^1 \wedge \ldots \wedge \bar \theta^n$. Since $\bar \theta^1 \wedge \ldots \wedge \bar\theta^k = \sqrt{\det(\bar g_{ij})} dy^1 \wedge \ldots dy^k$ and 
\[
\bar\theta^{k+1} \wedge \ldots \wedge \bar \theta^n = dy^{k+1} \wedge \ldots \wedge dy^n + \ \text{ terms with some } \, dy^i,
\]
the identity in (i) follows. To prove (ii), let $\Phi_t$ be the flow of $r \partial_r$, that is, $\Phi_t(v) = v e^t$. Then, $\Phi_t^* \bar{\omega} = e^{(n-k)t} \bar \omega$ and therefore,
differentiating and recalling that 
\[
\overline{\diver}(r \partial_r)\bar\omega = L_{r \partial_r} \bar\omega = \left.\frac{d}{dt}\right|_{t=0} \Phi_t^* \bar{\omega}
\]
we deduce \eqref{eq_diver_sasaki}.	
\end{proof}

We begin to construct our parametrix. We will use the following result:

\begin{prop}\label{uspr} If $\Delta_0 = - \diver(\nabla)$ is the Laplacian on functions, then \begin{equation}\label{weitz}
		\Delta(f\eta)=\Delta_0 f\cdot \eta+f\Delta \eta-2\nabla_{\nabla f}\eta\end{equation}
where $\nabla$ is the Levi-Civita connection.
\end{prop}
\begin{proof} This is standard. Probably, the fastest way to see it is to use Weitzenb\"ock formula (see (3.16) in \cite{BGV})
\[\Delta=\nabla^*\nabla+R\]
The identity (\ref{weitz}) is immediate for the Bochner Laplacian $\nabla^*\nabla$ on forms (see also Proposition 2.5 \cite{BGV}).
\end{proof}

Let 
\[ 
f_t(r):= \frac{1}{(4\pi t)^a}e^{-\frac{r^2}{4t}}, \qquad t,r>0, \qquad \text{where } \, a= \frac{n-k}{2}.
\]
For $N \in \mathbb{N}$ define 
\[
\begin{array}{c}
\disp \eta_t:=\eta_0+\eta_1t+\ldots+\eta_Nt^N,  \qquad \eta_i \in \Omega^{n-k}(M), \\[0.4cm]
\Psi_t := f_t(r)\eta_t
\end{array}
\]
and denote with $H:=\partial_t +\Delta$ the heat operator. We compute using Proposition \ref{uspr}
\[
H \Psi = (\partial_t f_t(r))\eta_t+f_t(r)(\partial_t\eta_t)+\Delta_0[f_t(r)]\cdot\eta_t-2 \nabla_{\nabla f_t(r)} \eta_t +f_t(r)\Delta\eta_t.
\]

We denote with $'$ the derivative in the variable $r$, so that
\[
\nabla f_t(r)=f'_t(r)\nabla r
\]
We compute
\[ 
\Delta_0[f_t(r)]=-\diver(\nabla f_t(r))=f'_t(r)\Delta_0r-f''_t(r),
\]
where we use that $|\nabla r|^2=1$. Moreover, away from $r=0$ we have 
\begin{equation}\label{eq2} 
f_t'(r)=-\frac{r}{2t}f_t(r)\qquad f_t''(r)=\frac{r^2-2t}{4t^2}f_t(r)\qquad \partial_t f_t(r) =\left(-\frac{a}{t}+\frac{r^2}{4t^2}\right)f_t(r).
\end{equation}

Therefore  by using (\ref{eq2}):
\begin{equation}\label{eq_Hpsi}
\begin{array}{lcl}
\disp H \Psi & =& \disp f_t(r)\left[\left(-\frac{a}{t}+\frac{r^2}{4t^2}-\frac{r\Delta_0r}{2t}-\frac{r^2-2t}{4t^2}\right)\eta_t+\frac{r}{t} \nabla_{\nabla r}\eta_t +(\partial_t+\Delta)\eta_t\right] \\[0.5cm]
& = & \disp f_t(r)\left[\frac{1}{2t}\left(1-2a-r\Delta_0r\right)\eta_t+\frac{r}{t}\nabla_{\nabla r} \eta_t +H\eta_t\right].
\end{array}
\end{equation}
We have
\[
H\eta_t=t^N\Delta\eta_N+\sum_{i=0}^{N-1}(\Delta\eta_i+(i+1)\eta_{i+1})t^i,
\]
so that, setting $\eta_{{-1}}:=0$, we can write 
\begin{equation}\label{eq_Hpsi_2}
\disp H \Psi = f_t(r)t^N \Delta \eta_N + f_t(r) \sum_{i=-1}^{N-1} B_i t^i,
\end{equation}
where\footnote{one encounters in Peter Li the same expression bar the sign of the Laplacian for the function Laplacian}
\[
B_i : = \frac{1-2a-r\Delta_0r}{2}\eta_{i+1}+\nabla_{_{r\nabla r}}\eta_{i+1}+\Delta\eta_i+(i+1)\eta_{i+1}
\]
It is useful to use the function
\[
\xi=\frac{1}{2}r^2
\] 
Then
\[
\nabla \xi =r\nabla r = r \partial_r, \qquad \Delta_0 \xi =-\diver(r\nabla r)=r\Delta_0 r-1,
\]
and we rewrite $B_i$ as
\begin{equation}\label{eq30}
\begin{array}{rcl}
B_i &=& \disp \left(-a+(i+1)-\frac{\Delta_0 \xi}{2}\right)\eta_{i+1}+\nabla_{_{\nabla \xi}}\eta_{i+1}+\Delta\eta_i \qquad \text{for } \, 0 \le i \le N-1 \\[0.5cm]
B_{-1} &=& \disp \left(-a-\frac{\Delta_0 \xi}{2}\right)\eta_0+\nabla_{_{\nabla \xi}}\eta_0 
\end{array}
\end{equation}
We shall choose $\eta_j$ smoothly on $D_{\eps}(\nu L)$ in such a way that $B_i = 0$ for $-1 \le i \le N-1$. To this aim, we shall study the function
\[
-a-\frac{\Delta_0 \xi}{2} = \frac{k-n - \Delta_0 \xi}{2}.
\]
\begin{lemma}\label{lem_chiave}
The following holds 
\begin{equation}\label{eq_delta0xi}
\frac{k-n - \Delta_0 \xi}{2} = \nabla \xi(\log \sqrt{j}).
\end{equation}
where $j \in C^\infty(D_{\eps}(\nu L))$ is the determinant of the exponential map $\exp^\perp : ( D_{\eps}(\nu L),\bar g) \to (T,\metric)$ and $\bar g$ is the Sasaki metric. In particular, $j \equiv 1$ on the zero-section .
\end{lemma}
\begin{proof}
Denote by $\omega$ and $\bar \omega$, respectively, the volume forms of $g = (\exp^\perp)^*\metric$ and of $\bar g$. Then, $\omega = j \bar \omega$. By Lemma \ref{lem_prop_Sasaki}, we get 
\[
\begin{array}{lcl}
(-\Delta_0 \xi)\omega & = & \disp \diver(r\partial_r) \omega = L_{r \partial_r}\omega \\[0.2cm]
& = & \disp L_{r \partial_r}(j \bar\omega) = r(\partial_r j) \bar\omega + \overline{\diver}(r \partial_r)\bar\omega \\[0.2cm]
& = & \disp L_{r \partial_r}(j \bar\omega) = r(\partial_r j) \bar\omega + j(n-k)\bar \omega \\[0.2cm]
& = & \disp \frac{\nabla \xi(j)}{j} \omega + (n-k)\omega = \left(n-k + \nabla \xi(\log j)\right)\omega,
\end{array}
\]
and rearranging we get \eqref{eq_delta0xi}.
\end{proof}
In view of the above Lemma, system \eqref{eq30} rewrites as
\begin{equation}\label{eq30_rearranged}
\begin{array}{lcl}
B_{-1} &=& \disp \nabla \xi(\log \sqrt{j})\eta_0+\nabla_{_{\nabla \xi}}\eta_0 = j^{-\frac{1}{2}} \nabla_{\nabla \xi}\left( j^{\frac{1}{2}}\eta_0 \right) \\[0.5cm]
B_i &=& \disp \left(\nabla \xi(\log \sqrt{j}) + (i+1)\right)\eta_{i+1}+\nabla_{_{\nabla \xi}}\eta_{i+1}+\Delta\eta_i \\[0.4cm]
 &=& \disp r^{-i-1}j^{-\frac{1}{2}} \left\{ \nabla_{\nabla \xi}\left(r^{i+1}j^{\frac{1}{2}} \eta_{i+1}\right) + r^{i+1}j^{\frac{1}{2}}\Delta\eta_i\right\} \qquad \text{for } \, 0 \le i \le N-1.
\end{array}
\end{equation}
Again since $\nabla \xi, = r \partial_r$, finding $\{\eta_i\}$ so that $B_i = 0$ for $-1 \le i \le N-1$ amounts to solve
\begin{equation}\label{eq_iteration}
\begin{array}{ll}
(i) & \disp \nabla_{\partial_r}\left( j^{\frac{1}{2}}\eta_0 \right) = 0, \\[0.4cm]
(ii) & \disp \nabla_{\partial_r}\left(r^{i+1}j^{\frac{1}{2}} \eta_{i+1}\right) = - r^{i}j^{\frac{1}{2}}\Delta\eta_i \qquad \text{for } \, 0 \le i \le N-1 
\end{array}
\end{equation}
Note that the equations (ii) correspond to singular linear ODE if the zero-section  is taken into account, os opposed to initial value problems.

Notice that $j^{1/2}$ is smooth on the entire $D_{\eps}(\nu L)$, with value $1$ along ${\bf 0}$. The system can be solved as follows: first, we choose the initial data
\begin{equation}\label{eq_eta0_L}
\eta_0 = dy^{k+1} \wedge \ldots \wedge dy^n \qquad \text{on } \, {\bf 0}.
\end{equation}
Let $i : L \to D_{\eps}(\nu L)$ be the zero-section , fix a local basis $\{\psi^I\}$ of sections of $i^* \Lambda^{n-k}(T^*D_{\eps}(\nu L))$ on $U \subset L$, and extend them by parallel translation along normal geodesics to get a basis for $\Omega^{n-k}(\pi^{-1}(U))$, still called $\{\psi^I\}$. Smooth dependence on initial data guarantee that $\psi^I$ is smooth on $D_{\eps}(\nu L)$. Writing 
\[
\eta_i = A_{i,I} \psi^I, \quad \Delta \eta_i = B_{i,I}\psi^I, \qquad A_{i,I}, B_{i,I} \in C^\infty(\pi^{-1}(U)), 
\]
identities $(i)$ and $(ii)$ become 
\[
\partial_r \big(j^{\frac{1}{2}}A_{0,I}\big) = 0, \qquad \partial_r\left(r^{i+1}j^{\frac{1}{2}} A_{i+1,I}\right) = - r^{i}j^{\frac{1}{2}}B_{i,I} \qquad \text{for each } \, I. 
\]
The first has solution 
\[
A_{0,I}(v_p) = A_{0,I}(0_p) j^{\frac{1}{2}}(0_p)j^{-\frac{1}{2}}(v_p) \in C^\infty(\pi^{-1}(U)),
\]
the second can be solved inductively by defining 
\[
A_{i+1,I} = - r^{-i-1} j^{-\frac{1}{2}} \int_0^r t^{i}j^{\frac{1}{2}}B_{i,I} dt,
\]
namely, 
\[
\begin{array}{lcl}
A_{i+1,I}(v) & = & \disp \disp - |v|^{-i-1} j(v)^{-\frac{1}{2}} \int_0^{|v|} t^{i}j\left(\frac{tv}{|v|}\right)^{\frac{1}{2}}B_{i,I}\left(\frac{tv}{|v|}\right) dt \\[0.4cm]
& = & \disp - j(v)^{-\frac{1}{2}} \int_0^{1} s^i j(sv)^{\frac{1}{2}}B_{i,I}(sv) ds,
\end{array}
\]
hence $A_{i+1,I}$ is smooth on the entire $D_{\eps}(\nu L)$ once so is $B_{i,I}$. System \eqref{eq_iteration} therefore admits a smooth solution $\eta_0, \ldots, \eta_{N}$.

 Note that the solution $\Psi$ is of class $C^{\infty}$ on the open subset $[0,\infty)\times T_{\epsilon}\setminus \{0\}\times L$ since this is where the function $f_t(r)$ is $C^{\infty}$. 

 Equation \eqref{eq_Hpsi_2} becomes
\[
H \Psi = f_t(r) t^N \Delta \eta_N = (4\pi)^{\frac{k-N}{2}} t^{N - \frac{n-k}{2}} e^{-\frac{r^2}{4t}} \Delta \eta_N.
\]
Choose $\zeta \in C^\infty_c([0,\eps))$ such that $\zeta \equiv 1$ on $[0, \eps/2]$, and consider the function $\hat \Psi = \zeta(r) \Psi$. Then, using Proposition \ref{uspr} again we get
\begin{equation}\label{eqn123}
H \hat \Psi =\zeta(r)(4\pi)^{\frac{k-N}{2}} t^{N - \frac{n-k}{2}} e^{-\frac{r^2}{4t}} \Delta \eta_N-   2\zeta'(r) \nabla_{\nabla r}\Psi + (-\zeta''(r) + \zeta'(r)\Delta_0 r)\Psi  ,
\end{equation}
In particular, since $\zeta' = \zeta'' = 0$ for $r \le \eps/2$ we get that for $N > \frac{n-k}{2}$ condition (i) in the definition of parametrix holds for $\hat{\Psi}$ with $(H\hat{\Psi})_0=0$.

 Regarding condition (ii) we prove it now.

\begin{prop}
The following holds
\[
\hat\Psi_t \to L \qquad \text{as $t \to 0$, in the sense of currents.}
\]
\end{prop}
\begin{proof}
By construction, $\hat\Psi_t \to 0$  as $t\ra 0$ uniformly on compacts contained  in  $M\setminus L$. Hence it is enough to check $\hat\Psi_t \to L$ by integrating test functions supported in $T$. For $\delta >0$, define $T_{\delta}:=\{v\in T~|~ |v|\leq \delta\}$.
By the definition of Poincar\'e dual, $\hat{\Psi}_t \to L$ if and only if  
\[
\lim_{t\ra \infty}\int_{T} \hat\Psi_t \wedge \omega \to \int_{L} \iota^*\omega,\qquad \forall \omega\in \Omega_c^k(T)
\]
For $\epsilon$ small enough
\[\hat{\Psi}=\Psi=f_t(r)(\eta_0+t\eta_1+\ldots+t^N\eta_N)\]
We will show that
\[ \lim_{t\ra 0}\int_{T} f_t(r) \eta_0\wedge \omega=\int_{L}\iota^*\omega,\qquad \lim_{t\ra 0}t\int_{T} f_t(r) \eta_i\wedge \omega=0,~~ \forall i>0
\]

We work within $\nu L$ and use the change of variables 
\[
\varphi_t \ : \ T_{\frac{\epsilon}{2\sqrt{t}}}\ra T, \qquad \varphi_t(p,v)=(p,2\sqrt{t}v)
\]
Then 
\begin{equation}\label{Ldt}\int_{T} f_t(r) \eta_0\wedge \omega=\frac{1}{\pi^{(n-k)/2}}\int_{T_{\frac{\epsilon}{2\sqrt{t}}}} e^{-r^2}(\eta_0\wedge \omega)_{(p,2\sqrt{t}v)}\end{equation}
The last formula makes sense for the following reason. We can use a linear Ehresmann connection on $\nu L$ to identify 
\[
T\nu L\simeq \pi^*\nu L\oplus \pi^*TL
\]
With this splitting a few good things are obvious. One is that we can identify the tangent spaces at points $(p, v)$ and $(p,2\sqrt{t}v)$. Another one is that the differential of the rescaling map is diagonal with respect to this splitting being the direct sum of multiplication by $2\sqrt{t}$ in the $\nu L$ directions with the identity morphism in the $TL$ directions. Hence
\[\varphi_t^*(\eta_0\wedge \omega)=2^{n-k}t^{(n-k)/2}(\eta_0\wedge \omega)_{(p,2\sqrt{t}v)}\] 

We have a $C^{\infty}$ uniform convergence on compacts of $\nu L$ when $t \ra 0$:
\begin{equation}\label{Ldt2}(\eta_0\wedge \omega)_{(p,2\sqrt{t}v)}\ra  \{(p,v)\ra \eta_0\wedge \omega_{(p,0)}\}=\pi^*( \eta_0\wedge \omega\bigr|_{[0]}).\end{equation}
where $\pi:\nu L\ra L$ is the projection and the symbol $\bigr|_{[0]}$ means restriction to the zero-section  and not pull-back.

Since  $\eta_0\bigr|_{[0]}=\vol_{\nu L/L}$ we therefore get from (\ref{Ldt2}) that 
\[(\eta_0\wedge \omega)_{(p,2\sqrt{t}v)}\ra \vol_{\nu L/L}\wedge \iota^*\omega\]
where $\iota:L\ra \nu L$ is the zero-section  inclusion.

Hence by Lebesgue dominated convergence we get from (\ref{Ldt})
\[\lim_{t\ra 0}\frac{1}{\pi^{(n-k)/2}}\int_{T_{\epsilon/2\sqrt{t}}} e^{-r^2}(\eta_0\wedge \omega)_{(p,2\sqrt{t}v)}=\frac{1}{\pi^{(n-k)/2}}\int_{\nu L} e^{-r^2}\pi^*(\vol_{\nu L/L}\wedge \iota^*\omega)= \]
\[=\int_{L}\left(\frac{1}{\pi^{(n-k/2)}}\int_{\nu L/L}e^{-r^2}\vol_{\nu L/L}\right)\wedge \iota^*\omega=\int_L\iota^*\omega\]
where we used in the last line the well-known fact that
\[\frac{1}{\pi^{(n-k)/2}}\int_{\bR^{n-k}}e^{-|v|^2}~dv=1\]

The same procedure shows that
\[
\lim_{t\ra 0}t\int_{T} f_t(r) \eta_i\wedge \omega=0\qquad \forall i>0.
\]

\end{proof}

\section{The parametrix and the genuine solution}

Let $\omega:=\omega_t$ be a family of smooth forms of degree $n-k$ which is continuous down to $t=0$ in the uniform norm and let $p^{n-k}_t$ be the heat kernel for the Laplacian on $n-k$ forms. We know that $p_t^{n-k}$ is an $(k,n-k)$-form on $M\times M$. 
\begin{lemma}\label{flc7} The following convolution operation is well-defined
\[ 
(p*\omega)_t:=\int_0^t(\pi_2)_*(p_{t-\theta}^{n-k}\wedge \pi_1^*\omega_{\theta})~d\theta=\int_0^tP_{t-\theta}^{n-k}(\omega_{\theta})~d\theta=\int_0^te^{-(t-\theta)\Delta}(\omega_{\theta})~d\theta
\] 
and produces a  continuous family $t\ra (p* \omega)_t$ of degree $n-k$ forms which is $0$ when $t=0$.
\end{lemma}
\begin{proof} 
The family of forms 
\[ (t,\theta)\ra P_{t-\theta}^{n-k}(\omega_{\theta})\]
is continuous in the uniform norm on $\{(t,\theta)\in \bR^2~|~ 0 \le \theta\leq t\}$ with the continuity along the diagonal being a consequence of property (iii) of heat kernels.

\end{proof}

We also have the following Calculus result.
\begin{lemma}\label{diflem} Let $(t,s)\ra F(t,s)$ be a continuous function of two variables on the set $\{(t,s)\in \bR^2~|~t\geq s\geq 0\}$ and $C^1$ in the first variable. 
Then
\[\frac{\partial}{\partial t}\int_0^tF_t(s)ds\biggr|_{t=t_0}=\int_0^{t_0} \frac{\partial F}{\partial t}(t_0,s)~ ds+F(t_0,t_0)
\]\end{lemma}
\begin{proof}   Let $G(t,x):=\int_0^x F(t,s)~ds$. One has to compute the derivative of
 \[t
 \ra G(t,t)
 \] 
 at $t=t_0$. The Fundamental Theorem of Calculus on one hand, and the commutativity of the integral (with fixed ends) in one variable and the partial derivative in another variable, on the other hand, do the job. 
\end{proof}
In our case, the function we want to consider is 
\[ (t,\theta)\ra P_{t-\theta}(\omega_{\theta})\]
where now $\theta\ra \omega_{\theta}$ is a family of smooth forms which is continuous in the $C^2$ norm on $M$ all the way down to $\theta=0$. We have for $\theta<t$:
\[
\partial_t\left(P_{t-\theta}(\omega_{\theta})\right)=\partial_t(e^{-(t-\theta)\Delta}\omega_{\theta})=-e^{-(t-\theta)\Delta}(\Delta\omega_\theta)=-P_{t-\theta}(\Delta\omega_{\theta})\]
and this is continuous along the diagonal $\theta=t$ in the uniform norm. As a consequence this makes the function 
\[(t,\theta)\ra P_{t-\theta}(\omega_{\theta})\]
  $C^1$ in the first variable. We can therefore use Lemma \ref{diflem} in order to prove 
  \begin{lemma}[Duhamel principle]\label{lem_duhamel}
  	 Let $t\ra \omega_t$ be a family of degree $n-k$ smooth forms continuous (in $t$) in the $C^2$ topology. Let $H=\partial_t+\Delta$ be the heat operator on forms of degree $n-k$. Then
  \[
  [H(p*\omega)]_t=\omega_t
  \]
\end{lemma}
\begin{proof}  Partial derivatives in the spatial direction commute with integration in the time variable hence we have:
\begin{equation}\label{DelH} \Delta \int_0^tP_{t-\theta}(\omega_{\theta})~d\theta=\int_0^t\Delta P_{t-\theta}(\omega_{\theta})~d\theta.
\end{equation}
Therefore
\[ [H(p*\omega)]_t=H\int_0^tP_{t-\theta}(\omega_{\theta})~d\theta=\frac{\partial}{\partial t}\int_0^tP_{t-\theta}(\omega_{\theta})~ d\theta+\int_0^t\Delta P_{t-\theta}(\omega_{\theta})~ d\theta=\]\[=\int_0^tH(e^{-(t-\theta)\Delta}\omega_{\theta})~ d\theta+ P_0(\omega_t)=\omega_t\] since 
\[
(\partial_t +\Delta)(e^{-(t-\theta)\Delta}(\omega_{\theta}))=-\Delta e^{-(t-\theta)\Delta}\omega_{\theta}+\Delta e^{-(t-\theta)\Delta}\omega_{\theta}=0.
\]

\end{proof}
Let $N>\frac{n-k}{2}$ and let $\Psi_t^N$ be a formal solution of the initial value problem as constructed in the previous section, where $N$  is the order where we cut the formal series. 

Let $\zeta$ be a cut-off function on $M$ which is $1$ when $r<\epsilon/2$  and $0$ when $r\geq\epsilon$, for some small $\epsilon$  and let  
\[\Psi=\Psi_t:=\zeta\Psi_t^N=\zeta(r) f_t(r)\eta_t^N.\] 
\begin{lemma} \label{lemconvH} Let  $l\geq 0$ be an integer such that $l < N-(n-k)/2$. Then $t\ra H(\Psi)_t$ is a family of smooth forms which is continuous in the $C^l$ topology all the way down to $t=0$. 
\end{lemma}
\begin{proof} We have already seen that $t\ra H(\Psi)_t$ is a parametrix, hence uniformly continuous at $t=0$. Recall (\ref{eqn123}) which reads with the new notation:
\[H  \Psi =\zeta(4\pi)^{\frac{k-N}{2}} t^{N - \frac{n-k}{2}} e^{-\frac{r^2}{4t}} \Delta \eta_N-   2\zeta'\left(\frac{\partial f_t}{\partial r} \eta_t^N+f_t\nabla_{\nabla r}\eta_t^N\right) + (-\zeta'' + \zeta'\Delta_0 r)f_t\eta_t^N \]

The claim follows from an estimate
\[ \| H(\Psi)_t\|_l\leq O (t^{N-\frac{n-k}{2}-l
})
\] 
which is clear on compacts at a distance bigger than $\epsilon/2$ than $L$ since then
\[e^{-\frac{r^2}{4t}}|Q|<e^{-\frac{\epsilon^2}{16t}}|Q|<C(t^{N-\frac{n-k}{2}-{l}})\]
for any Laurent polynomial $Q$ in $t$ when $t\ra 0$.   Close to $L$, one has $H\Psi= f_t(r) \Delta \eta_N$. Then  for any vector field $X$ one has
\begin{equation}\label{hkeqn1} X(e^{-r^2/t})=-e^{-r^2/t}\frac{r}{t}2X(r)=O(t^{-1})\end{equation}
 This accounts for the drop in regularity. 
\end{proof}

\begin{theorem}\label{teo_exiHeat} 
	For $N>\frac{n-k}{2} +2$, the family of forms  $L_t:=\Psi-p*H(\Psi)$ is the  solution to the IVP
\[ 
\left\{\begin{array}{ccc}
H(L_t)&=&0\\
\disp\lim_{t\ra 0}L_t&=&L
\end{array}
\right.
\]
\end{theorem}
\begin{proof} 
	By Lemma \ref{lemconvH}, $ t \mapsto H(\Psi)_t$ in continuous in $C^2$ up to $t=0$. We can therefore apply the Duhamel principle in Lemma \ref{lem_duhamel} to deduce
\[ H(\Psi-p*H(\Psi))=H(\Psi)-H(p*H(\Psi))=H(\Psi)- H(\Psi)=0
\]
Be Lemma \ref{flc7},  $t\ra (p* H(\Psi))_t$ is a continuous family of forms which is $0$ at $t=0$, while any formal solution $\Psi_t$ has the property that
\[ \lim_{t\ra 0}\Psi_t=L.\]
This concludes the proof.
\end{proof}

The difference between the formal solution and the genuine solution of the IVP satisfies the following.
\begin{theorem}\label{fundest5} For $l \tcr{<}  N-\frac{n-k}{2}$, the following holds in the  $C^l$ topology:
\[ \lim_{t\ra 0}\partial_t^j(p* H(\Psi)) =0,\qquad \forall 0\leq j< N-\frac{n-k}{2}-l.
\]
\end{theorem}
\begin{proof} By property (iii) of heat kernels  (see Section \ref{S5}), for every smooth form $\eta$
\begin{equation}\label{clnorm} \lim_{t\searrow 0} P_t(\eta)=\eta\end{equation} in the $C^l$ norm.

For $j=0$, by Lemma \ref{lemconvH}  we know that $t\ra H(\Psi)_t$ is continuous in the $C^l$ norm. A straightforward adaptation of Lemma \ref{flc7} proves that in fact $p*H(\Psi)$ is continuous in the $C^l$ norm with $[p*H(\Psi)]_0=0$.

In order to prove the result for $j>0$ one notes that:
\[\|[\partial_t^j H(\Psi)]_t\|_l=O(t^{N-\frac{n-k}{2}-l-j}). \]
The fact that the time derivatives of $H(\Psi)$ diminish the regularity only by $t^{-1}$ and not $t^{-2}$ is based on the fact that $y\ra e^{-y^2}y$ is a bounded function where we take $y=\frac{r}{\sqrt{t}}$. 

This implies  that $[\partial_t^j H(\Psi)]_0\equiv 0$ and by Lemma \ref{diflem} that 
\[\partial_t^j(p*H)=p*\partial_t^j H\]
Another application of Lemma \ref{flc7} in the $C^l$ topology finishes the proof.

 \end{proof}
 
 We finish with an application of what was done in this section.
 \begin{lemma}\label{Lem2}
 	Let $S$ be an exact current, representable by integration on $M$. Then 
 	\[\int_1^{\infty}e^ {-t\Delta}S~ dt\]
 	is a smooth form.
 \end{lemma}
 \begin{proof}
 	Let $S_1:=e^{-\Delta}S$. By Theorem \ref{exunisol}, $S_1$ is a smooth form on $M$. Due to the commutativity of $d$ and $e^ {-t\Delta}$ for every $t$ we get that $S_1$ is an exact form (see e.g. \ref{beq3}). We then have by the semi-group property of  the heat kernel that
 	\[\int_1^{\infty}e^ {-t\Delta}S~dt=\int_1^ {\infty}e^{-(t-1)\Delta}S_1~ dt=\int_0^{\infty}e^{-s\Delta}S_1~ ds=G(S_1)\]
 	and the latter is a smooth form.
 \end{proof}
  
  \begin{theorem}\label{lt1}
  	Let $L$ be smooth closed submanifold of $M$ and $H$ be the harmonic representative of its Poincar\'e dual. Let $\hat{L}_t:=e^{-t\Delta}(L-H)=L_t-H$. Then the current $G(L-H)$ is represented by the following form on $M$ with locally integrable coefficients which is smooth on $M\setminus L$:
  	\[\int_0^{\infty}L_t~ dt.\]
  \end{theorem}
  \begin{proof} We already know by Corollary \ref{GLH} that 
  	\[
  	G(L-H)(\phi)=\int_0^{\infty} \hat{L}_t(\phi) dt 
  	\]
thus using Lemma \ref{Lem2} it is enough to prove that
  	 \[
  	 \int_0^ 1\hat{L}_t~ dt
  	 \]
  	 has locally integrable coefficients. Pick a parametrix $\Psi = \zeta \Psi^N$ with $N> \frac{n-k}{2} + 2$. By Theorem \ref{teo_exiHeat} and since the heat kernel acts as the identity on the space of harmonic forms, we have
  	 \[
  	 \hat{L}_t=(e^ {-t\Delta}L)-H = \Psi - p * H(\Psi) - H.
  	 \]
  	 By Theorem \ref{fundest5}, we need only look at the extendibility of $\Psi$, i.e. of forms of the type
  	 \begin{equation}\label{neweq}
  	 \int_0^{1} e^{-\frac{r^2}{4t}}t^{-j/2}\eta ~ dt
  	 \end{equation}
  	 with $\eta$ a smooth form on $M$ and $j \le n-k$ an integer.
%
%
Lemma \ref{Hfunc0} below proves that \eqref{neweq} is a $C^1$-smooth form on $M$ for $j \le 0$, it is continuous on $M$ for $j=1$, it is of type $O(r^{2-j})$ close to $L$ for $j\geq 3$ while it is of type $O(\ln(r))$ for $j=2$. 
In any case, since $j\leq n-k$ one easily shows in a tubular neighborhood that both $r^{2-j}$ and $\ln(r)$ are locally integrable functions at all points along $L$.
\end{proof}



We conclude this section with a couple of technical propositions quoted in the above proof, which will be useful in the next section. For $j \in \R$ define
\[
H_j \ : \ [0,\infty) \to 
[0,\infty), \qquad H_j(r) := \int_0^1e^{-\frac{r^2}{4t}}t^{-j/2}~ dt
\]
	\begin{lemma} \label{Hfunc0} The following holds: 	
		\begin{itemize}
			\item[(i)] if $j \le 0$, then $H_j \in C^1([0,\infty))$ and $H_j'(0)=0$;
			\item[(ii)] if $j=1$, 
			\[H_1(r)=2e^{-\frac{r^2}{4}}-r\Theta(\frac{r^2}{4})\]
			where $\Theta(x)=\int_x^{\infty}e^{-u} u^{-1/2}~du$, in particular $\Theta(0)=\Gamma(1/2)=\sqrt{\pi}$. Note that $r\ra \Theta(r^2)$ is a smooth function of $r$ on $[0,\infty)$;
			\item [(iii)] if $j=2$, 
			\[
			H_2(r)=-\ln\left(\frac{r^2}{4}\right)e^{-\frac{r^2}{4}}+\Upsilon(\frac{r^2}{4})
			\]
			where $\Upsilon(x)=\int_x^{\infty}\ln(u)e^{-u}~du$. Note that $r\ra \Upsilon(r^2)$ is a $C^1$ function of $r$ on $[0,\infty)$.
			\item[(iv)] For $j\geq 3$ an integer,
			\[H_j(r)=2^{j-2}r^{2-j}F_j(\frac{r^2}{4})\]
			where $F_j(x)$ is a primitive of $-e^{-x}x^{j/2-2}$ such that  $F_j(0)=\Gamma(j/2-1)$. Note that $r\ra F_j(r^2)$ is smooth on $[0,\infty)$. 
		\end{itemize}
	\end{lemma}
	\begin{proof} 
		(i). It is enough to consider $r$ in a bounded interval of $[0,\infty)$. Writing $g(r,t)=e^{-\frac{r^2}{4t}}t^{-j/2}$, we have $g(r,t)\leq 1$ for $t\in(0,1]$ and for any $0<\beta<1$
		\begin{equation}\label{neweq1}
			\left|\frac{\partial g}{\partial r}\right|=\frac{1}{2}e^{-\frac{r^2}{4t}}rt^{-\frac{j}{2}-1}=e^{-\left(\frac{r}{2\sqrt{t}}\right)^2}\left(\frac{r}{2\sqrt{t}}\right)^{\beta}\left(\frac{r}{2}\right)^{1-\beta}t^{- \frac{j}{2}+\beta/2-1}\leq C_{\beta}r^{1-\beta}t^{-\frac{j}{2}+\beta/2-1}\end{equation}
		where we used the fact that $y\ra e^{-y^2}y^{\beta}$ is a bounded function on $[0,\infty)$.
		 Since $\beta>0$, $t^{-\frac{j}{2}+\beta/2-1}$ is  integrable and thus, by Theorem 2.27 in \cite{Fo}, we can differentiate under the integral sign:
		\[
		H_j'(r)=-\frac{r}{2}\int_0^1e^{-\frac{r^2}{4t}}t^{-\frac{j}{2}-1}~dt
		\]
		By using again (\ref{neweq1}) one gets that $\lim_{r\ra 0}H_j'(r)=0$.\\
		(ii) The change of variables $u=\frac{r^2}{4t}$ gives
		\[H_j(r)=2^{j-2}r^{2-j}\int_{\frac{r^2}{4}}^{\infty}e^{-u}u^{j/2-2}~du\]
		and we immediately get (iv) with $F_j(x)=\int_x^{\infty}e^{-u}u^{j/2-2}~du$. In order to get (iii) and (iv) one integrates by parts.
	\end{proof}
	%
	The following Corollary is of interest for the next section. 
	\begin{corollary}\label{corHj} Let $r$ be the distance function to the submanifold $L$ and suppose that $j$ is a positive integer. Then for points where $r>0$ one has
		\begin{equation}\nabla [H_j(r)]=r^{1-j}\tilde{H}_j(r)\nabla r\end{equation}
		where $\tilde{H}_j(r)$ is smooth on $[0,\infty)$ if $j \neq 2$, while it is $C^1$ on $[0,\infty)$ if $j=2$.
	\end{corollary}

\section{Extendibility of the Biot-Savart form}

We are interested in the regularity properties of the smooth form on $M\setminus L$:
\[
\BS(L):=d^*G(L-H)
\]

Following Theorem \ref{lt1}, we write
\[BS(L)=d^ *\int_0^ {\infty}e^{-t\Delta}(L-H)~dt=d^ *\int_0^ 1e^{-t\Delta}(L)~ dt+d^ *\int_1^{\infty}e^ {-t\Delta}(L-H)~dt.\]
 Since the second term in this sum is smooth on $M$ by Lemma \ref{Lem2}, we restrict attention to
  \[d^ *\int_0^ 1e^{-t\Delta}(L)~ dt=d^*\int_{[0,1]}dt\wedge L_t.\]
If $\Psi_t$ is a parametrix chosen for $N$ sufficiently large then, by Lemma \ref{lemconvH}, the difference $ \Psi_t-L_t= [p*H(\Psi)]_t$ satisfies 
\[\int_0^1 [p*H(\Psi)]_t~dt\in C^1(M)\]
and therefore it is enough to prove that $d^*\int_0^1\Psi_t~ dt$ is blow-up extendible. Of course, we can just restrict to a tubular neighborhood $U$ of $L$ where 
\[\Psi_t=f_t(r)\left(\sum_{i=0}^{N}\eta_it^i\right)\qquad f_t(r)=\frac{1}{(4\pi  t)^{(n-k)/2}}e^{-\frac{r^2}{4t}}\]
for some smooth forms $\eta_i \in \Omega^{n-k}(M)$,
so we need to look at forms of type
\[
d^*\left[\left(\int_0^1f_t(r)t^i~dt\right)\eta_i\right]
\]
By Lemma \ref{Hfunc0}, if the power of the monomial in $t$ inside the integral $f_i(r):=\int_0^1f_t(r)t^i~dt$ is non-negative then $f_i(r)$ is a $C^1$ function of $r$ including at $r=0$ and this is enough to conclude the extendibility to the blow-up of 
\[
d^*(f_i(r)\eta_i)=-f_i'(r)\iota_{\nabla r }\eta_i+f_i(r)d^*\eta_i
\]
since $\nabla r$ extends smoothly to the blow-up, as it is the pushforward of a vector field in a neighbourhood of $\partial \Bl_L(M)$. Thus, we need to restrict attention to terms in the Laurent polynomial $\Psi_t$ with a negative power of $t$. We start with an easy case:
\begin{theorem} 
	Let $n\geq 3$. If $k=0$ then the form $\BS({\pt})$ is blow-up extendible at the point $L=\pt\in M$. 
\end{theorem}
\begin{proof}
 We argue that all forms $d^*(H_j(r)\eta_{(n-j)/2})$ are extendible where $j \le n-k$ 
 and $H_j$ are defined in Lemma \ref{Hfunc0}. This is enough because the other terms that make up  the sum of $d^*\int_0^1\Psi_t~ dt$ are obviously extendible. Now 
\[\eta_i=f_i\vol_M\]
and according to Corollary \ref{corHj}
\begin{equation}\label{eqHj} \Omega_j:=d^*(H_j(r)\eta_{(n-j)/2})=-r^{1-j}\tilde{H}_j(r)f_{\frac{n-j}{2}}\iota_{\nabla r}(\vol_M)+H_j(r)d^*(f_{\frac{n-j}{2}}\vol_M)\end{equation}
We argue that both terms of the sum in (\ref{eqHj}) are extendible. Since all the forms $\Omega_j$ have degree $n-1$ the criterion of Proposition \ref{prwexB} (see also Proposition \ref{prwexR})   says that it is enough that the families parametrized by $r$ of forms on a geodesic sphere around $\pt$:
 \begin{equation} \label{creq1} v\ra r^{n-1}\Omega_j(rv)\quad\mbox{ and }\quad r^{n-2}\iota_{\nabla r}\Omega_j(rv)\end{equation}  be uniformly continuous at $r=0$.

Now, multiplying (\ref{eqHj}) with $r^{n-1}$, the first term of the sum is 
\[r^{n-1}r^{1-j}\tilde{H}_j(r)f_{\frac{n-j}{2}}\iota_{\nabla r}(\vol_M)=r^{n-j}\tilde{H}_j(r)f_{\frac{n-j}{2}}\cdot \iota_{\nabla r}(\vol_M)\]
which is clearly uniformly continuous for all $j\leq n$ on a geodesic sphere owing to the fact $\iota_{\nabla r}(\vol_M)$ is the volume form of the geodesic sphere. On the other, since $r^{n-1}H_j(r)$ extends and $d^*(f_{\frac{n-j}{2}}\vol_M)$ is already continuous, we conclude that the first piece of (\ref{creq1}) is uniformly continuous. 

Note now that contraction with $\nabla r$ of the first piece in the sum of (\ref{eqHj}) is $0$. 

As for the second term in (\ref{eqHj}) note that Lemma \ref{Hfunc0} gives that, for $n\geq 3$, $r^{n-2}H_j(r)$ extends continuously at $r=0$ when $j\leq n$ even when $j=2$. Clearly $\iota_{\nabla r}(d^*(f_{(n-j)/2}\vol_M))$ is extendible. This finishes the proof.
\end{proof}

\begin{theorem} If $n-k=1$ the form $\BS(L)$ is blow-up extendible along $L$. 
\end{theorem}
\begin{proof} In this case, $\BS(L)$ is a function.  We need to look at $d^*(H_1(r)\eta_0)$. By Corollary \ref{corHj}:
\[ \Omega:=d^*(H_1(r)\eta_0)=-\tilde{H}_1(r)\iota_{\nabla r}\eta_0+H_1(r)d^*\eta_0\]
where $\tilde{H}_1$ depends smoothly on $r$.  Since $d^*\eta_0$ is smooth everywhere, by Lemma \ref{Hfunc0} this  extends.
\end{proof}
\begin{remark}\label{lr} According to Appendix \ref{apB}, the analysis of extendibility in higher codimension has to take into account the decomposition of $\eta_0$ into its components of various vertical degrees. The vertical-horizontal decomposition is induced by a choice of a \emph{radial trivialization} (see Definition \ref{fibtr}) of the tangent bundle $TU$. A radial trivialization is a bundle isomorphism
\begin{equation}\label{TUpiL} TU\simeq \pi^* \nu L\oplus \pi^*TL\end{equation}
where $\pi:U\ra L$ is the closest point projection with certain good properties with respect to the gradient vector field $\nabla r$  (hence "radial"). Due to (\ref{TUpiL}), a vector field on $U$ gets  decomposed into a vertical (i.e. from $\pi^*\nu L$) and a horizontal component (i.e. from $\pi^*TL$). Similarly, a $k$-differential form  on $U$ decomposes into a sum of sections of bundles like $\Lambda^i\pi^*\nu ^*L\otimes \Lambda^{k-i}\pi^*T^*L$, $i=0,\ldots,k$. If all the components with $i<k$ vanish we call the form on $U$ \emph{purely vertical}. 

There are many choices of such splittings.  Appendix \ref{apB} is built around this kind of structure and familiarity with it might be a good idea to proceed. 

According to Section \ref{IVPs} the form $j^{1/2}\eta_0$ where $j$ is a smooth function is obtained by parallel transporting the normal volume form along normal geodesics. This suggests we use the splitting on $TU$ obtained by parallel transporting with the Levi-Civita connection along normal (i.e. radial) geodesics the following decomposition 
\[TU\biggr|_{L}=TL\oplus \nu L\]
valid at $L$. Clearly, with this choice of splitting, $j^{1/2}\eta_0$ and so $\eta_0$ are  purely vertical forms. 

Fortunately, this splitting is radial (see Example \ref{Partrc}). For the next results we use the criteria of blow-up extendibility as described in Theorem \ref{thmainA} and Corollary \ref{maincorA}.  

\end{remark}

  \begin{theorem}\label{codim3} If $n-k=3$, the form $\BS(L)$ is blow-up extendible along $L$.
\end{theorem}
  
\begin{proof} By Lemma \ref{Hfunc0} we only need to look at the terms in the asymptotic approximation of $BS(L)$ that correspond to negative powers of $t$. Using the notation of Lemma \ref{Hfunc0} we therefore need to analyse the extendibility of
\begin{equation}\label{eqMT} d^*(H_3(r)\eta_0)+d^*(H_1(r)\eta_1).\end{equation}

By Lemma \ref{Hfunc0}  (a) (see also Corollary \ref{corHj}) we conclude that the second term in (\ref{eqMT})
\[d^*(H_1(r)\eta_1)=-\tilde{H}_1(r)\iota_{\nabla r}\eta_1+H_1(r)d^*\eta_1\] is blow-up extendible.

For the first term in (\ref{eqMT}) we will ignore $F_3(\frac{r^2}{4})$ which is a smooth function in Lemma \ref{Hfunc0} (c) and look only at
\begin{equation}\label{eqsum1} d^*(r^{-1}\eta_0)=-\iota_{\nabla( r^{-1})}(\eta_0)+r^{-1}d^*\eta_0\end{equation}

We will show that both terms in the sum (\ref{eqsum1}) are blow-up extendible. 

By Remark \ref{lr}, $\eta_0$ is purely vertical. 
 Since $\nabla(r^{-1})$ is vertical we get that $\iota_{\nabla (r^{-1})}(\eta_0)$ is a vertical form of degree two. We use Corollary \ref{maincorA}, case $(i,j,a)=(2,1,2)$ for 
\[\omega=-\iota_{\nabla( r^{-1})}(\eta_0)=r^{-2}\iota_{\nabla r}(\eta_0).\]  in order to conclude that $\iota_{\nabla(r^{-1})}\eta_0$ is blow-up extendible.

We turn to the second term of (\ref{eqsum1}). We decompose
\[d^*\eta_0=\gamma_0+\gamma_1+\gamma_2\]
into the components of vertical degree $0$, $1$ and $2$. We use again Corollary \ref{maincorA} for each $r^{-1}\gamma_i$, $i=0,1,2$.
\begin{itemize}
 \item[(i)]   If $\gamma_0(0)=0$, i.e. $\gamma_0\equiv 0$ along $L$ then $r^{-1}\gamma_0$ extends along $L$ which implies the blow-up extendibility of $r^{-1}\gamma_0$ (Corollary \ref{maincorA} case $(i,j,a)=(0,0,1)$).
 \item[(ii)]  If $\gamma_1(0)=0$, i.e. $\gamma_1\equiv 0$ along $L$ then $r^{-1}\gamma_1$ extends continuously along $L$. This implies the blow-up extendibility of $r^{-1}\gamma_1$ (Corollary \ref{maincorA} case $(i,j,a)=(1,0,1)$).
 \item[(iii)] The blow-up extendibility of $r^{-1}\gamma_2$ is automatic by Corollary \ref{maincorA} case $(i,j,a)=(2,0,1)$ which gives $a<i+j$.
 \end{itemize} 
 
 In conclusion, if we can prove that the following equality holds
 \[d^*\eta_0(0)=(d^*\eta_0)_{[2]}(0)\] 
 i.e. that $d^*\eta_0$ restricted to $L$ is a purely vertical form then we have guaranteed the extendibility. 
This is the object of Proposition \ref{below} which finishes the proof.
\end{proof}

The next statement is valid for general $k$ and $n$.  Denote by 
\[H:L\ra \nu^* L,\qquad H(p)=\sum_{i=n-k+1}^{n}P^{\nu_pL}(\nabla_{e_i}e_i)\]
the mean curvature of $L$ where $e_{n-k+1},\ldots,e_{n}$ is an orthonormal basis of $T_pL$ and $\nabla$ is the Levi-Civita connection of $M$. Recall also the volume density $j$ of Lemma \ref{lem_chiave}.
 \begin{prop}\label{below} The form $d^*\eta_0$ when restricted to $L$ is a purely vertical form, i.e. the components of vertical degree $i$ with $0\leq i\leq n-k-1$  vanish along $L$,  more precisely
\begin{equation}\label{fjfi2} d^*\eta_0\bigr|_{L}=\iota_X\vol_{\nu L}\end{equation}
with $X=\frac{1}{2}\nabla j+H$.
 \end{prop}
 \begin{proof} Let $\hat{\eta}_0:=j^{1/2}\eta_0$. By Lemma \ref{lem_chiave}, $\hat{\eta}_0$ is parallel along normal geodesics. On the other hand, $j$ is smooth and constante qual to $1$ along $L$, hence the gradient satisfies $\nabla j\biggr|_{L}\in \nu L$, i.e.  $\nabla j$ is vertical along the zero-section .  From 
\begin{equation} \label{fjfi1} d^*\eta_0=-\iota_{\nabla (j^{-1/2})}\hat{\eta}_0+j^{-1/2}d^*\hat{\eta}_0=\frac{1}{2}j^{-3/2}\iota_{\nabla j}\hat{\eta}_0+j^{-1/2}d^*\hat{\eta}_0\end{equation}
we conclude that to prove the statement we only need to prove that $d^*\hat{\eta}_0$ is purely vertical along $L$. 

We fix a point $p\in L$ and use the following
 \begin{equation}\label{fjfi0} d^*\hat{\eta}_0=-\sum_{i=1}^n\iota_{e_i}(\nabla_{e_i}\eta_0)\end{equation}
 for an orthonormal basis $\{e_i\}$ of $T_pM$ at the point $p\in L$.
 
 We choose the basis such that $e_1,\ldots, e_{n-k}$ span $\nu_p L$. Let $f_i:=e_i$, $1\leq i\leq n-k$.  Extend $f_i$ to  $\tilde{f}_i$ via parallel transport along normal geodesics. Then 
 \begin{equation}\label{fjfi} (\nabla_{f_j}\tilde{f}_i)=0 \qquad \mbox{ at } p.
 \end{equation}
  We know from Section \ref{IVPs} that $\hat{\eta}_0=\tilde{f}_1^*\wedge \ldots\wedge \tilde{f}_{n-k}^*$ where $*$ means here dual basis. It follows from (\ref{fjfi}) and Leibniz rule that
 \begin{equation} \nabla_{f_j}\hat{\eta}_0=0  \qquad \mbox{ at } p\label{naf}
\end{equation}
 
 On the other hand for $i>n-k$
 \[\iota_{e_i}(\nabla_{e_i}\hat{\eta}_0)=\iota_{e_i}((\nabla_{e_i}\tilde{f}_1^*)\wedge \ldots \wedge f_{n-k}^*)+\ldots+\iota_{e_i}({f_1^*}\wedge \ldots^*\wedge (\nabla_{e_i}\tilde{f}_{n-k}^*))\]
 and,  since $\iota_{e_i}(f_j^*)=0$ we get for $i>n-k$
 \begin{equation}\label{naf1}\iota_{e_i}(\nabla_{e_i}\hat{\eta}_0)=\sum_{a=1}^{n-k} (-1)^{a-1}[(\nabla_{e_i} \tilde{f}_a^*)(e_i)]f_{\{a\}^c}^*=-\sum_{a=1}^{n-k} (-1)^{a-1}\langle f_a, \nabla_{e_i}e_i\rangle f_{\{a\}^c}^*
 \end{equation}
where $f_{\{a\}^c}^*=f_1^*\wedge\ldots \wedge\widehat{f_a^*}\wedge \ldots \wedge f_{n-k}^*$.

Then (\ref{naf}) and (\ref{naf1}) say that $\iota_{e_i}(\nabla_{e_i}\hat{\eta}_0)$ is purely vertical at any $p\in L$ for every $i=1,\ldots, n$ and by combining with (\ref{fjfi0})  we see that
\[ (d^*\hat{\eta}_0)\bigr|_{L}=\iota_H\vol_{\nu L}\]
Then relation (\ref{fjfi2}) follows from (\ref{fjfi1}).
 \end{proof}
There are other good uses of Proposition \ref{below} when combined with the following computation, which holds in any codimension.
\begin{prop}\label{below1} The following holds along $L$:
\[\nabla j\bigr|_{L}=-H\]
\end{prop}
\begin{proof} Let $D_{\eps}(\nu L),T,\nu_\alpha, g,\bar g, \nu_\alpha$ and coordinates $\{y^i,y^\alpha\}$ be defined as at the beginning of Section \ref{IVPs}. In particular, $L$ is identified with the image of the zero-section $L \to \nu L$, which we name ${\bf 0}$. By construction, 
\[
\exp^\perp_* \left( \frac{\partial}{\partial y^\alpha} \right) = \nu_\alpha \qquad \text{along } \, {\bf 0}.
\]
From $j \equiv 1$ on ${\bf 0}$ we deduce that $\nabla j$ is perpendicular to ${\bf 0}$. Therefore, along ${\bf 0}$ one has
\[
\nabla j = \disp g\left( \frac{\partial}{\partial y^\alpha}, \nabla j \right) \frac{\partial}{\partial y^\alpha} = \frac{\partial j}{\partial y^\alpha} \frac{\partial}{\partial y^\alpha}.  
\]
If $g_{ab}$ are the components of $g$ in coordinates $\{y^i,y^\alpha\}$, by Lemma \ref{lem_prop_Sasaki} the volume densities of of $g$ and $\bar g$ are, respectively,  
\[
\sqrt{\det(g_{ab})} \qquad \text{and} \qquad \sqrt{\det(g_{ij} \circ \pi)}.
\] 
It follows that
\[
\log j = \log \sqrt{\det(g_{ab})} -\log \sqrt{\det(g_{ij} \circ \pi)}. 
\]
Hence, taking into account that the second term does not depend on $y^\alpha$,
\[
\begin{array}{lcl}
\disp \frac{\partial \log j}{\partial y^\alpha} & = & \disp \frac{\partial}{\partial y^\alpha} \log \sqrt{\det(g_{ab})} = \frac{1}{2} g^{ab} \frac{\partial g_{ab}}{\partial y^\alpha}
\end{array}	
\]
(this last identity is not restricted to ${\bf 0}$). However,  on  ${\bf 0}$ we have  $g^{i\alpha} = 0$, $g^{\alpha\beta} = \delta^{\alpha \beta}$, 
hence 
\begin{equation}\label{eq_log j}
\disp \frac{\partial \log j}{\partial y^\alpha} =\frac{1}{2}\left( g^{ij} \frac{\partial g_{ij}}{\partial y^\alpha} + \sum_\beta \frac{\partial g_{\beta \beta}}{\partial y^\alpha}\right).
\end{equation}
Metric compatibility and symmetry of the Levi-Civita connection give
\[
\frac{\partial g_{ij}}{\partial y^\alpha} = -2 A^\alpha_{ij},
\qquad \frac{\partial g_{\beta \beta}}{\partial y^\alpha} = 2  g\left( \nabla_{\frac{\partial}{\partial y^\alpha}} \frac{\partial}{\partial y^\beta}, \frac{\partial}{\partial y^\beta} \right) = 2 \Gamma_{\beta \alpha}^\delta g_{\delta \beta},
\]
where $A = A^{\alpha}_{ij} dy^i \otimes dy^j \otimes \frac{\partial }{\partial y^\alpha}$ is the second fundamental form of ${\bf 0} \hookrightarrow (D_{\eps}(\nu L),g)$ and $\Gamma^a_{bc}$ are the Christoffel symbols of $\nabla$. Since, for each fixed $p \in L$ and $w = w^\alpha \nu_\alpha \in T_p L$, the curve $\gamma : (-\eps,\eps) \to D_{\eps}(\nu L)$ with components $\gamma^a = y^a \circ \gamma$ satisfying $\gamma^j = x^j(p)$, $\gamma^\alpha = w^\alpha t$ is a geodesic, the geodesic equation gives
\[
0 = \nabla_{\dot\gamma}{\dot\gamma} = w^\alpha w^\beta \Gamma_{\alpha \beta}^c \frac{\partial}{\partial y^c} \qquad \text{at } \, \gamma(t).
\] 
Evaluating at $t=0$, by the arbitrariness of $w$ and $p$ we deduce $\Gamma_{\alpha \beta}^c = 0$ along ${\bf 0}$. Inserting into \eqref{eq_log j} we obtain 
\[
	\disp \frac{\partial \log j}{\partial y^\alpha} = -\frac{1}{2}2 g^{ij}A^\alpha_{ij} = - H^\alpha,
\]
so the desired conclusion follows since $j \equiv 1$ on ${\bf 0}$. 
\end{proof}

 We return to the  codimension $2$ situation.
\begin{theorem} \label{codim2}  If $n-k=2$ and $L$ is a minimal submanifold, the form $\BS(L)$ is blow-up extendible along $L$. 
\end{theorem}
\begin{proof}
For $n-k=2$ we still have only one singular term apriori, namely 
\[\Omega:=d^*(H_2(r)\eta_0)=-r^{-1}\tilde{H}_2(r)\iota_{\nabla r}\eta_0+H_2(r)d^*\eta_0\]
 as Lemma \ref{Hfunc0} takes care of the rest of the terms.

We use Corollary \ref{maincorA} for the radial trivialization as described in Remark \ref{lr}.

 For $\omega=r^{-1}\tilde{H}_2(r)\iota_{\nabla r}\eta_0$ we use Corollary \ref{maincorA} in the case $(i,j,a)=(1,1,1)$, i.e. $a<i+j$.

As for the term $H_2(r)$ note  that $H_2(r)$ is singular at $r=0$ but $rH_2(r)$ extends continuously at $0$. 

We write
 \[H_2(r)d^*\eta_0=r^{-1}(rH_2(r))[(d^*\eta_0)_0+(d^*\eta_0)_1]\]
 By Proposition \ref{below} and \ref{below1} we have for a minimal submanifold $L$ that $d^*\eta_0\bigr|_{L}=0$, therefore $r^{-1}(d^*\eta_0)_0$ and $r^{-1}(d^*\eta_0)_1$ extend along $L$. We use Corollary \ref{maincorA} for the cases $(i,j,a)=(0,0,1)$ and $(i,j,a)=(1,0,1)$  to conclude that $H_2(r)d^*\eta_0$ is also blow-up extendible. This finishes the proof.
\end{proof}

\begin{corollary}\label{extcod2} The form $BS(\delta^S)$ where  $\delta^S\subset S\times S$ is the diagonal  of a compact, oriented Riemann surface $S$ extends to the blow-up.
\end{corollary}
\begin{proof}
The diagonal $\delta^S$ is the fixed point set of the isometry 
\[ R:S\times S\ra S\times S,\qquad R(x,y)=(y,x).\] Hence by a Theorem of Kobayashi \cite{Ko} (Theorem 5.1 Ch.2), $\delta^S$ is totally geodesic, hence minimal. 
\end{proof}

\begin{definition}\label{defMc} Let $M$ be a compact, oriented manifold of dimension $n+1$. Let $N$ be an oriented $n$-dimensional compact manifold. Let $F:N\ra M$ be a smooth map such that $F_*(N)=0$ in $H_n(M)$. Let $p_0,p_1\in M\setminus F(N)$. The $2$-points winding number of $w(F,p_0\ra p_1)$ is defined as the intersection number $I(F,C)$ of any closed curve $C$ connecting $p_0$ and $p_1$ that intersects $F$ transversely with $F$. 
\end{definition}
 The one point winding number notion needs $M$ non-compact. For more on this see \cite{OR}.

Definition \ref{defMc} is very similar to the one of the linking number when one of the submanifolds has dimension $0$. Corollary \ref{extcod2} implies the next simple result.

\begin{corollary}  Let $S$ be a compact, oriented Riemann surface. The smooth $1$-form $\mathscr{W}:=\BS(\delta^S)$ on $S\times S\setminus \delta^S$ has the property that for every oriented immersed, closed curve $\iota: C\hookrightarrow S$ with $\iota_*[C]=0$ in $H_1(S)$ and for every pair of points $p_0,p_1\in S\setminus \iota(C)$ one has
\[\int_C \Psi_{p_1}^*\mathscr{W}-\int_C \Psi_{p_0}^*\mathscr{W}=w(C,p_0\ra p_1)\]
where for any point $p\in S$, $\Psi_p:C\ra S\times S$ is the natural map $\Psi_p(x)=(\iota(x),p)$.  
\end{corollary}

  \appendix
    \section{Sign conventions and properties of forms} \label{Sc}

\subsection{Sign conventions - summary}
We make the following conventions in the category of oriented manifolds. The  oriented manifold $M$ will have dimension $n$ in what follows.
\begin{itemize}
\item[(1)] the orientation of the boundary in Stokes: exterior normal first, this is standard.
\item[(2)] the  choice of decomposition order  in integration along fiber: fiber first, base second  (Halperin-Vanstone use the opposite, Nicolaescu uses the same). 
\item[(3)] the orientation of the total bundle of a fiber bundle $P\ra M$: fiber first - base second. Note that this is independent of the choice in (2).

  The fiber first convention for $T_pP$ says that if $g_1,\ldots,  g_f$ is a (positively) oriented  basis for $\Ker d_p\pi$ then $\{g_1,g_2,\ldots g_f,e_1,\ldots,e_n\}$ is an oriented basis for $T_pP$ if and only if $\{d_p\pi(e_1),\ldots,d_p\pi(e_n)\}$ is an oriented basis for $T_{\pi(p)}M$. 

\item[(4)] the choice of the embedding $\Omega^{n-k}(M)\ra \mathscr{D}'_{k}(M)$:
 \[\omega\ra T_{\omega}(\eta)=\int_{M}\omega\wedge\eta\]
 (3) and (4) combine well to guarantee the following equality for a fiber bundle $\pi:P\ra M$ and $\omega\in\Omega^*_{\cpt}(P)$:
\[\pi_*(T_{\omega})=T_{\int_{\pi}\omega}\]
\item[(5)] the definition of $d:\mathscr{D}'_k(M)\ra\mathscr{D}'_{k-1}(M)$:
\[dT(\eta)=(-1)^{n-k-1}T(d\eta)\]
This has the advantage of recovering $d$ on currents induced by smooth forms. But beware that for $\alpha:M^m\ra N^n$ a proper map
\[d\alpha_*=(-1)^{m-n}\alpha_*d\]
\item[(6)] the orientation convention for a submanifold: normal bundle first; this means that the orientation of the ambient manifold is the orientation of the normal bundle wedge the orientation of the submanifold; this can be seen as an orientation convention for the submanifold or for the normal bundle, depending on necessity. Note that (6) and (3) fit with the exponential map and (1), (3) and (6) make the two orientations of the fiber of the  spherical normal bundle coincide
\item[(7)] the pull-back push-forward convention for a kernel $\bK\in \mathscr{D}'_{k,p}(M\times M)$:  the operator $K:\Omega^k_{\cpt}(M)\ra \mathscr{D}'_{p}(M)$ it determines is defined by 
\begin{equation}\label{Ko} K(\omega)(\eta):=(\pi_2)_*(\bK\wedge \pi_1^*\omega)(\eta)=(\bK\wedge \pi_1^*\omega)(\pi_2^*\eta)=\bK(\pi_1^*\omega\wedge \pi_2^*\eta)
\end{equation}
 Because $M\times M$ has the canonical product orientation only $\pi_2 $ (of the two projections onto the factors) satisfies (3), which makes it the candidate for push-forward. The wedging of a current and a smooth form in (\ref{Ko}) is always well-defined and in the case when $\bK\in L^1_{\loc}\Omega^*(M\times M)$ this is truly $\bK\wedge \pi_1^*\omega$.
\end{itemize}

One of the fundamental operations with forms is integration along the fibers of a fiber bundle $\pi:P\ra M$ with typical fiber a manifold $F$ of dimension $f$. It is important to understand what signs this operation brings on. We work in the category of \emph{oriented fiber bundles}, i.e. $VP:=\Ker d\pi$ is endowed with an orientation and fiber oriented diffeomorphisms.

Fix $\omega\in \Omega^k_{\cpt}(P)$. If one chooses a chart $U\subset M$ and a (structure) diffeomorphism $\varphi:P\bigr|_{U}\ra F\times  U$ then the form $(\varphi^{-1})^*\omega$ can be written as a finite sum
\begin{equation}\label{ord1}\sum_{i=0}^k\gamma_i\wedge \eta_{k-i}\end{equation}
where $\eta_{k-i}:F\times U\ra \pi_1^*\Lambda^{k-i} T^*U$ and $\gamma_i:F\times U\ra \pi_2^*\Lambda^{i}T^*F$ are purely horizontal (i.e. base) and purely vertical (i.e. fiber) forms respectively. This is what we mean by the sign convention number (2) above, that we choose to write the decomposition of $\omega$ by writing the base component first. In fact, if $U$ is chosen diffeomorphic to some $\bR^n$ one can refine (\ref{ord1}) and write $(\varphi^{-1})^*\omega$ as
\[\sum_{i=0}^k\gamma_i\wedge \pi_1^*\eta_{k-i}\] 
where $\eta_{k-i}\in \Omega^{k-i}(U)$ and $\gamma_i:F\times U\ra \pi_2^*\Lambda^{i}T^*F$ is a family of forms with compact support on $F$ parametrized by $U$, i.e. one can treat $\gamma_i:U\ra \Omega^{i}(F)$. Assume first for convenience that $\supp \omega\subset P\bigr|_{U}$. Then the definition of the fiber integral is
\[\int_{\pi}\omega:=\left(\int_{F}\gamma_f\right)\eta_{k-f}\in \Omega^{k-f}(U)\]
where $\int_{F}\gamma_f$ is a function on $U$. One shows that this definition does not depend on the choices made and it extends via a partition of unity to forms $\omega\in \Omega^k_{\cpt}(P)$ to give a well-defined operation
\[ \int_{\pi}=\int_{P/M} \ : \ \Omega^{k}_{\cpt}(P)\ra \Omega^{k-f}(U),\qquad \omega\ra \int_{\pi}\omega\]
with the following  properties
\begin{itemize} 
\item[(1)] Projection formula:
\begin{equation}\label{Pom1}\int_{\pi} \omega_1\wedge \pi^*\omega_2=\left(\int_{\pi}\omega_1\right)\wedge\omega_2\end{equation}
\item[(2)] (Fubini) If $P$ and $M$ are also oriented and the fiber first orientation convention holds for the orientation on $P$,  then for a top degree form $\Omega$
\begin{equation}\label{POm} \int_P\Omega=\int_{M}\int_{\pi}\Omega\end{equation}
A particular case of this for $|\omega_2|\leq\dim M$ and $|\omega_1|=\dim P-|\omega_2|$ is:
\[\int_P\omega_1\wedge\pi^*\omega_2=\int_{M}\int_{\pi}\omega_1\wedge\pi^*\omega_2=\int_M\left(\int_{\pi}\omega_1\right)\wedge \omega_2\]
where we use (\ref{POm}) for the first and (\ref{Pom1}) for the second equality.
 If, on the contrary, the orientation on $P$ satisfies a base first orientation convention, then the left hand side of (\ref{POm}) gets a $(-1)^{fn}$ while the right hand side stays unchanged.
\item[(3)] If $\varphi:P_1\ra P_2$ is a smooth mapping between two fiber bundles $\pi_i:P_i\ra M$ ($i=1,2$) which is an oriented fiberwise diffeomorphism then for every $\Omega\in\Omega^{*}(P_2)$
\[\int_{\pi_1}\varphi^* \Omega=\int_{\pi_2}\Omega\]
\item[(4)]  If $\omega\in \Omega^k_{\cpt}(P)$ then $\pi_*(T_\omega)$ is a current represented by $\int_{\pi}\omega$, i.e.
\[\pi_*(T_{\omega})(\eta):=T_{\omega}(\pi^*\eta)=\int_P\omega\wedge \pi^*\eta=\int_{M}\left(\int_{\pi}\omega\right)\wedge \eta.\]
This is the main reason for choosing the fiber first convention.
\end{itemize}

\begin{example} Let $M$ be a compact oriented manifold and let $M\times M$ be endowed with the canonical product orientation. Let $R(x,y)=(y,x)$ and $\Omega\in \Omega^{2n}(M\times M)$. Then
\[\int_{\pi_1}R^*\Omega=\int_{\pi_2}\Omega\]
and, since $\pi_2$ is fiber first oriented and $\pi_1$ is base first we have
\[\int_M\int_{\pi_2}\Omega=\int_{M\times M}\Omega=(-1)^{n^2}\int_{M\times M}R^*\Omega=\int_M\int_{\pi_1}R^*\Omega\]
\end{example}

Now the choice of fiber first orientation convention (3)  for a fiber bundle, almost immediately requires the normal directions first convention (6) when orienting a submanifold, because of the normal exponential which gives the tubular isomorphism theorem. This convention fits well with transversality.

\vspace{0.4cm}
\subsection{A representation formula}
Let $M$ be an oriented compact manifold.
We want to prove now a representation formula for the current $P(T)$ when $P:\Omega^k(M)\ra \Omega^k(M)$ is a smoothing operator  and $T\in\mathscr{D}'_k(M)$ is a current.

We define $P(T)$ by duality:
\[P(T)(\phi):=T(P(\phi)),\qquad\]

 Let $p\in \Omega^{n-k,k}(M\times M)$ be the smooth kernel of $P:\mathscr{D}'_k(M)\ra \mathscr{D}'_k(M)$. On one hand, the pull-back $\pi_2^*T$ is well-defined for all  currents $T$
\[\pi_2^*T(\eta):=T\big((\pi_2)_*\eta\big)
\]
Then
\[P(T)(\phi)=T(P(\phi))=T((\pi_2)_*(p\wedge \pi_1^*\phi))=(\pi_2^*T)(p\wedge \pi_1^*\phi)=(\pi_1)_*(p\wedge \pi_2^*T)(\phi)\]
We claim that $(\pi_1)_*(p\wedge \pi_2^*T)$ is representable by a $n-k$-form on $M$ which we now describe.

For every $x\in M$ the inclusion 
  \[ \iota_x: M\hookrightarrow M\times M,\qquad y\ra (x,y)\]
  induces a form $p\circ \iota_x\in\Omega^k(M)$ with values in $V_x:=\Lambda^{n-k}T^*_xM$, i.e. an element of $\Omega^k(M)\otimes V_x$. Given $T\in\mathscr{D}_k'(M)$  one considers 
 \[ (T\otimes \id_{V_x})(p\circ \iota_x)\in V_x.\]  
and this correspondence gives an $n-k$ form $T(p\circ \iota)$ on $M$:
  \[x\ra  (T\otimes \id_{V_x})(p\circ \iota_x) \]  An adaptation of Hormander  \cite{H1} Theorem 2.1.3  shows that this form is smooth. In order to state the next result we need a definition.  If $L:\Omega^{n-k}(M)\ra \Omega^{j}(M)$ is a differential operator then for every $l$ define $L_x:\Omega^{n-k,l}(M\times M)\ra \Omega^{j,l}(M\times M)$:
\[L_x(\pi_1^*\omega\wedge \pi_2^*\eta):=\pi_1^*(L\omega)\wedge \pi_2^*(\eta)\]
Since $\Omega^{n-k}(M)\otimes \Omega^{l}(M)$ is dense in the $\Omega^{n-k,l}(M\times M)$ this extends by continuity to all of $\Omega^{n-k,l}(M\times M)$. 
  \begin{prop}\label{Repfor} For every $T\in\mathscr{D}' _k(M)$ and every $\phi\in \Omega^k(M)$ the following holds
  \[(-1)^k\int_MT(p\circ \iota)\wedge \phi=T\left(\int_{\pi_2}p\wedge \pi_1^*\phi \right)=(\pi_1)_*(p\wedge \pi_2^*T)(\phi)
  \] Moreover if $L:\Omega^{n-k}(M)\ra \Omega^l(M)$ is a differential operator, then 
\begin{equation}\label{Lpi}L[(\pi_1)_*(p\wedge \pi_2^*T)]=(\pi_1)_*((L_xp)\wedge \pi_2^*T)\end{equation}
  \end{prop}
  \begin{proof} The first statement is clearly true for kernels $p\in \Omega^{n-k}(M)\otimes \Omega^{k}(M)$ i.e. if $p=\pi_1^*\xi\wedge\pi_2^*\zeta$ where $\xi\in \Omega^{n-k}(M)$ and $\zeta\in\Omega^{k}(M)$. Indeed
  \begin{equation}\label{aeq1}\int_{\pi_2}p\wedge \pi_1^*\phi =(-1)^{kn}\zeta\int_M \phi\wedge\xi
  \end{equation}
  while $p\circ \iota_x= \zeta\otimes \xi_x $ and therefore 
    \begin{equation}\label{aeq2}(-1)^k T(p\circ \iota)\wedge \phi=(-1)^{k^2}T(\zeta)\xi\wedge \phi=(-1)^{k^2+k(n-k)}T(\zeta)\phi\wedge \xi.\end{equation}
  Applying $T$ to (\ref{aeq1}) and $\int_M$ to (\ref{aeq2}) gives the claim.

 On the other hand, the tensor space $\Omega^{n-k}(M)\otimes \Omega^{k}(M)$ is dense in  $\Omega^{n-k,k}(M\times M)$ endowed with the standard $C^{\infty}$ topology. Hence 
  \[
  p=\lim_{j\ra \infty} p_j
  \]
  with $p_j\in \Omega^{n-k}(M)\otimes \Omega^{k}(M)$. Then the operation of wedging with $\phi$ and integrating over the fiber $\int_{\pi_1}$ are both continuous operations in the test functions topology. Hence
  \[ 
  T\left(\int_{\pi_1}p_j\wedge \pi_2^*\phi \right)\ra T \left(\int_{\pi_1}p\wedge \pi_2^*\phi \right).
  \]
  Moreover $T(p_j\circ \iota)$ converges to $T(p\circ \iota)$ in $\Omega^{n-k}(M)$ and therefore $T(p_j\circ \iota)\wedge \phi$ converges uniformly to $T(p\circ \iota)\wedge \phi$. So we can interchange integral and limit.

In order to prove (\ref{Lpi}) one checks that it is true on forms $p=\pi_1^*\xi\wedge \pi_2^*\zeta$ and uses the continuity of any differential operator in the $C^{\infty}$ topology.
  \end{proof}

\section{Blow-up extendibility of differential forms} \label{apB}

This section is about extendibility of differential forms $\omega$ defined in the complement $M\setminus L$ of a closed submanifold $L$, to the oriented blow-up $\Bl_L(M)$. The total space of $\Bl_{L}(M)$ is a manifold with  boundary   and any  Riemannian metric on $M$ can be used for a concrete realization. More precisely,  in order to glue $[0,\epsilon)\times S(\nu L)$ and $M\setminus L$, one uses the composition of  two maps, namely (next $D_{\epsilon}(\cdot)$ is the disk bundle of radius $\epsilon$)
\begin{equation}\label{introblw}[0,\epsilon)\times S(\nu L)\ra D_{\epsilon}(\nu L),\qquad (t,p,v)\ra (p,tv),\qquad p\in L, ~v\in S_p(\nu  L)\end{equation}
 mapping which is a diffeomorphism away from $\{0\}\times S(\nu L)$ and
\[ \exp^{\perp}:D_{\epsilon}(\nu L)\ra M,\qquad (p,w)\ra \exp_p(w)\]
which is a difeomorphism onto a neighborhood $U$ of $L$. We therefore see that $\Bl_L(M)$ is, close to its boundary, difeomorphic with a neighborhood of the blow-up of the zero section of $\nu  L$. This blow up of the zero section of $\nu L$ is given by the expression of (\ref{introblw}). 
For the question of extendibility, there is no significant simplification in considering the specific case of the bundle $\nu L \to L$.   

We are therefore led to study the extendibility of a form $\omega$ in the following setting:
\begin{itemize}
	\item[-] $\pi : E \to B$ is a vector bundle over a closed Riemannian manifold $(B,g_B)$; elements of $E$ will be denoted by $(p,v)$, $p \in B$, $v \in E_p$.
	\item[-] $\metric$ is a bundle metric on $E$;
	\item[-] ${\bf 0} \subset E$ is the image of the zero-section  $B \to E$; 
	\item[-] $\omega \in \Omega^k(E \setminus {\bf 0})$ is a smooth form. 
\end{itemize}
Let $SE \subset E$ be the subbundle of unit vectors, and let
\[
\Bl: [0,\infty)\times SE\ra E,\qquad \Bl(t,p,v)=(p,tv)
\]
be the blow-up of the zero-section. We also write
\[
\Bl_{{\bf 0}}(E) : = [0,\infty)\times SE
\]
for the total space of the blow-up, and we will use $\pi_2$ for the projection $[0,\infty)\times SE \ra SE$.

The map $\Bl$ induces a natural bigrading on forms by writing
	\[
	\Bl^* \omega = \hat\omega_0 + dt \wedge \hat\omega_1,
	\]
	where 
	\[
	\hat\omega_1 := \iota_{\partial_t} \Bl^*\omega, \qquad \hat\omega_0 : = \Bl^*\omega - dt \wedge \hat\omega_1.
	\]
	The forms $\hat\omega_0$ and $dt \wedge \hat \omega_1$ are called the components of $\Bl^*\omega$ with bigrading $(0,k)$ and $(1,k-1)$, respectively. By construction, $\iota_{\partial_t} \hat \omega_j =0$ for $j \in \{0,1\}$, whence in particular $\hat \omega_1$ has bigrading $(0,k-1)$.
	\begin{definition}  $\omega \in \Omega^k( E \setminus {\bf 0})$ is said to be 
		\begin{itemize}
			\item[-] (blow-up) extendible if $\Bl^*\omega$ extends to a continuous differential form on all of $\Bl_{{\bf 0}}(E)$.
			\item[-] weakly (blow-up) extendible  if  the bidegree $(0,k)$ part of $\Bl^*\omega$ extends to a continuous form on $\Bl_{{\bf 0}}(E)$. 
		\end{itemize}
	\end{definition}

\begin{remark}
	It follows directly from the definition that $\omega$ is extendible if and only if both $\hat\omega_0$ and $\hat\omega_1$ extend continuously up to $t=0$.
\end{remark}

We define 
\[
r : E \to [0,\infty), \qquad r(p,v) = |v| \doteq \sqrt{ \langle v,v\rangle}
\]
Let $VE:=\Ker d\pi\subset TE$ be the tangent bundle to the fibers of $\pi : E \to B$. The natural isomorphism $VE\simeq \pi^*E$ as  bundles over $E$ is used almost as an identification, and endow $VE \to E$ with a bundle metric still denoted by $\metric$. Thus, we can consider the vertical gradient $\nabla r:E\setminus {\bf 0}\ra VE$ of $r$, i.e. the metric dual of $dr\bigr|_{VE}$. In fact,
\[\nabla r= r^{-1}\mathcal{R}\]
where $\mathcal{R}:E \ra VE \simeq \pi^*E$ is the tautological section.

\begin{remark}
	When $E = \nu L$,  a neighborhood of the zero section is also endowed with the Riemannian metric $g$ obtained by pulling back the metric of $M$ via $\exp^\perp$. In this case, up to shrinking the neighborhood the function $r$ is the distance from ${\bf 0}$ in the metric $g$, and $\nabla r$ coincides with the $g$-gradient $\nabla^g r$ of $r$. In particular, the flow of $\nabla r$ is via $g$-geodesics. We will come back to this point in more details later, in connection to Example \ref{Partrc}. However, in most of what follows we do not assume that $E$ (or a neighbourhood of ${\bf 0}$) is endowed with a Riemannian metric.
\end{remark}

To state our first characterization of (weak) extendibility, for $\lambda > 0$ define the scaling morphism
\[
\varphi_\lambda : SE \to E, \qquad (p,v) \mapsto (p, \lambda v).
\]

\begin{prop} \label{prwexB} 
A form $\omega\in \Omega^k(E\setminus {\bf 0})$ is: 
\begin{itemize} 
\item[(i)] weakly (blow-up) extendible if and only if the family of forms on $SE$, $\lambda\ra \varphi_\lambda^*\omega$ is uniformly continuous on some interval $(0,\epsilon]$.
\item[(ii)]  (blow-up) extendible if and only if both $\omega$ and $\iota_{\nabla r}\omega$ are weakly extendible.
\end{itemize}
\end{prop}

\begin{remark}\label{MetE}
	{The notion of uniform continuity for the family of forms $\lambda\ra \varphi_\lambda^*\omega$ requires to specify a Banach space structure on $\Omega^*(SE)$, which we take to be the $C^0$-norm induced by the choice of a Riemannian metric on $SE$. The compactness of $SE$ guarantees that different choices of the metric lead to equivalent norms. Hereafter, we fix one such metric without further mentioning. 
	}
\end{remark}

\begin{proof} Write $\Bl^*\omega = \omega_0 + dt \wedge \omega_1$ as above. 
	
For (i), note that the $(0,k)$-part $\omega_0$ of $\omega$ can be seen as a family of forms on $SE$ parametrized by $t \in (0,\epsilon]$. Precisely, we can identify $\omega_0$ with the family $\{j_t^*\omega_0\}_{t>0}$, where 
$j_t : SE \to [0,\infty) \times SE$, $j_t(p,v) = (t,p,v)$. The extendibility of $\omega_0$ at $t=0$ is equivalent to the uniform continuity of the family for $t \in (0, \epsilon]$. Since $\Bl \circ j_\lambda = \varphi_\lambda$ we deduce $j_\lambda^* \omega_0 = \varphi_\lambda^* \omega$, and the conclusion follows.

For (ii) we note that the part of bidegree $(1,k-1)$ of $\Bl^*\omega$ can be written as
\[ 
dt\wedge \iota_{\partial_t}\Bl^*\omega,
\]
i.e. $\omega_1 = \iota_{\partial_t}\Bl^*\omega$. By construction, $d\Bl(\partial_t)$ is the normalized tautological vector field on $E$,  and therefore for $t\neq 0$:
 \[
 d\Bl(\partial_t)=\nabla r.
 \] 
 Therefore, over $(0,\epsilon)\times SE$ one has
\[ 
\iota_{\partial_t}\Bl^*\omega=\Bl^* ({\iota}_{\nabla r}\omega).
\]
The form $ \iota_{\partial_t}\Bl^*\omega$ has bidegree $(0,k-1)$, hence the weak extendibility of ${\iota}_{\nabla r}\omega$ is equivalent to the extendibility of $\omega_1 =  \iota_{\partial_t}\Bl^*\omega$, as claimed.
\end{proof}

Observe that the proof of the above proposition shows that decomposing 
\[
\omega = \omega_0 + dr \wedge \omega_1
\]
with 
\[
\omega_1 = \iota_{\nabla r} \omega, \qquad \omega_0 = \omega - dr \wedge \omega_1
\]
then one has
\[
\Bl^* \omega_j = \hat \omega_j, \qquad i_{\nabla r} \omega_j = 0
\]
and $\omega$ is extendible if and only if both $\omega_0$ and $\omega_1$ are extendible.

Proposition \ref{prwexB} leads to effective criteria to check the extendibility of $\omega$ provided one can make explicit the action of the scaling morphism $\varphi_\lambda$. This leads to study decompositions of $\omega$ into components with different ``vertical degree", on which $\varphi_\lambda$ acts as a multiplication by a suitable power of $\lambda$. To begin with, we study the special case in which $B$ is a point.

\subsection*{The case $B$ a point} 

Let us look at the particular situation
\[ 
B=\{0\}\subset \bR^n=:E.
\]

 Let $r:\bR^n\ra [0,\infty)$, $r(v):=|v|$ be the distance function to the origin with gradient $\nabla r$ and differential $dr\in \Omega^1(\bR^n\setminus \{0\})$. 

The next statement is a particular case of Proposition \ref{prwexB} which makes explicit the action of pull-back via the rescaling morphism. However, we state it in a slightly different form, which tacitly assumes a \emph{canonical} identification of vector spaces as subspaces of $\bR^n$:
\begin{equation}\label{Parti}T_{tv}[S^{n-1}(t)]=T_{v}S^{n-1},\qquad t>0, \; v\in S^{n-1}\end{equation} given by parallel transport. 

With this identification, a linear map $\omega_{tv}:\Lambda^k T_{tv}S^{n-1}(t)\ra \bR$ induces a linear map  $\Lambda^k T_vS^{n-1}\ra \bR$ denoted by the same symbol.  The main novelty  with respect to Proposition \ref{prwexB} is that the action of the rescaling pull-back operator $\varphi_{t}^*$ is now clearly identified with multiplication by $t^{\deg}$ where $\deg$ is the degree of the form to which the operator applies.
\begin{prop}\label{prwexR} Let $\omega\in\Omega^{k}(\bR^n\setminus \{0\})$ be a continuous form. Then $\omega$ is blow-up extendible if and only if the $1$-parameter families of forms on $S^{n-1}$:
\[ 
t\ra \{v\ra t^k(\omega_0)_{tv}\},\qquad t\ra \{v\ra t^{k-1}(\omega_1)_{tv}\},\qquad v\in S^{n-1} 
\]
are uniformly continuous on some interval $(0,\epsilon]$ where 
\[\omega=\omega_0+dr\wedge \omega_1\]
is the decomposition of $\omega$ induced by $T\bR^n\bigr|_{\neq \{0\}}=\langle \nabla r\rangle\oplus \langle \nabla r\rangle^{\perp}$. In particular, $\omega_1=\iota_{\nabla r}\omega$.

The form $\omega$ is weakly extendible if and only if the first family involving $\omega_0$ is uniformly continuous.
\end{prop}
\begin{example} The $2$-form 
\[B=\frac{x_1}{|x|^3}dx_2\wedge dx_3-\frac{x_2}{|x|^3}dx_1\wedge dx_3+\frac{x_3}{|x|^3}dx_1\wedge dx_2=\frac{1}{|x|^2}\iota_{\nabla r}(\vol_{\bR^3})\] defined on $\bR^3\setminus \{0\}$ is fully extendible to the blow-up of $0$. 
In fact, since $\iota_{\eR}B=0$ the full-extendability is equivalent in this case with the weak-extendability. The latter is immediate from 
\[ r^2 B_{rv}=v_1dv_2\wedge dv_3-v_2dv_1\wedge dv_3+v_3dv_1\wedge dv_2=\vol_{S^2} \]
We note that the condition of uniform continuity at $0$ of the family of $2$-forms on $S^2$ 
\[r\ra \{v\ra r^2B_{rv}\}\] is weaker than the condition that the $2$-form on $B^3\setminus \{0\}$:  \[ x\ra r^2 B_{rx}\] be continuously extendible  at $x=0$. The latter is not even  satisfied in this example. 
\end{example}

\begin{example}\label{B8} Every $1$-form $\omega$ on $\bR^2\setminus \{0\}$ can be written uniquely as
\[\omega=f_0(x,y)dr^{\perp}+f_1(x,y)dr\]
where 
\[r=\sqrt{x^2+y^2}\qquad\quad dr=\frac{x}{r}dx+\frac{y}{r}dy\qquad\quad dr^{\perp}=\frac{-y}{r}dx+\frac{x}{r}dy.\]
Denote $dS^1:=dr^{\perp}\bigr|_{S^1}$. For $j\in\{0,1\}$, the families of $(1-j)$-forms on $S^1$,  $t\ra (\omega_j)_t$ are
\[t\ra \{S^1\ni (x,y)\ra f_0(tx,ty)dS^1\}\qquad\qquad t\ra \{S^1\ni (x,y)\ra f_1(tx,ty)\}\]
The blow-up extendibility of $\omega$ is equivalent to the uniform continuity of the families of functions on $S^1$: \[ t\ra tf_0(t\cdot,t\cdot)\qquad \qquad t\ra f_1(t\cdot,t\cdot).\]
\end{example}

\subsection*{The general case} 

In order to investigate the action of $\varphi_t^*$ on forms on $E \backslash {\bf 0}$, one needs some type of extra-structure. Let us first see what  else we can do with the available data .

{For $t>0$ let $S^tE:=r^{-1}\{t\}$ be the level sets of $r$ (so $SE=S^1E$). The total tangent bundle to the foliation by spherical bundles $S^tE$ over $E \backslash {\bf 0}$ is denoted by $T\hat{S}E$, so that 
\[
T\hat{S}E\bigr|_{SE}=TSE.
\]
Define $VSE:=\Ker d(\pi\bigr|_{SE})$ and $V\hat{S}E$ for the orthogonal complement of $\langle\nabla r \rangle$ inside $VE\bigr|_{E\setminus {\bf 0}}$. 
%
First a general observation:
\begin{lemma}\label{natu} There exists a natural isomorphism of bundles over $E\setminus {\bf 0}$:
\[T\hat{S}E\simeq \Pi^* TSE\]
where $\Pi:E\setminus {\bf 0}\ra SE$ is the radial projection $(p,v)\ra (p,|v|^{-1}v)$. 
\end{lemma}
\begin{proof} Consider the flow of the vertical vector field $\nabla r$ on $E\setminus {\bf 0}$. Since $S^tE$ are level sets of $r$ the induced family of diffeomorphisms  take level sets to level sets and thus their differentials give bundle isomorphisms $TS^tE\ra TSE$. These isomorphisms, put together, prove the statement. 
\end{proof}
Let $\omega\in \Omega^{k}(E\setminus {\bf 0})$. The decomposition 
\begin{equation}\label{om01}
	\omega=\omega_0+dr\wedge \omega_1
\end{equation}
into its components with bidegree $(0,k)$ and $(1,k-1)$ is the one naturally induced by the splitting 
\[
TE\bigr|_{E\setminus {\bf 0}} =\bR\nabla r\oplus T\hat{S}E, 
\]
}
An immediate consequence of Lemma \ref{natu} is the following
\begin{lemma}\label{om01t} Every form $\omega\in \Omega^k(E\setminus {\bf 0})$ gives rise and is uniquely determined by two $1$-parameter families of forms $((\omega_j)_t)_{t\in (0,\infty)}$ on $SE$ with $j=0,1$ of degrees $k-j$:
\[ t\ra (\omega_j)_{t}:=\{(p,v)\ra (\omega_j)_{(p,tv)}\},\qquad \forall (p,v)\in SE\]
\end{lemma}
\begin{proof} By Lemma \ref{natu} we have the identification for every $t>0$:
\[T_{(p,tv)}S^tE\simeq T_{(p,v)}SE,\qquad \forall (p,v)\in SE\]
and so $(\omega_0)_{(p,tv)}$ and $(\omega_1)_{(p,tv)}$ from (\ref{om01}) can be considered as  multilinear functionals on $T_{(p,v)}SE$. We thus get the two claimed families. It is obvious that the process can be reversed.
\end{proof}
\begin{remark} {In the particular case of a point, the identification in Lemma \ref{natu} coincides with the parallel transport identification (\ref{Parti}) along rays.} The rescaling diffeomorphisms $\varphi_r$ are restrictions of the flow diffeomorphisms for the tautological vector field $\mathcal{R}$, as opposed to the flow of $\nabla r$. In the case of a point,  as displayed in Proposition \ref{prwexR} and Example \ref{B8},  {after the identification of Lemma \ref{natu} the map $\varphi_r^*$ acts by} multiplication by a power of $r$ that depends on the degree of the form under consideration. In general, {as we shall see,} this multiplication action happens only along vertical directions and depends on the "vertical degree". 
\end{remark}
We  consider the first type of extra structure in order to make precise the ideas of the previous Remark.

\begin{definition}
An  Ehresmann connection on the vector bundle $E$ is  a splitting $s:TE\ra VE$ of the exact sequence of vector bundles over $E$:
\begin{equation} \label{exseq}0\rightarrow VE \hookrightarrow TE{\overset {d\pi}{\longrightarrow} }\pi^*TB\ra 0\end{equation}
{that is, $s$ restricts to the identity on $VE$.} Let $HE:=\Ker s$.  The Ehresmann connection is \emph{linear} if the following two properties hold:
\begin{itemize} 
\item[(i)] $H_{(p,0)}E=T_p B$ where $T_pB\subset T_{(p,0)}E$ via the differential of the zero-section .
\item[(ii)] $d\varphi_{\gamma}(H_{(p,v)}E)=H_{(p,\lambda v)}E$ where $\varphi_{\gamma}:E\ra E$ is the rescaling diffeomorphism by factor $\gamma$.
\end{itemize}
\end{definition}
The Ehresmann connection induces an isomorphism of vector bundles 
\[TE\simeq \pi^*E\oplus \pi^*TB=\pi^*(TE\bigr|_{[0]}).\]

When $E\ra B$ is a Riemannian vector bundle {with bundle metric $\metric$}, a \emph{linear} Ehresmann connection $s:TE\ra VE$ is said to be \emph{metric compactible} if

\[ H_{(p,v)}E\subset T_{(p,v)}SE,\qquad \forall (p,v)\in SE.\]

\begin{example} A metric Koszul connection $\nabla:\Gamma(TB)\times \Gamma(E)\ra \Gamma(E)$ naturally determines a metric compatible Ehresmann connection. The {horizontal}  subspace $H_{(p,v)}E$ consists of those vectors $u\in T_{(p,v)}E\setminus V_{(p,v)}E$ for which there exists a local section $s:U\ra E\bigr|_{U}$ such that $s(p)=(p,v)$, $u=d_ps(d\pi(u))$ and $\nabla_{d\pi(u)}s=0$. It is standard that the Ehresmann connection thus defined is linear. The metric compatibility follows from the following relation. {Define 
\[
f:E\ra \bR,\qquad f(p,v)= \langle v,v \rangle_p.
\]
} Then one computes
\[d_{(p,v)}f(w,u)=2 \langle v,w \rangle_p +(\nabla_{d\pi(u)} \metric)(v,v),\qquad \forall (w,u)\in V_{(p,v)}E\oplus H_{(p,v)}E=E_p\oplus H_{(p,v)}E
\]
Since $\nabla$ is metric compatible, the second term on the right vanishes and so, for $|v|=1$ one has 
 \[T_{(p,v)}SE=\Ker d_{(p,v)}f\supset H_{(p,v)}E.\]
\end{example}
\begin{remark} The notion of  Ehresmann connection makes sense on any fiber bundle $P\ra B$. In particular, if $P=[0,\infty)\times SE$, a metric Ehresmann connection on $E\ra B$ induces an obvious Ehresmann connection on $P$ such that $d\Bl_{[0]}$ is a bundle isomorphism (over $P$) between $HP$ and $\Bl_{[0]}^*HE$ as one can easily check.
\end{remark}

An Ehresmann connection allows to split the bundle of $k$-forms as follows:
\[ \Lambda^kT^*E\simeq \bigoplus_{i=0}^k\Lambda^iV^*E\otimes \Lambda^{k-i}H^*E.\]
The extra decomposition $VE=\bR\nabla r\oplus V\hat{S}E$ along $E\setminus {\bf 0}$ coupled with  a \emph{metric} Ehresmann connection gives a more refined splitting
\begin{equation}\label{om03} \Lambda^kT^*E\simeq \left(\bigoplus_{i=0}^k\Lambda^iV^*\hat{S}E\otimes \Lambda^{k-i}H^*E\right)\oplus \left(\bR dr\otimes \bigoplus_{i=1}^{k}\Lambda^{i-1}V^*\hat{S}E\otimes \Lambda^{k-i}H^*E\right)\end{equation}
compatible with the decomposition
\begin{equation}\label{om04}\Lambda^kT^*E\simeq\Lambda^kT^*\hat{S}E\oplus \bR dr\otimes \Lambda^{k-1}T^*\hat{S}E.\end{equation}
Compatibility means here that the big brackets  in (\ref{om03}) are  equal separately to the factors in (\ref{om04}). Hence given $\omega\in \Omega^k(E\setminus {\bf 0})$  and recalling (\ref{om01}) we have, for $j=0,1$   a decomposition
\begin{equation}\label{om05}\omega_j:=\sum_{i=0}^{k-j}\omega_{ji}\end{equation}
with 
\[\omega_{ji}\in \Gamma(E\setminus {\bf 0};\Lambda^{i}V^*\hat{S}E\otimes \Lambda^{k-i-j}H^*E).
\] 
The following analogue of Lemma \ref{natu} holds. 
\begin{lemma} There exist natural isomorphisms of vector bundles over $E\setminus {\bf 0}$:
\[V\hat{S}E\simeq \Pi^*VSE,\qquad HE\simeq  \pi^*TB\simeq \Pi^*(HE\bigr|_{SE}) \]
\end{lemma}
\begin{proof} The first isomorphism uses the same argument as Lemma \ref{natu} because the diffeomorphisms of the flow commute with the projection $\pi$ onto $B$.

The second isomorphism is a consequence of the definition of an Ehresmann connection and the isomorphism is $d\pi\bigr|_{HE}$. But of course we also have $HE\bigr|_{SE}\simeq (\pi\bigr|_{SE})^*TB$. 
\end{proof}
It follows that each of the forms $\omega_{ji}$ can be considered as a 1-parameter family of sections 
\[t\ra (\omega_{ji})_{t}=\{(p,v)\ra({\omega}_{ji})(p,tv),\quad (p,v)\in SE\}\] of the vector bundle $\Lambda^iV^*SE\otimes \Lambda^{k-i-j}H^*E\ra SE$. In fact, from the above Lemma it follows that the identification $HE\simeq \Pi^*(HE\bigr|_{SE})$ is done through the auxiliary bundle $\pi^*TB$. Hence it is more useful to consider the related 1-parameter families
\[t\ra (\hat{\omega}_{ji})_t\]
where $(\hat{\omega}_{ji})_t\in \Gamma(\Lambda^iV^*SE\otimes \Lambda^{k-i-j}\pi^*T^*B)$ are obtained from $(\omega_{ji})_t$ via the isomorphism $HE\simeq \pi^*TB$.

\begin{definition} Let $\pi:P\ra B$ be a fiber bundle and let $F\ra B$ be a vector bundle. A vertical form on $P$ with values in  $F$ of degree $i$ is a section of $\Lambda^iV^*P\otimes \pi^*F$.
\end{definition}
The forms $(\hat{\omega}_{ji})_t$ are thus vertical forms on $SE$ with values in the bundle $\Lambda^*T^*B$.

A first generalization of Proposition \ref{prwexR} is the following.

\begin{theorem} \label{mainthE}Let $\omega\in \Omega^k(E\setminus {\bf 0})$.
 \begin{itemize}\item[(i)] The form $\omega$  is weakly extendible if and only if the following $1$-parameter families of vertical forms  on $SE$:
\[
t\ra t^i(\hat{\omega}_{0i})_{tv}, \qquad {i  \in \{0,\ldots, k\}}
\]
are uniformly continuous for $t\in (0,\epsilon]$.
\item[(ii)] The form $\omega$ is blow-up extendible if and only if, {for each $j \in \{0,1\}$}, the following $1$-parameter families of vertical forms on $SE$
\[
t\ra t^i(\hat{\omega}_{ji})_{tv}, \qquad {i  \in \{0,\ldots, k-j\}}
\]
are uniformly continuous for $t\in (0,\epsilon]$
\end{itemize}
\end{theorem}
\begin{proof} We only sketch the main idea of the proof because the result is also a consequence of a second generalization, namely Theorem \ref{thmainA} below.

 One direct proof uses Proposition \ref{prwexB} and the fact that the differential of the rescaling diffeomorphism $\varphi_t:SE\ra S^tE$ is of diagonal type 
\[\left(\begin{array}{cc} t&0\\
0&1\end{array}\right)\]
with respect to the decomposition $TSE\simeq VSE\oplus \pi^*TB$, i.e. after the identification of $HE\bigr|_{SE}$ with $\pi^*TB$. 
\end{proof}

\vspace{0.5cm}

 We now consider more general splttings of $TE$ into "vertical" and "horizontal" components. The relevant type of operation whose importance will become clear later,  in Theorem \ref{indth}, is the next one.

\begin{definition} Let $F\ra N$ be a vector bundle over a compact manifold $N$ with a $\bZ_2$-grading $F=F_0\oplus F_1$. An endomorphism $\tilde{\gamma}\in \End(F)$  is said to be \emph{even} if $\tilde{\gamma}(F_j)\subset F_j$, $j=0,1$.

The bundle endomorphism $\mathcal{M}:=(\mathcal{M}_t)_{t\geq 0}\in \End(\pi_2^*F)$ acting on $\pi_2^*F\ra [0,\infty)\times N$
\[\mathcal {M}_{(t,n)}(f_0,f_1)=(f_0,tf_1),\qquad \forall (t,n)\in  [0,\infty)\times N, ~ (f_0,f_1)\in F_{n}\]
is called a multiplication operator.

By extension, a multiplication operator is any of the endomorphisms $\Lambda^q\mathcal{M}\in \End(\pi_2^* \Lambda^ qF)$ or their adjoints $\Lambda^q\mathcal{M}^*\in \End(\pi_2^* \Lambda^ qF^*)$. The action of $\Lambda^q\mathcal{M}$ with respect to the splitting
\[\Lambda^q F=\bigoplus_{i=0}^q \Lambda^iF_1\otimes \Lambda^{q-i}F_0\]
is as follows for $s=(s_t)_{t\geq 0}=\sum_{i=0}^q((s_i)_t)_{t\geq 0}$, with  {$s_i \in \Gamma(\pi_2^* \Lambda^i F_1 \otimes \pi_2^*\Lambda^{q-i} F_0)$:}
\[\Lambda^q\mathcal{M}(s)=\Lambda^q\mathcal{M}_t(s_t)=\sum_{i=0}^q t^i(s_i)_t.\]
\end{definition}
 Typically, $F\ra N$ is endowed with a Riemannian metric which makes $\Gamma(F)$ a Banach space with the $\sup$ norm.  The next {equivalence, while direct,} is stated to emphasize the relevance for our purposes of  vector bundles of type $\pi_2^*F$  over $[0,\infty)\times N$,  i.e. bundles with "constant fibers" in $t$.  Clearly, every vector bundle over $[0,\infty)\times N$ is isomorphic non-canonically to such a vector bundle.

The restriction $\pi_2^*F\bigr|_{(0,\infty)\times N}$ is denoted $\pi_2^*F\bigr|_{t>0}$.

\begin{prop}\label{stprop}  Let $s:=(s_t)_{t>0}\in \Gamma(\pi_2^*F\bigr|_{t>0})$ be a continuous section. Then $s$ extends continuously at $t=0$ if and only if $t\ra s_t\in \Gamma(F)$ is uniformly continuous.
\end{prop}
The abstract result we need is the following one. Though straightforward this is crucial to check the extendibility property.
\begin{theorem}\label{kth} \begin{itemize} 
\item[(i)] Let $\gamma:=(\gamma_t)_{t\geq 0}\in \Gamma(\End(\pi_2^*F))$ be of class $C^1$. If $\gamma(0)\in \End(F)$ is even then $\mathcal{M}_t\gamma_t\mathcal{M}_t^{-1}\in \Gamma(\End(\pi_2^*F\bigr|_{t>0}))$ is continuously extendible at $t=0$. Consequently, the conjugation by $\Lambda^q\mathcal{M}^*$ of $\Lambda^q\gamma^*$ is also continuously extendible at $t=0$.
\item[(ii)] Let $s:=(s_t)_{t>0}\in \Gamma(\pi_2^*F\bigr|_{t>0})$ be a continuous section. Let $\gamma:=(\gamma_t)_{t\geq 0}\in \Gamma(\Aut(\pi_2^*F))$ be of class $C^1$ such that $\gamma(0)$ is even.  Then $\mathcal{M}(s)$ extends continuously at $t=0$ if and only if $\mathcal{M}(\gamma(s))$ extends continuously at $t=0$. If $\gamma(0)= {\rm id}$, then $\mathcal{M}(s)$ and $\mathcal{M}(\gamma(s))$ extend by the same value.
\item[(iii)] Let $s:=(s_t)_{t>0}\in\Gamma(\pi_2^*\Lambda^q F^*\bigr|_{t>0})$ and  $\hat{\gamma}:=\Lambda^q\gamma^*\in\Gamma(\Aut(\pi_2^*F^*))$ a bundle automorphism arising from $\gamma\in \Gamma(\Aut(\pi_2^*F))$ which is assumed $C^1$ and {with $\gamma(0)$} even. Then $\Lambda^q\mathcal{M}^*(s)$ extends continuously at $0$ if and only if $\Lambda^q\mathcal{M}^*(\hat{\gamma}(s))$ extends continuously. They extend by the same value if $\gamma(0)= {\rm id}$.
\end{itemize}

\end{theorem}
\begin{proof} For (i) we write the block-decomposition relative $F=F_0\oplus F_1$ for $\gamma_t$:
\[\gamma_t:=\left(\begin{array}{cc} a_t&b_t\\
 c_t&d_t\end{array}\right) \Rightarrow\mathcal{M}_t\gamma_t\mathcal{M}_t^{-1}=\left(\begin{array}{cc} a_t & t^{-1}b_t\\
 t c_t& d_t \end{array}\right)\]
If $\gamma_0$ is even it follows that $b_0=0$ and, since $\gamma$ is $C^1$,  the continuous extendibility of $\mathcal{M}_t\gamma_t\mathcal{M}_t^{-1}$ follows. 

{For item (ii), write $\mathcal{M}(\gamma(s))=(\mathcal{M}\gamma\mathcal{M}^{-1})(\mathcal{M}(s))$. Item (i) guarantees that $\mathcal{M}\gamma\mathcal{M}^{-1}$ extends continuously at $t=0$, hence if $\mathcal{M}(s)$ extends continuously so does $\mathcal{M}(\gamma(s))$.} This proves one implication. Since $\gamma$ is a bundle automorphism one gets the full statement. Note that $\gamma(0)={\rm id}$ implies that $(\mathcal{M}\gamma(0)\mathcal{M}^{-1})=\id$ and therefore the extensions are equal.

For item (iii) we have that the extendibility of $\mathcal{M}\gamma(0)\mathcal{M}^{-1}$ at $t=0$ implies the extendibility of
\[
\Lambda^q\mathcal{M}^*_t\circ\Lambda^q\gamma_t^*\circ(\Lambda^q\mathcal{M}^*_{t})^{-1}=\Lambda^q \big(\mathcal{M}_t^*\circ\gamma_t^*\circ (\mathcal{M}^*_t)^{-1} \big).
\]
\end{proof}

In our case, we will take 
 \[F:=\pi^*E\oplus \pi^*TB\ra SE=:N.\]
It is implicit  in what follows that this bundle is endowed with a fixed Riemannian metric.

 Note that $F$ has two $\bZ_2$-gradings $F_0\oplus F_1$ such that $F_0\supset \pi^*TB$. The obvious one, which we call  the \emph{standard grading}:
 \[F_0:=\pi^*TB\quad \mbox{ and }\quad F_1:=\pi^*E.\] 
 The other one, which we will call the \emph{blow-up grading} has
\[F_0:=\pi^*TB\oplus \langle\nabla r\rangle\quad\mbox{and}\quad F_1:=VSE=\langle \nabla r\rangle^{\perp} \quad \mbox{ complement in } \pi^*E.\]


  Let us note a basic result which lies behind the first criterion, i.e. Proposition \ref{prwexB} of  blow-up extendibility.
\begin{lemma} There exists  natural isomorphisms of vector bundles over $[0,\infty)\times SE$:
\begin{equation}\label{eqmultop2}T\Bl_{\bf 0}(E)\simeq \bR\oplus \pi_2^*TSE\simeq  \pi_2^*(TE\bigr|_{SE})\end{equation}
where $\pi_2:\Bl_{{\bf 0}}(E)\ra SE$ is the projection onto the second coordinate.
\end{lemma}
\begin{proof} The first isomorphism is clear, because $\Bl_{{\bf 0}}(E)=[0,\infty)\times SE$. For the second, let us note the natural splitting $TE\bigr|_{SE}\simeq\bR\oplus {TSE}$ where the ``radial" component of the direct sum, induced by the non-vanishing (tautological section) $\nabla r$, and $TSE\subset TE$ are clearly direct summands of $TE$ . This splitting naturally carries over to a  splitting of $\pi_2^*(TE\bigr|_{SE})$.
\end{proof}

Given a form $\omega\in \Omega^q(E)$, the pull-back $\Bl^*\omega =\Lambda^q (d\Bl)^*(\omega_{\Bl})$ is the composition of two operations:
\begin{itemize}
\item[(i)]  the composition $\omega_{\Bl}:=\omega\circ \Bl$ which is a section of $\Bl^*\Lambda^qT^*E$
\item[(ii)] the action of the morphism $\Lambda^q (d\Bl)^*$ on $\omega_{\Bl}$.
Note that
 \[d\Bl:T\Bl_{{\bf 0}}(E)\ra \Bl^*TE.\]
It turns out that $ \Bl^*TE$ is also isomorphic to $\pi_2^*(TE\bigr|_{SE})$, however not ``canonically". With the ultimate purpose of interpreting $\Lambda^q (d\Bl)^*$ as a multiplication operator, we need one more definition.
\end{itemize}

Let $\pi^*E:=\langle \nabla r\rangle\oplus \langle \nabla r\rangle^{\perp}$ be a decomposition into the tautological line bundle  and its orthogonal complement along $E\setminus {\bf 0}$. Note that in fact $\langle \nabla r\rangle^{\perp}\bigr|_{SE}=V{\hat S}E$ is the vertical tangent bundle of the foliation by spherical bundles $S^rE$.

We will use $\nabla r$ for the normalized tautological section $r^{-1}\mathcal{R}$ of $\pi^*E$, but also for {$\pi_2^*\nabla r$},  which is a section of $\pi_2^*F\ra [0,\infty)\times SE$. We use freely the identication $\pi^*E=VE\subset TE$ so $\nabla r$ is also a section of $TE\bigr|_{E\setminus {\bf 0}}$.

\begin{definition}\label{fibtr} A fiber trivialization of $TE\ra E$ is an isomorphism of vector bundles over $E$ 
\[ 
\Psi: TE\ra \pi^*(TE\bigr|_{[0]})=\pi^*(E\oplus TB)=VE\oplus \pi^*TB
\]
such that 
\[
{\Psi =\id_{E\oplus TB} \qquad \text{on } \, TE\bigr|_{[0]}.}
\]
A radial trivialization of $TE\ra E$ is a fiber trivialization 
such that over $E\setminus {\bf 0}$ the following holds:
  \[\Psi(\nabla r)=\nabla r,\quad \mbox{and}\quad \Psi(T{\hat S}E)=\langle \nabla r\rangle^{\perp}\oplus \pi^*TB\]
\end{definition}
A fiber trivialization induces a decomposition $TE:=V_{\Psi}E\oplus H_{\Psi}E$ where
\[
V_{\Psi}E:=\Psi^{-1}(\pi^*E)\qquad\mbox{and}\qquad H_{\Psi}E:=\Psi^{-1}(\pi^*TB)
\]
Set
\[
V_{\Psi}\hat{S}E:=\Psi^{-1}(\langle \nabla r\rangle^{\perp}), \qquad T_{\Psi}\hat{S}E:= V_{\Psi}\hat{S}E\oplus H_{\Psi}E
\]
Note that for a radial trivialization $T_{\Psi}\hat{S}E=T\hat{S}E$, but this may not be the case in general.

There are two main examples of \emph{radial trivializations}.
\begin{example}[\textbf{Ehresmann connection}] An Ehresmann connection $s:TE\ra VE=\pi^*E$ induces an isomorphism
 \[ \Psi:TE\ra VE\oplus \pi^*TB,\qquad \Psi_{(p,v)}(w):=(p,s(w),d_{(p,v)}\pi(w)),\quad (p,v)\in E, ~w\in T_{(p,v)}E\]
We have $V_{\Psi}E=VE$ and $H_{\Psi}E=HE$. By the definition of an Ehresmann connection  $\Psi$ is the identity on $VE$, so
\[
\Psi(\nabla r)=\nabla r.
\]

Only for a \emph{metric} Ehresmann connection,  i.e. $HE\bigr|_{SE}\subset TSE$, we have $T_{\Psi}SE=TSE$.  
\end{example}

\begin{example}[\textbf{Parallel transport}] \label{Partrc} Assume that the manifold $E$ is endowed with a Riemannian metric  $g$ such that the following hold:
\begin{itemize}
\item[(i)] {the restriction of $g$ to $TE\bigr|_{[0]} = E\oplus TB$ is the product metric $\metric + g_B$: 
\[ 
g(p,0) = \metric_p + g_B(p)
\qquad \forall p\in B
\]
where $\metric$ is the bundle metric on $E\ra B$ and $g_B$ the Riemannian metric on $B$}.
\item[(ii)] {along $E\setminus {\bf 0}$ the vector field $\nabla r$ coincides} with the gradient of the function $r:E\setminus {\bf 0}\ra \bR$ for the metric $g$, and
\[|\nabla r|_{g}=1\]
\end{itemize}  
Condition (ii) implies by \cite[Section 5.5]{petersen} that $\nabla r$ is a geodesic vector field, meaning that  the integral curves (i.e. the rays $t\ra (p,tv)$ with $|v|_{E}=1$) of $\nabla r$ are geodesics for $g$.

 The isomorphism $\Psi$ is induced by parallel transport through the Levi-Civita connection of $g$ is a radial trivialization for the following reasons. {Let $(p,v)\in TE$, $v \neq 0$ and let 
\[
\gamma:[0,|v|]\ra E\qquad  \gamma(t)=\left(p,t v/|v|\right).
\]
Then parallel transport induces an isomorphism (taking the values at $|v|$  and $0$) 
\[ 
\mathcal{P}_{\gamma}^{-1}:T_{(p,v)}E\ra T_{(p,0)}E=(\pi^*E\oplus \pi^* TB)_{(p,v)} 
\]
Since $\gamma$ is an integral curve for $\nabla r$  it is also a $g$-geodesic and we have
 \[
 \mathcal{P}_{\gamma}^{-1}((\nabla r)_{(p,v)})=\mathcal{P}_{\gamma}^{-1}\left(p,v, \frac{v}{|v|}\right)=\mathcal{P}_{\gamma}^{-1}(\gamma'(|v|))=\gamma'(0)= \frac{v}{|v|} =(\nabla r)_{(p,v)}\in (\pi^*E)_{(p,v)}
 \]
}
Based on the last computation, the isometry $\mathcal{P}_{\gamma}$ takes the linear space $\langle v\rangle^{\perp}\oplus T_pB$ over $(p,0)$ (which is the $g$-orthogonal complement to $\langle v\rangle$ by using (i) above) to the  $g$-orthogonal complement of $(\nabla r)_{(p,v)}=(p,v,v/|v|)$. The latter is {$T_{(p,v)}\hat SE$, since $\nabla r$ is the $g$-gradient of $r$ and the spheres $S^tE$ are level sets of $r$}. 

\vspace{0.3cm}

We remark that the metric on a neighborhood of the zero-section  on the normal bundle $\nu L$ (of a compact submanifold $L\subset M$) induced by pulling back the metric on $M$ via the exponential map satisfies the properties (i) and (ii) above. Hence the parallel transport gives a radial trivialization of $T\nu L$, albeit only in a neighborhood of the zero-section , which is enough for extendibility purposes.
\end{example}

\vspace{0.3cm}

A fiber trivialization gives more than just an isomorphism $\Bl^*TE\simeq \pi_2^*TE\bigr|_{SE}$ as follows.

\begin{prop}\label{multop1} A  fiber trivialization $\Psi$ of $TE$ induces the following bundle isomorphisms: 
\begin{equation} \label{eqmul1} TE\bigr|_{SE}\simeq \pi^*E\oplus \pi^*TB=F,\qquad\qquad \Bl^*TE\simeq \pi_2^*F
\end{equation}
 Suppose $\Psi$ is induced by a metric Ehresmann connection. Using the identifications (\ref{eqmultop2}), (\ref{eqmul1})  the operator
 \[\Lambda^q (d\Bl)^*: \pi_2^*\Lambda^qF^* \ra \pi_2^*\Lambda^qF^*  \]
 becomes a multiplication operator on $\pi_2^*\Lambda^q F^*\ra [0,\infty)\times SE$ for the blow-up $\bZ_2$-grading.
\end{prop}
\begin{proof}  The first isomorphism is the restriction of $\Psi$ to $SE$. The second isomorphism is a consequence of the equality
 \[
 \pi \circ \Bl=\pi\bigr|_{SE}\circ \pi_2\qquad: [0,\infty)\times SE\ra B.
 \]
Indeed, one pulls-back $E\oplus TB$ on both sides and uses $TE\simeq \pi^*E\oplus \pi^*TB$. 
 
For the second statement, we claim that 
\[
d \Bl \ : \ T\Bl_{[0]}(E) \simeq \pi_2^* F \to \Bl^* TE \simeq \pi_2^* F
\]
is a multiplication operator for \emph{the blow-up grading}. 

For a metric Ehresmann connection, $d\Bl$ takes $H_{\Psi}E\bigr|_{SE}=HE\bigr|_{SE}\subset TSE$ to $\Bl^*H_{\Psi}E=\Bl^*HE\subset \Bl^*TE$. In fact, $d_{(t,p,v)}\Bl_{[0]}\bigr|_{H_{(p,v)}E}$ is the restriction of the differential of the rescaling mapping which by definition takes $H_{(p,v)E}$ to $H_{(p,tv)}$. After the identifications  \[ H_{(p,v)}E\simeq T_pB\simeq H_{(p,tv)}E\] via $d\pi$ the restriction $d_{(t,p,v)}\Bl\bigr|_{HE}$ becomes the identity morphism on this space.

Note that $(t,p,v)\ra d\Bl_{(t,p,v)}(\partial_t)$ is in fact {$\pi_2^* \nabla r$} where $\nabla r$ is here the tautological section of $\pi^*E\ra SE$. With the identifications of (\ref{eqmultop2}), (\ref{eqmul1}) the vector field $\partial_t$ corresponds to $\nabla r$. Hence, the restriction of $d_{(t,p,v)}\Bl_{[0]}$ to $\langle \nabla r \rangle \oplus HE$ is the identity. On the other hand, if $w\in V_{\Psi}\hat{S}E\bigr|_{SE}=VSE$ we have 
\[(t,p,v)\ra d\Bl_{(t,p,v)}(w)=tw\]
and this proves the claim.
\end{proof}

For a form $\omega\in \Omega^q(E\setminus {\bf 0})$ and a fiber trivialization $\Psi$ of $TE$, define the following $1$-parameter family of sections of $\pi_2^*(\Lambda^qF^*)\ra (0,\infty)\times SE$:
\[ \omega^{F\Psi}=((\Lambda^q\Psi)^{*}\circ \Bl)(\omega \circ \Bl).\footnote{The reason for using $F\Psi$ in the notation is to separate this family from families of forms on $SE$ that will appear later.}\]
Alternatively,  $\omega^{F\Psi}=(\omega_t^{F\Psi})_{t>0}\in \Gamma(\pi_2^*(\Lambda^q F^*))$ is defined via
\[\omega_{(t,p,v)}^{F\Psi}:=(\Lambda^q\Psi_{(p,tv)})^{*}(\omega_{(p,tv)})\]

  The section $\omega^{F\Psi}$ is essentially  $\omega\circ \Bl$ after the identification via $\Psi$ of $\Bl^*TE$ with $\pi_2^*F$. Fix a Riemannian vector bundle metric on $F\ra SE$.

\begin{theorem}\label{indth}  Let $\omega\in \Omega^q(E\setminus {\bf 0})$ be a continuous form. Let $\mathcal{M}_t$ be the multiplication operator associated to the blow-up $\bZ_2$-grading of $F\ra SE$. Let $\Psi$ be a fiber trivialization.
\begin{itemize}
\item[(i)] If $\Psi$ is induced by a metric Ehresmann connection then the blow-up extendibility of $\omega$ is equivalent to the uniform continuity {for $t \in (0,\epsilon]$} of the family $t\ra \Lambda^q\mathcal{M}_t^* (\omega_t^{F\Psi})$, a section of $\pi_2^*(\Lambda^qF^*)\bigr|_{{t\neq0}}\ra (0,\infty)\times SE$.
\item[(ii)]  The uniform continuity of $t\ra \Lambda^q\mathcal{M}_t^* (\omega_t^{F\Psi})$ is independent of the fiber trivialization $\Psi$.
\end{itemize}
\end{theorem}
\begin{proof} We start with a general observation. If $F'\ra [0,\infty)\times SE$ is a vector bundle with a fixed isomorphism $\Theta: F'\simeq \pi_2^*\tilde{F}$ to a Riemannian vector bundle $\tilde{F}\ra SE$ then the continuous extension of a section $s_1\in \Gamma(F'\bigr|_{t>0})$ is equivalent with the continuous extension of $\Theta(s_1)\in \Gamma(\pi_2^*\tilde{F})$ which by Proposition \ref{stprop} is equivalent with the uniform continuity of the $1$-parameter family of sections 
\[t\ra \Theta(s_1)_t\in \Gamma(\tilde{F})\]

 For (i) we now note that the pull-back $\Bl^*\omega$ is $\Lambda^q(d\Bl)^*$ applied to $\omega\circ \Bl$. 

 By Proposition \ref{multop1}, after the appropriate bundle isomorphisms, the later coincides with a multiplication operator for the blow-up $\bZ_2$-grading of $F$. Therefore the uniform continuity of $t\ra \Lambda^q\mathcal{M}_t^* (\omega_t^{F\Psi})$ characterizes the blow-up extendibility of the form $\omega$.

For (ii), let $\Psi_1$ and $\Psi_2$ be two fiber trivializations. For all $t\geq 0$, $(p,v)\in SE$ let $\gamma\in \Gamma(\Aut(\pi_2^*F))$ be defined as follows.
\[\gamma_{(t,p,v)}:=(\Psi_2)_{(p,tv)}\circ (\Psi_1)_{(p,tv)}^{-1}\in \Aut(E_p\oplus T_pB).\]
Note that $\gamma_t$ is $C^1$ up to $t=0$, and that $\gamma_0=\id$ by Definition \ref{fibtr}. In order to use Theorem \ref{kth} for $\gamma$ and the blow-up $\bZ_2$-grading of $F$ it would be enough  that $\gamma_{0}=\id_E\oplus A$ where $A$ is an automorphism of $TB$. But $\gamma_0={\rm id}$ guarantees not only the independece of fiber trivialization but also that the limits at $t=0$ of the extensions are equal. 
\end{proof}

 We will see  that  Theorem \ref{indth} gives rise to a more practical, sufficient criterion of extendibility when the trivializations are \emph{radial} and the forms involved are functions of $r$ times smooth (everywhere) forms.

Theorem \ref{indth} part (ii) offers a way to check whether a form $\omega$ is blow-up extendible after we have fixed a fiber trivialization $\Psi$. Let us make it more explicit. 

We write the blow-up grading:
\[F=F_0\oplus F_1,\qquad F_0=\langle \nabla r \rangle\oplus \pi^*TB,\qquad F_1:=\langle \nabla r\rangle^{\perp}\]
 Relative to this decomposition we have correspondingly
\[
\omega^{F\Psi}=\sum_{i=0}^k\omega_i^{F\Psi}
\]
where $\omega^{F\Psi}_i\in \Gamma(\Lambda^i F_1^*\otimes \Lambda^{k-i}F_0^*)$. 
The multiplication operator acts as follows:
\[ \Lambda^q\mathcal{M}^* (\omega^{F\Psi})=\sum_{i=0}^kt^i\omega^{F\Psi}_i\]
On the other hand, due to a secondary decomposition  $F_0=L\oplus F_0'$ where $L= \langle\nabla r \rangle$ is a trivial line bundle  and $F_0'=\pi^*TB$, we correspondingly have for every $i=0,\ldots, k$
\[\omega^{F\Psi}_i=\omega^{F\Psi}_{i,0}+dr\wedge \omega^{F\Psi}_{i,1}\]
with $\omega^{F\Psi}_{i,0}\in \Gamma(\Lambda^i F_1^*\otimes \Lambda^{k-i}F_0'^*)$ and $\omega^{F\Psi}_{i,1}\in \Gamma(\Lambda^iF_1^*\otimes\Lambda^{k-i-1}F_0'^* )$.
 We spell out the criterion of extendibility which is immediate from Theorem \ref{indth} part (ii)
\begin{corollary}\label{Psi1} Let $\Psi$ be any fiber trivialization.  A form $\omega\in \Omega^k(E\setminus {\bf 0})$ is blow-up extendible if and only if the following $1$-parameter families of sections  $t\ra t^i(\omega_{i,j}^{F\Psi})_t$ for $j=0,1$ and $i=0,\ldots k-j$, with $ (\omega_{i,j}^{F\Psi})_t\in \Gamma(\Lambda^iF_1^*\otimes \Lambda^{k-i-j}F_0'^*)$ are uniformly continuous for $t\in (0,\epsilon]$.
\end{corollary}


It turns out that when $\Psi$ is a radial trivialization this criterion takes on a formulation that is closer in spirit with Theorem \ref{mainthE}. The fastest way to obtain such a formulation is to use the isomorphism $\Psi^{-1}$ and move from the $F$ to $TE\bigr|_{SE}$ (again a bundle where a metric has been fixed, but no relation with the metric on $F$ is required).  There are a few ways one can get families of forms on $SE$ from $\omega\in \Omega^k(E\setminus {\bf 0})$ with the same indexing as in the corollary. For radial trivialization they all coincide.
\begin{itemize}
\item[(i)] One way is to simply use the exterior powers of the isomorphism $\Psi:TE\bigr|_{SE}\ra F$ and simply apply them to  $\omega^{F\Psi}_{i,j}$ in order to get families of sections $\omega^{\Psi}_{i,j}\in \Gamma(\Lambda^iV_{\Psi}^*SE\otimes \Lambda^{k-i-j}H_{\Psi}^*E\bigr|_{SE})$. Since, for a radial trivialization we have
\[V_{\Psi}SE\oplus H_{\Psi}E\bigr|_{SE}=T_{\Psi}SE=TSE\]
these $\omega^{\Psi}_{i,j}$ are truly families of forms on $SE$.
\item[(ii)] Use the decomposition on $E\setminus {\bf 0}$
\[TE=V_{\Psi}\hat{S}E\oplus (\langle\nabla r\rangle \oplus H_{\Psi}E)\]
in order to write
 \[\omega=\sum_{i=0}^k\hat{\omega}_i^{\Psi}\] and then decompose further for each $i$:
\[\hat{\omega}_i^{\Psi}=\hat{\omega}_{i,0}^{\Psi}+dr\wedge \hat{\omega}_{i,1}^{\Psi}\]
where $\hat\omega_{i,j}^{\Psi}\in \Gamma(\Lambda^iV^*_{\Psi}\hat{S}E\otimes \Lambda^{k-i-j}H^*_{\Psi}E)$.  From these $\hat{\omega}_{i,j}^{\Psi}$ defined on $E\setminus {\bf 0}$ one obtains $1$-parameter families of sections of the restrictions
\[\left(\Lambda^iV^*_{\Psi}\hat{S}E\otimes \Lambda^{k-i-j}H^*_{\Psi}E\right)\biggr|_{SE}=\Lambda^iV_{\Psi}^*SE\otimes \Lambda^{k-i-j}H_{\Psi}^*E\bigr|_{SE}\]
if we use the identification of $T_{(p,tv)}\hat{S}E$ with $T_{(p,v)}SE$ provided by $\Psi$ (and not by the Lemma \ref{natu}).  More precisely we have isomorphisms for all $(p,v)\in SE$, $t>0$:
\[ { T_{(p,tv)}\hat{S}E{\overset {\Psi_{(p,tv)}}{\longrightarrow}}  \langle\nabla r\rangle^{\perp}\oplus T_pB} {\overset {\Psi_{(p,v)}}{\longleftarrow}}  T_{(p,v)}SE\]
and they are compatible with the decompositions
\[ T_{(p,tv)}\hat{S}E=(V_{\Psi}\hat{S}E)_{(p,tv)}\oplus H_{\Psi}E_{(p,tv)}\qquad T_{(p,v)}{S}E=(V_{\Psi}SE)_{(p,v)}\oplus H_{\Psi}E_{(p,v)} \]
\item[(iii)] One can first decompose on $E\setminus {\bf 0}$
\[TE=\langle\nabla r \rangle\oplus T\hat{S}E\]
and correspondingly
\[\omega=\tilde{\omega}_0+dr\wedge \tilde{\omega}_1,\qquad \tilde{\omega}_j\in \Gamma(\Lambda^{k-j}T^*\hat{S}E).\]
Now take the decomposition
\[ T\hat{S}E=V_{\Psi}\hat{S}E\oplus H_{\Psi}E\]
and decompose further 
\[\tilde{\omega}_j=\sum_{i=0}^{k-j}\tilde{\omega}_{i,j}^{\Psi},\qquad \tilde{\omega}_{i,j}^{\Psi}\in \Gamma(\Lambda^iV_{\Psi}^*\hat{S}E\oplus \Lambda^{k-j-i}H_{\Psi}^*E).\]
Then pass to families of sections of $\Lambda^iV_{\Psi}^*{S}E\oplus \Lambda^{k-j-i}(H_{\Psi}^*E\bigr|_{SE})\ra SE$ as  in item (ii).
\item[(iv)] Start as in item (iii) but before taking the second decomposition use first $\tilde{\omega}_j$ to obtain families of sections of $\Lambda^{k-j}TSE$ again through $\Psi$ and only then use the decomposition 
\[TSE=V_{\Psi}SE\oplus H_{\Psi}SE\bigr|_{SE}\]
to obtain families $t\ra (\check{\omega}_{i,j}^{\Psi})_t$.
\end{itemize}

We then have the following
\begin{prop}\label{Psi2} For a radial trivialization $\Psi$,  the families of sections 
\[t\ra(\omega_{i,j}^{\Psi})_t,  \qquad t\ra (\hat{\omega}_{i,j}^{\Psi})_t,\qquad t\ra (\tilde{\omega}_{i,j}^{\Psi})_t,\qquad t\ra (\check{\omega}_{i,j}^{\Psi})_t\]
of the bundle $\Lambda^iV_{\Psi}^*SE\otimes \Lambda^{k-i-j}H_{\Psi}^*E\bigr|_{SE}\ra SE$    coincide for each fixed pair of indices $(i,j)$.
\end{prop}
\begin{proof} One unravels the definition of $\omega^{F\Psi}$ and sees that the isomorphisms involved, all induced by $\Psi$, are compatible with all the decompositions described when $\Psi$ is a radial trivialization. 
\end{proof}
\begin{remark} When $\Psi$ is induced by an Ehresmann connection one can use the diffeomorphisms of the flow of $\nabla r$ as in Lemma \ref{natu} in order to obtain families of forms on $SE$. The families thus obtained coincide with the others for this particular type of radial trivialization, but not in general. 
\end{remark}

Denote by $\omega_{i,j}^{\Psi}$ any of the families from Proposition \ref{Psi2}. They are truly families of forms on $SE$ of degree $k-j$. Our main characterization of blow-up differentiability is  an immediate consequence of Corollary \ref{Psi1} and Proposition \ref{Psi2}.
\begin{theorem}\label{thmainA} Fix a radial trivialization $\Psi$. A form $\omega\in \Omega^k(E\setminus {\bf 0})$ is blow-up extendible  if and only if for each $j=0,1$ and $i=0,\ldots k-j$   the family of forms on $SE$:
\[t\ra  t^i(\omega_{i,j}^{\Psi})_t,\]
is uniformly continuous on an  interval $(0,\epsilon]$.
\end{theorem}

\vspace{0.5cm}

From Theorem \ref{thmainA} we get a sufficient condition for blow-up extendibility.

\begin{prop}\label{mainpropA} Let $\Psi$ be a radial trivialization and let $\omega\in\Omega^k(E\setminus {\bf 0})$ be a form which   can be written either  as
\begin{itemize}
\item[($\mathrm{O}$)]$\omega=f(r)\omega_i$, statement indexed by $j=0$, or as
\item[($\mathrm{I}$)]$\omega=f(r)\iota_{\nabla r}\omega_{i+1}$, statement indexed by $j=1$
\end{itemize}
where  $\omega_{i+j}\in\Gamma(E,\Lambda^{i+j}V^*_{\Psi}E\otimes \Lambda^{k-i}H^*_{\Psi}E)$ ($j\in\{0,1\}$) is smooth everywhere  and $f:(0,\infty)\ra \bR$ is a continuous function.

If $i+j=0$ occurs, assume that $f(r)\omega_0$ extends continuously along ${\bf 0}$.

If $i+j>0$ occurs, assume that $r^{i+j-1}f(r)\omega_{i+j}$ extends continuously along ${\bf 0}$. 

Then $\omega$ is blow-up extendible.
\end{prop}
\begin{proof} We have the decomposition
\[
\omega=\hat{\omega}_0+dr\wedge \hat{\omega}_1,\qquad\mbox{with}\quad \hat{\omega}_1=\iota_{\nabla r}\omega
\]
where, for $l\in\{0,1\}$, 
\[
\hat{\omega}_l\in\Gamma(E\setminus {\bf 0},\Lambda^{i-l}V^*_{\Psi}\hat{S}E\otimes \Lambda^{k-i}H^*_{\Psi}E)). 
\]
In the case ($\mathrm{I}$), one obviously has $\hat{\omega}_1=0$.

 Both $\hat{\omega}_l$ give rise to families of forms on $SE$ and according to Theorem \ref{thmainA} we need to understand when are the families $r\ra r^{i-l}(\hat{\omega}_l)_{r}$  uniformly continuous. In the case ($\mathrm{I}$), just $l=0$ is necessary. But clearly the extension along ${\bf 0}$ of $r^{i-1}f(r)\omega_i$ in case ($\mathrm{O}$) or of $r^{i}f(r)\omega_{i+1}$ in case ($\mathrm{I}$) is a stronger condition then the uniform continuity of the required families.
\end{proof}

When $B$ is a point it is standard when a function/section $\frac{g}{r^{a}}$ is continuously extendible at $\{0\}$ for $g$  smooth everywhere. One needs the vanishing of $g^{(l)}(0)=0$ for $0\leq l\leq a-1$. We describe analogous  conditions that imply the hypothesis of Proposition \ref{mainpropA}.

{If $G \to B$ is a vector bundle on $B$ for sections  $s$ of a pull-back bundle $\pi^*G\ra E$ it makes sense to differentiate $s$ in the vertical directions. One takes the differential of the restriction $s\bigr|_{E_p}:E_p\ra G_p$ in the $v\in E_p$ direction.  We get a morphism
\[Ds:VE\ra \pi^*G\]


Since $VE=\pi^*E$, the process can be repeated in order to get the higher differentials which are symmetric bundle morphisms
\[
VE\times\ldots\times VE\ra \pi^*G
\]
We denote by $D^{(j)}s$  the homogeneization of this morphism, i.e. the composition with the main diagonal morphism
\[VE\ra VE\times \ldots\times VE\]
Then $D^{(j)}_0s:E\ra G$ is the restriction of $D^{(j)}s$ to points on the zero-section. The Taylor Theorem can be easily adapted to compact families of mappings (i.e. $B$ is compact) and we have the following approximation at $0$ for a section $s:E\ra \pi^*G$:
\begin{equation}\label{TT}s(p,v)=\sum_{j=0}^k\frac{1}{j!}D^{(j)}_0s(p,v)+o(|v|^{k+1})\end{equation}

 As an immediate consequence of Proposition \ref{mainpropA} we have the next statement, whose proof we leave to the reader as we only need "the border cases" (no vertical derivatives) which are straightforward.

\begin{corollary} \label{maincorA} Suppose that one of the conditions ($\mathrm{O}$) or ($\mathrm{I}$) of Propostion \ref{mainpropA} is fulfilled.  Assume $f(r)=(r^{-a})g(r)$ for some natural number $a>0$ and continuous $g$ all the way to $r=0$.
 Let $\tilde{\omega}_{i+j}:=(\Lambda^k\Psi)^*(\omega_{i+j})$ be the section of $\Lambda^{i+j}\pi^*E^*\otimes \Lambda^{k-i}\pi^*T^*B$ which corresponds to $\omega_{i+j}$ after the fiber trivialization. 

If any of the following situations occurs, then $\omega$ is blow-up extendible:

\begin{itemize}
\item[(A)] $i+j=0$  and $D^{(l)}_0\tilde{\omega}_0=0$ for $0\leq l\leq a-1$ or
\item[(B)]  $0<i+j\leq a$  and $D^{(l)}_0\tilde{\omega}_{i+j}=0$ for $0\leq l\leq a-i-j$ or
\item[(C)] $a<i+j$.
\end{itemize}
In particular, (A) occurs if $(i,j,a)=(0,0,1)$ and $\omega_0\bigr|_{{\bf 0}}\equiv 0$  and (B) occurs if $a=i+j$ and $\omega_{i+j}\bigr|_{{\bf 0}}\equiv 0$.

\end{corollary}

The "border cases" consist of  (C) plus the particular cases described in the last statement of Corollary \ref{maincorA}.

\end{document}